\documentclass[a4paper,12pt]{article} 
\usepackage[T2A]{fontenc}			
\usepackage[utf8]{inputenc}			
\usepackage[english]{babel}	
\usepackage{textcomp}
\usepackage{amsmath,amsfonts,amssymb,amsthm,mathtools} 
\usepackage{icomma} 
\usepackage[all]{xy}

\newcommand{\LL}{{\mathcal{L}}}
\newcommand{\A}{{\mathcal{A}}}
\newcommand{\J}{{\mathcal{J}}}

\renewcommand{\leq}{\leqslant}
\renewcommand{\geq}{\geqslant}

\usepackage{euscript}	 
\usepackage{mathrsfs} 

\newtheorem{theorem}{Theorem}[section]
\newtheorem{lemma}[theorem]{Lemma}

\DeclareMathOperator{\Real}{\mathop{Re}}

\usepackage{graphicx}  
\graphicspath{{images/}{images2/}}  
\usepackage{wrapfig} 

\usepackage{array,tabularx,tabulary,booktabs} 
\usepackage{longtable}  
\usepackage{multirow} 

\numberwithin{equation}{section}

\author{Vitalii V. Iudelevich }
\title{On the sum-of-squares function \footnote{This is a preprint of an article accepted for publication in the journal Izvestiya RAN. Distribution rights are held by Steklov Mathematical Institute of Russian Academy of Sciences.}}
\date{\today}

\begin{document} 
	\maketitle
\begin{abstract}
	In this paper, we derive the following asymptotic formula
	$$ \mathop{{\sum}'}_{n\leqslant x}\dfrac{r(n)}{r(n+1)} = {x}{(\ln x)^{-3/4}}(c+o(1)),\ \ x \to +\infty,$$
	where $r(n)$ is the number of representations of $n$ as a sum of two squares, $c$ is a positive constant, and the prime indicates summation over those $n$ for which $r(n+1)\neq 0$.
	
\end{abstract}
	\section{Introduction}
	One of the key problems in analytic number theory is to obtain asymptotic formulas for sums of the form
	\begin{equation}\label{Cfg}
		\mathscr{C}_{f, g}(x) = \sum_{n\leqslant x}f(n)g(n+1),
	\end{equation}
	where $x \to +\infty$, and $f$ and $g$ are certain arithmetic functions. A typical example of such a problem arises when one of the functions $f$ or $g$ coincides with the divisor function
	$$\tau(n) = \sum_{d | n} 1.$$
	
	One of the earliest results for sums of this kind is due to Ingham \cite{Ingham27}. He proved that as $x\to +\infty$, the following equality holds:
	\begin{equation}\label{Ingham}
		\sum_{n\leqslant x}\tau(n)\tau(n+1) =  \dfrac{6}{\pi^2}x(\ln x)^2+ O\left(x\ln x\right)\!.
	\end{equation}
	Ingham's arguments is elementary; he expresses the left-hand side of the equality \eqref{Ingham} in terms of the number of solutions to the equation
	$$ab - cd = 1$$
	subject to the conditions $ab, cd\leqslant x$, and then asymptotically computes it using formulas for solutions to linear Diophantine equations in two variables. Subsequently, the equality \eqref{Ingham} has been refined and generalized by many authors (see \cite{Ester} -- \cite{Topac}).
	
	A more challenging problem is to find the asymptotic behavior for the sum \eqref{Cfg} when $f(n) = \tau(n)$ and $g(n)$ is an arithmetic function distinct from $f$. In 2004, A. A. Karatsuba posed the problem of determining the asymptotics for the sum
	$$S(x) = \sum_{ n \leqslant x}\dfrac{\tau(n)}{\tau(n+1)}.$$ 
	In 2008, F. Luca and I. Shparlinski \cite{Luca} established the correct order of magnitude of $S(x)$, showing that
	$$S(x)\asymp x\sqrt{\ln x}.$$
	Later, in 2010, M. A. Korolev \cite{Korol2010} obtained the asymptotic formula
	$$S(x) = Kx\sqrt{\ln x}+O(x\ln\ln x),$$
	where
	$$K = \dfrac{1}{\sqrt{\pi}}\prod_{p}\left(\dfrac{1}{\sqrt{p(p-1)}} + \sqrt{1-\dfrac{1}{p}}(p-1)\ln\dfrac{p}{p-1}\right) \approx 0{.}75782.$$
	A key ingredient that allows for the derivation of the asymptotic formula for $S(x)$ is the following analogue of Bombieri–Vinogradov's theorem (see \cite[Lemma $13$]{Korol2010}).
	\begin{theorem}\label{MainKorolev}
		Let $d\geqslant 1$ be a fixed integer. Then for any fixed $B>0$, there exists $A = A(B) > 0$ such that the inequality
		$$R = \sum_{q\leqslant Q}\dfrac{1}{\varphi(q)}\sum\limits_{\substack{\chi\,\textup{mod\,} q \\ \chi\neq\chi_0}}\left|\sum\limits_{\substack{n\leqslant N \\ (n, d) = 1}}\dfrac{\chi(n)}{\tau(n)}\right|\ll x(\ln x)^{-B}$$
		holds for all $Q, N$ with the conditions $Q\leqslant \sqrt{x}(\ln x)^{-A}, N\leqslant x$, where the constant in the symbol $\ll$ is ineffective and depends on $B$ and $d$.
	\end{theorem}
	In this paper, we investigate the sum 
	$$Q(x) = {\sum_{n\leqslant x}}'\dfrac{r(n)}{r(n+1)},$$
	where $r(n)$ is the sum-of-squares function,
	$$r(n) = \#\left\{(a, b) \in \mathbb{Z}^2: a^2 + b^2 = n\right\}\!,$$
	and the prime indicates summation over those $n$ for which $r(n+1)\neq 0$.
	The function $r(n)$ is related to the divisor function due to the equality 
	\begin{equation}\label{rsum}
		r(n) = 4\sum_{d|n}\chi_4(d),
	\end{equation} 
	where $\chi_4$ is the unique non-principal character modulo $4$:
	\begin{equation}\label{chi4}
		\chi_4(n)=
		\begin{cases}
			(-1)^{\frac{n-1}{2}}, & \text{if $n$ is odd,}\\
			0, & \text{otherwise.}
		\end{cases}
	\end{equation}
	The main result of the paper is the following
	\begin{theorem}\label{MainTheorem} For any fixed $\varepsilon>0$, as $x \to +\infty$, we have
		$$Q(x) = \dfrac{ c_1 x}{(\ln x)^{\frac{3}{4}}} + O_{\varepsilon}\left(\dfrac{x}{(\ln x)^{1-\varepsilon}}\right)\!,$$
		where
		$$c_1 = \dfrac{\pi^{\frac{3}{4}}}{2\Gamma(\frac{1}{4})}\prod_{p\equiv 1\!\!\!\!\!\pmod{4}}\sqrt[4]{\dfrac{p-1}{p+1}}\left(\dfrac{1}{p-1} + (p-1)\ln\dfrac{p}{p-1}\right)\approx 0{.}339385.$$
	\end{theorem}
	We briefly outline the proof scheme. In $\S 3$, using the equality \eqref{rsum}, we express the sum $Q(x)$ as
	$$Q(x) = 4\sum_{d\leqslant x}\chi_4(d){\sum\limits_{\substack{n \leqslant x \\ n\equiv 1\!\!\!\!\!\pmod{d}}}}'\dfrac{1}{r(n)} + O(1).$$
	Next, we denote by $Q_1, Q_2$, and $Q_3$ the contributions from those $d$ for which 
	$$d\leqslant \sqrt{x}(\ln x)^{-A},\ \ \sqrt{x}(\ln x)^{-A}<d \leqslant \sqrt{x}(\ln x)^{A},\ \ d>\sqrt{x}(\ln x)^{A}$$ 
	respectively. Here, $A>0$ is a sufficiently large constant, the exact value of which will be chosen later. Assuming that
	$${\sum\limits_{\substack{n\leqslant x \\ n \equiv 1\!\!\!\!\!\pmod{d}}}}' \dfrac{1}{r(n)} \approx \dfrac{1}{\varphi(d)}{\sum\limits_{\substack{n \leqslant x \\ (n, d) = 1}}}'\dfrac{1}{r(n)},$$
	and making the corresponding substitution,
	we obtain the expression
	$$Q^{MT} = 4\sum_{d\leqslant x}\dfrac{\chi_4(d)}{\varphi(d)}{\sum\limits_{\substack{n \leqslant x \\ (n, d) = 1}}}'\dfrac{1}{r(n)},$$
	which is later transformed into the main term of the asymptotics in $\S 13$ using the method of contour integration. The remainder terms obtained from sums $Q_1$ and $Q_3$ during this substitution are also estimated as in the work \cite{Korol2010}. To estimate them in $\S 4$, we establish an analogue of Theorem \ref{MainKorolev}.
	
	The greatest difficulty lies in estimating the remainder from the sum $Q_2$, specifically estimating the expression
	\begin{multline}\label{QQErr2}
		Q^{Err_2} = 4\sum_{\frac{\sqrt{x}}{(\ln x)^A} < d \leqslant \sqrt{x}(\ln x)^{A}}{\chi_4(d)}\left( {\sum\limits_{\substack{\ell \leqslant x \\ \ell\equiv 1\,(\text{mod}\, d)}}}'\dfrac{1}{r(\ell)} - \dfrac{1}{\varphi(d)}{\sum\limits_{\substack{\ell \leqslant x \\ (\ell, d) = 1}}}'\dfrac{1}{r(\ell)} \right)\!,
	\end{multline}
	to which the remaining paragraphs of the article are dedicated.
	In the work \cite{Korol2010}, the corresponding remainder was estimated trivially, which was sufficient to obtain asymptotics. In our case, a trivial estimate yields only the rough inequality
	$$Q^{Err_2}\ll \dfrac{x\ln\ln x}{(\ln x)^{\frac{3}{4}}}.$$
	A satisfactory estimate for the remainder $Q^{Err_2}$ is obtained by following the dispersion method technique from the work of E. Fouvry and M. Radziwi{\l\l} \cite{Fouvry Radziwill}.
	
	Roughly speaking, and not in the greatest generality, the dispersion method consists of estimating the square of the difference between certain double sums and their approximate values. An example of such a difference can be expressed as
	$$\mathcal{D}(N, M; d) = \sum\limits_{\substack{N<n \leqslant  2N\\ M<m\leqslant 2M \\ nm\equiv 1\!\!\!\!\!\pmod{d}}} a(m)b(n) -\dfrac{1}{\varphi(d)} \sum\limits_{\substack{N<n \leqslant  2N\\ M<m\leqslant 2M \\ (nm,d) = 1}} a(m)b(n),$$
	where $N, M$ are some positive numbers, and $a$ and $b$ are some generally complex-valued functions. Such a sum arises in $\S 6$, where we show that
	\begin{equation*}
		Q^{Err_2} \approx \dfrac{1}{4\pi}\sum_{k<\LL^{1-\varepsilon}}\dfrac{1}{k}{\sum\limits_{\substack{D, N, M }}}\int_{-\infty}^{+\infty}F(it)\tilde{U}(D, N, M, t)x^{it} dt,
	\end{equation*}
	\begin{equation*}
		\tilde{U}(D, N, M, t) = \sum_{D<d\leqslant  2D}\chi_4(d)\mathcal{D}(N, M).
	\end{equation*}
	Here, the double sum in $\tilde{U}$ arises from representing $\ell$ as $\ell = nm$, where $n$ is a prime from some fixed interval. The summation over $k$ arises from accounting for the number of such representations. The presence of an integral and a smooth function $F$ comes from replacing the condition $nm \leqslant x$ with a smooth factor $f(nm/x)$, which serves as an approximation to the characteristic function of the interval $[0,1]$, followed by applying the Mellin transform to it.
	
	When squaring the sum $\tilde{U}$, we obtain typical sums for the dispersion method, namely $W$, $V$, and $U$:
	\begin{equation*}
		|\tilde{U}(D, N, M, t)|^2 \leqslant \left\|a\right\|_2^2\left(W - 2\Real V + U \right)\!. 
	\end{equation*}
	Here, $\left\|a\right\|_2$ denotes the $\ell_2$-norm of a sequence $a$,
	$$W = \sum_{m=-\infty}^{+\infty}\psi\left(\dfrac{m}{M}\right)\sum_{d_1, d_2\sim D}\chi_4(d_1)\chi_4(d_2)\sum\limits_{\substack{n_1, n_2\sim N \\ n_1m\equiv 1\!\!\!\!\!\pmod {d_1} \\ n_2m\equiv 1\!\!\!\!\!\pmod {d_2}}}b(n_1)\overline{b}(n_2),$$
	$$V= \sum_{m=-\infty}^{+\infty}\psi\left(\dfrac{m}{M}\right)\sum_{d_1, d_2\sim D}\dfrac{\chi_4(d_1)\chi_4(d_2)}{\varphi(d_2)}\sum\limits_{\substack{n_1, n_2\sim N\\ n_1m\equiv 1\!\!\!\!\!\pmod {d_1} \\ (n_2m, d_2)=1}}b(n_1)\overline{b}(n_2),$$
	$$U= \sum_{m=-\infty}^{+\infty}\psi\left(\dfrac{m}{M}\right)\sum_{d_1, d_2\sim D}\dfrac{\chi_4(d_1)\chi_4(d_2)}{\varphi(d_1)\varphi(d_2)}\sum\limits_{\substack{n_1, n_2\sim N \\ (n_1m, d_1)=1 \\ (n_2m, d_2)=1}}b(n_1)\overline{b}(n_2),$$
	and the function $\psi(t)$ is a smooth function supported on $[1/2, 5/2]$ and equal to $1$ for $1 \leqslant t \leqslant 2$. 
	The transformation of sums $V$ and $U$ poses no difficulties. Elementary asymptotic counting shows that these sums are equal to each other up to a lower-order term and are equal to some real number $U^{MT}$.
	
	By changing the order of summation for the sum $W$, we have
	$$W = \sum_{d_1, d_2 \sim D}\chi_4(d_1)\chi_4(d_2)\sum\limits_{\substack{n_1, n_2 \sim N \\ (n_1, d_1) = 1\\ (n_2, d_2) = 1 }}b(n_1)\overline{b}(n_2)\sum\limits_{\substack{m=-\infty \\ mn_1 \equiv 1 \!\!\!\!\! \pmod{d_1} \\ mn_2 \equiv 1 \!\!\!\!\! \pmod{d_2}}}^{+\infty}\psi\left(\dfrac{m}{M}\right)\!.$$
	Applying a variant of the Chinese remainder theorem to the inner sum where the moduli are not necessarily coprime gives us from the Poisson summation formula for the sum over $m$ that
	$$\sum_{m \equiv \nu \!\!\!\!\! \pmod{q}}\psi\left(\dfrac{m}{M}\right) =   \dfrac{M}{q}\hat{\psi}(0) + \dfrac{M}{q}\sum_{m \neq 0}\hat{\psi}\left(\dfrac{Mm}{q}\right)e^{2\pi i \frac{ma}{q}},$$
	where $q$ is the least common multiple of $d_1$ and $d_2$, and $\nu$ is some residue modulo $q$.
	The contribution from the first term in the sum $W$ will be denoted as $W^{MT}$. In what follows, the difference $W^{MT}-U^{MT}$ is transformed into a more convenient form using a standard formula for dispersion and is estimated using theorems on the distribution of prime numbers in arithmetic progressions. The contribution from the remaining sum over small values of $m$ is obtained using estimates on Kloosterman sums. The sum over large values of $m$ is estimated trivially.
	
	Finally, we note that following the proof of Theorem \ref{MainTheorem}, one can obtain the next term in asymptotics for the sum $S(x)$. Specifically, it can be shown that for some constant $K_{1}$ and any fixed $\varepsilon > 0$, the following equality holds:
	$$S(x) = Kx\sqrt{\ln x} + \dfrac{K_{1}x}{\sqrt{\ln x}} + O_\varepsilon\left(\dfrac{x}{(\ln x)^{1-\varepsilon}}\right)\!.$$
	
	Eventually, we expect (and leave this for further work) that a modification of methods from \cite{Fouvry Tenenbaum} will allow us to obtain an asymptotic expansion for the sum $Q(x)$:
	$$Q(x) = \dfrac{c_{1}x}{(\ln x)^{3/4}}+\dfrac{c_{2}x}{(\ln x)^{7/4}}+\cdots+\dfrac{c_{m}x}{(\ln x)^{m-1/4}}+O_m\left(\dfrac{x}{(\ln x)^{m+3/4}}\right)\!.$$
	In this expression, $m \geqslant  1$ is an arbitrarily fixed integer and $c_i$ are some constants.
	\section{Auxiliary Results}
	The following three lemmas are necessary for proving the main statements in \S\ref{Korol}.
	\begin{lemma}\label{L1}
		Let $\chi_1$ and $\chi_2$ be two distinct characters modulo $d \geqslant 2$, and let $\chi_4$ be defined in \eqref{chi4}. Then $\chi_1\chi_4$ and $\chi_2\chi_4$ are also two distinct characters, each of whose modulus divides $4d.$
	\end{lemma}
	\begin{proof}
		Assume $d$ is even. Since $\chi_1 \neq \chi_2$, there exists an integer $n$ that is coprime to $d$, such that 
		$$\chi_1(n) \neq \chi_2(n).$$
		Since $n$ is odd, we have $\chi_4(n) \neq 0$, thus
		$$\chi_1\chi_4(n) \neq \chi_2\chi_4(n)$$
		and, consequently, $\chi_1\chi_4 \neq \chi_2\chi_4.$
		
		For a prime power $p^\nu$, we will write $p^\nu || n$ if $p^\nu | n$ and $p^{\nu+1} \nmid n$.
		Assume $d$ is odd; then, by the uniqueness of character decomposition, there exists a prime power $p^\nu\,||\, d$ such that 
		$$\chi_1(\cdot|p^\nu) \neq \chi_2(\cdot|p^\nu),$$
		where $\chi_1(\cdot|p^\nu)$ and $\chi_2(\cdot|p^\nu)$ are certain characters modulo $p^\nu$, arising from the decompositions of $\chi_1$ and $\chi_2$, respectively.
		
		Since $\chi_1(\cdot|p^\nu)$ and $\chi_2(\cdot|p^\nu)$ are included in the decompositions of the characters $\chi_1\chi_4$ and $\chi_2\chi_4$, we again obtain $\chi_1\chi_4 \neq \chi_2\chi_4.$
		
		Now for $i=1, 2$, by the properties of characters, the modulus of $\chi_i\chi_4$ divides the least common multiple $[d, 4]$, which means it also divides $4d.$
		The lemma is proved.
	\end{proof}
	Next, we will use the notation $\mathbb{I}_A = \mathbb{I}(A)$ for the indicator function of a statement $A$.
	
	\begin{lemma}\label{L2}
		Let $\chi$ be a primitive character modulo $d \geqslant 3$, and let the character $\chi_4$ be defined in \eqref{chi4}. Then the primitive character $\chi^*$, inducing $\chi\chi_4$, satisfies the equality
		$$\chi^*(n) = \begin{cases} 
			\chi\chi_4(n), & \text{if } d \text{ is odd or } 8|d,\\
			\chi\left(n\left|\frac{d}{4}\right)\right.\!\!, & \text{if } 4\,||\, d.
		\end{cases}$$
		Moreover, the modulus of $\chi^*$ is $4d$ for odd $d$, $d/4$ for $4\,||\, d$, and $d$ for $8|d$.
	\end{lemma}
	\begin{proof}
		Assume $d$ is odd; then $\chi\chi_4$ is a primitive character as a product of primitive characters with coprime moduli. Since $(4, d) = 1$, the modulus of $\chi^*$ is $4d$.
		
		Now assume $d$ is even; since $\chi$ is a primitive character modulo $d$, we have
		$$d \not\equiv\,2 \pmod 4.$$
		This implies that either $4\, ||\, d$, or $8|d$. In the first case, since $\chi_4$ is the only primitive character modulo $4$, we have 
		$$\chi(n) = \chi\left( n \left| \frac{d}{4}\right.\right)\chi_4(n),$$
		where $\chi(\cdot|d/4)$ is some primitive character modulo $d/4$. Thus, since 
		$$\chi\chi_4(n) = \chi\left( n \left| \frac{d}{4} \right.\right)\mathbb{I}_{(n,2) = 1} = \chi\left( n \left| \frac{d}{4} \right.\right) \chi_0(n),$$
		where $\chi_0$ is the principal character modulo $d$, we obtain
		$$\chi^*(n) = \chi\left( n \left| \frac{d}{4} \right.\right),$$
		and the modulus of $\chi^*$ is equal to $d/4$.
		
		Now assume that $8|d$; then for some odd integer $d_1$ and $\alpha \geqslant 3$, we have $d = 2^\alpha d_1$ and
		$$\chi(n) = \chi(n|2^\alpha)\chi(n|d_1),$$
		where $\chi(\cdot|2^\alpha)$ and $\chi(\cdot|d_1)$ are primitive characters modulo $2^\alpha$ and $d_1$, respectively. Since $\alpha \geqslant 3$, it is known that $\chi(n|2^\alpha)$ has the form
		$$
		\chi(n|2^\alpha) = 
		\begin{cases}
			(-1)^{m_0\gamma_0} e^{2\pi i\frac{m_1\gamma_1}{2^{\alpha-2}}}, & \text{if } n \text{ is odd,}\\
			0, & \text{otherwise.}
		\end{cases}
		$$
		Here, $\gamma_0$ and $\gamma_1$ are systems of indices modulo $2^\alpha$:
		$$n \equiv (-1)^{\gamma_0} 5^{\gamma_1}\!\!\!\!\!\pmod {2^{\alpha}}$$
		and
		$$m_0\in\left\{0, 1\right\}\!,\ m_1\in\left\{0, 1, \ldots, 2^{\alpha-2}-1\right\}\!,\ (m_1, 2) = 1.$$
		Since 
		$$\chi_4(n) = (-1)^{\gamma_0}\mathbb{I}_{(n, 2) = 1},$$
		it follows that $\chi(\cdot|2^\alpha)\chi_4$ is also a primitive character modulo $2^\alpha$. Hence,
		$$\chi\chi_4(n) = (\chi(n|2^\alpha)\chi_4(n))\chi(n|d_1)$$
		is a primitive character modulo $d$. Thus, we have $\chi^* = \chi\chi_4$, and the modulus of $\chi^*$ equals $d$. The lemma is proved.
	\end{proof}
	The next lemma is a slight modification of a well-known result of Montgomery (see Lemma 6 in \cite{Korol2010}).
	\begin{lemma}\label{L3}
		For any $Q\geqslant 3, T\geqslant 3$, and $1/2\leqslant \sigma \leqslant 1$, we define
		$$S(Q; T, \sigma) = \sum_{q\leqslant Q}{\sum_{\chi\, \textup{mod}\, q}}^{*} N(\sigma; T, \chi\chi_4),$$
		where $*$ denotes summation over primitive characters, and
		$$N(\sigma; T, \chi) = \#\left\{\rho = \beta+i\gamma: L(\rho, \chi) = 0,\ \beta\geqslant \sigma, |\gamma|\leqslant T\right\}.$$
		Then 
		$$S(Q; T, \sigma) \ll (Q^2T)^{\vartheta(\sigma)}(\ln(QT))^{14},$$
		where
		$$\vartheta(\sigma) = \begin{cases}
			\frac{3(1-\sigma)}{2-\sigma}, & \text{if } \frac{1}{2}\leqslant \sigma\leqslant \frac{4}{5},\\
			\frac{2(1-\sigma)}{\sigma}, & \text{if } \frac{4}{5}\leqslant \sigma\leqslant 1,
		\end{cases}$$
		and the constant in the symbol $\ll$ is absolute.
	\end{lemma}
	\begin{proof}
		We have
		$$S(Q; T, \sigma) = S_1 + S_2 + S_3,$$
		where
		$$S_1 = \sum\limits_{\substack{q\leqslant Q \\ q\, \text{is odd}}} {\sum_{\chi\, \textup{mod}\, q}}^{*} N(\sigma; T, \chi\chi_4),$$
		$$S_2 = \sum\limits_{\substack{q\leqslant Q \\ 4\,||\, q}} {\sum_{\chi\, \textup{mod}\, q}}^{*} N(\sigma; T, \chi\chi_4)$$
		and
		$$S_3 = \sum\limits_{\substack{q\leqslant Q \\ 8 | q}} {\sum_{\chi\, \textup{mod}\, q}}^{*} N(\sigma; T, \chi\chi_4).$$
		Note that for any $\sigma\geqslant 1/2$, the equality 
		\begin{equation}\label{N}
			N(\sigma, T, \chi) = N(\sigma, T, \chi^*)
		\end{equation}
		holds, where $\chi^*$ is the primitive character inducing the character $\chi$. Indeed, if $q_1 | q$ is the modulus of the character $\chi^*$ and $\chi_0$ is the principal character modulo $q$, then
		\begin{multline}
			L(s, \chi) = \prod_p\left(1-\dfrac{\chi^*(p)\chi_0(p)}{p^s}\right)^{-1} \\ 
			= \prod\limits_{\substack{p\,\nmid\, q \\ p\,\nmid\, q_1}}\left(1-\dfrac{\chi^*(p)}{p^s}\right)^{-1} = L(s, \chi^*)\prod\limits_{\substack{p\, |\, q \\ p\,\nmid\, q_1}}\left(1-\dfrac{\chi^*(p)}{p^s}\right)\!.
		\end{multline}
		From this it follows that $L(s, \chi)$ and $L(s, \chi^*)$ have the same nontrivial zeros, counted with multiplicity. Thus, we have the equality (\ref{N}). Therefore, by Lemmas~\ref{L1} and~\ref{L2}, we obtain
		$$S_1 \leqslant \sum_{q\leqslant 4Q}{\sum_{\chi\, \textup{mod}\, q}}^{*} N(\sigma; T, \chi)$$
		and
		$$\max(S_2, S_3) \leqslant \sum_{q\leqslant Q}{\sum_{\chi\, \textup{mod}\, q}}^{*} N(\sigma; T, \chi).$$
		Applying Lemma 6 from work~\cite{Korol2010}, we obtain the required result.
		
	\end{proof}
	
	The following auxiliary statements are used in the sections where the sum $Q^{Err_2}$, defined in \eqref{QQErr2}, is estimated.
	
	\begin{theorem}[Shiu]\label{L3.1}
		Let $f$ be a non-negative multiplicative function such that
		\begin{enumerate}
			\item there exists a constant $A_1>0$ such that for any powers of primes it holds $f(p^\nu)\leqslant A_1^\nu$;
			\item for any $\varepsilon > 0$, there exists a quantity $A_2(\varepsilon)>0$ such that for any $n\geqslant 1$, it holds
			$f(n)\leqslant A_2(\varepsilon)n^{\varepsilon}.$ 
		\end{enumerate}
		Furthermore, let $a$ and $q$ be coprime integers such that $0\leqslant a<q$, and let $0<\alpha, \beta<\frac{1}{2}$. Then for $q<y^{1-\alpha}$ and $x^\beta<y\leqslant x$, for sufficiently large $x$, the inequality 
		$$\sum\limits_{\substack{x-y<n\leqslant x \\ n\equiv a\!\!\!\!\!\pmod{q}}} f(n)\ll \dfrac{y}{\varphi(q)}\dfrac{1}{\ln x}\exp\left(\sum\limits_{\substack{p\leqslant x \\ p\nmid q}}\dfrac{f(p)}{p}\right)\!$$
		holds,
		where the constant in the symbol $\ll$ depends on $\alpha, \beta, A_1$, and $A_2(\varepsilon)$.
	\end{theorem}
	This is Theorem 1 from \cite{Shiu80}.
	
	\begin{theorem}[The Richert~--~Halberstam inequality]\label{L3.2}
		Let $f$ be a non-negative multiplicative function such that for some constants $A, B>0$, the inequalities hold
		$$\sum_{p\leqslant y} f(p)\ln p\leqslant Ay,\ \ y\geqslant 0,$$
		$$\sum_{p}\sum_{\nu\geqslant 2}\dfrac{f(p^\nu)\log p^\nu}{p^\nu}\leqslant B.$$
		Then for $x>1$, the following estimate holds:
		$$\sum_{n\leqslant x} f(n)\leqslant (A+B+1)\dfrac{x}{\ln x}\sum_{n\leqslant x}\dfrac{f(n)}{n}.$$
	\end{theorem}
	This is Theorem 01 from Chapter 0 of the book \cite{Divisors}.
	
	By virtue of the inequality
	\begin{multline*}
		\sum_{n\leqslant x}\dfrac{f(n)}{n}
		\leqslant \prod_{p\leqslant x}\left(1+\dfrac{f(p)}{p}+\dfrac{f(p^2)}{p^2}+\cdots\right)\\
		=\exp\left(\sum_{p\leqslant x} \ln\left(1+\dfrac{f(p)}{p}+\dfrac{f(p^2)}{p^2}+\cdots \right) \right)\\
		\leqslant \exp\left(\sum_{p\leqslant x}\left( \dfrac{f(p)}{p}+\dfrac{f(p^2)}{p^2}+\cdots\right)  \right)
		= \exp\left(B + \sum_{p\leqslant x} \dfrac{f(p)}{p} \right)\!,
	\end{multline*}
	for any function 
	$f$
	 satisfying the conditions of Theorem \ref{L3.2}, we have:
	$$\sum_{n\leqslant x} f(n)\leqslant (A+B+1)e^B\dfrac{x}{\ln x}\exp\left(\sum_{p\leqslant x}\dfrac{f(p)}{p}\right)\!.$$
\begin{lemma}\label{L4}
	For a real number $x$, define
	$$
	\rho(x)=
	\begin{cases}
		\exp\left(\dfrac{1}{x^2-1}\right)\!, & \text{if}\ |x|<1,\\
		0, & \text{otherwise.}
	\end{cases}
	$$
	Then the function $\rho(x)$ is infinitely differentiable, and for any non-negative integer $j\geqslant 0$ and any $x$, the following estimate holds:
	$$|\rho^{(j)}(x)|\leqslant (2^j j!)^2.$$
\end{lemma}
This is Lemma 9 from \cite{BomFried86}.
	
From the definition of the function $\rho(x)$, it follows that for any $j\geqslant 0$, the following holds:
\begin{equation}\label{rhojx}
	\rho^{(j)}(x) = 0,\ \ (|x|\geqslant 1).
\end{equation}

\begin{lemma}\label{L5}
	Let $a<b$ and $0<\delta<b-a$. Then there exists a function $\sigma(x) = \sigma(a, b, \delta; x)$ from the class $C^\infty(\mathbb{R})$ such that for any real number $x$, it holds $0\leqslant \sigma(x)\leqslant 1$, specifically:
	\begin{equation}\label{sigma(x)}
		\sigma(x)=
		\begin{cases}
			1,\ & \text{if}\ a+{\delta}/{2}\leqslant x\leqslant b-{\delta}/{2},\\
			0,\ & \text{if}\ x\leqslant a-{\delta}/{2}\ \text{or}\ x\geqslant b+{\delta}/{2}.
		\end{cases}
	\end{equation}
	Moreover, for any non-negative integer $\nu\geqslant 0$, we have
	\begin{equation}\label{sigma_deriv}|\sigma^{(\nu+1)}(x)|=
		\begin{cases}
			O\left( \left(\dfrac{8}{\delta}\right)^{\nu+1}(\nu!)^2\right)\!\  &  \text{for any real number}\ x,\\
			0,\ & \text{if}\ x\in \mathcal{I},
		\end{cases}
	\end{equation}
	where the constant in the symbol $O$ is absolute and
	$$\mathcal{I} = (-\infty, a-{\delta}/{2}] \cup  [a+{\delta}/{2}, b-{\delta}/{2}] \cup [b+{\delta}/{2}, +\infty)\!.$$
\end{lemma}
	\begin{proof}
		Let us take
		$$
		\sigma(x) = \dfrac{1}{c}\int\limits_{x-b}^{x-a}\rho\left(\dfrac{2t}{\delta}\right)dt,
		$$
		where
		$$
		c = \int\limits_{-\infty}^{+\infty}\rho\left(\dfrac{2t}{\delta}\right)dt = \int\limits_{-{\delta}/{2}}^{{\delta}/{2}}\rho\left(\dfrac{2t}{\delta}\right)dt.
		$$ 
		Since $\rho(x)$ is non-negative, we obtain that $0\leqslant \sigma(x)\leqslant 1$ for any $x$. Suppose that $x\leqslant a-{\delta}/{2}$; then 
		$$
		\dfrac{2t}{\delta}\leqslant \dfrac{2(x-a)}{\delta}\leqslant -1,
		$$
		and thus $\sigma(x) = 0.$ Similarly, $\sigma(x) = 0$ when $x\geqslant b+{\delta}/{2}$.
		
		Now suppose that $a+{\delta}/{2}\leqslant x\leqslant b-{\delta}/{2}$; then
		$x-b\leqslant -\delta/2$ and $x-a\geqslant \delta/2$. Therefore,
		$$
		\sigma(x) = \dfrac{1}{c}\int\limits_{-\delta/2}^{\delta/2}\rho\left(\dfrac{2t}{\delta}\right)dt = 1.
		$$
		
		To prove the equation \eqref{sigma_deriv}, note that
		$$
		\sigma^{(\nu+1)}(x) = \dfrac{1}{c}\left(\dfrac{2}{\delta}\right)^\nu \left( \rho^{(\nu)}\left(\dfrac{2(x-a)}{\delta} \right) - \rho^{(\nu)}\left(\dfrac{2(x-b)}{\delta} \right) \right)\!.
		$$
		
		From Lemma~\ref{L4}, considering the equality $$c = \dfrac{\delta}{2}\int_{-\infty}^{+\infty}\rho(w) dw \asymp \delta $$ we obtain
		$$\sigma^{(\nu+1)}(x)\ll \left(\dfrac{8}{\delta}\right)^{\nu+1}(\nu!)^2,\ \ \nu\geqslant 0.$$
		
		The second equality in the equation \eqref{sigma_deriv} follows from the equations \eqref{sigma(x)} and \eqref{rhojx}. The lemma is proved.
	\end{proof}
	Let us define for further use
	\begin{equation}\label{psi(x)}
		\psi(x) = \sigma\left(\frac{3}{4}, \frac{9}{4}, \frac{1}{2}; x \right)\!.
	\end{equation}
	Then for any $x$, it holds $0\leqslant \psi(x)\leqslant 1$, specifically 
	$$
	\psi(x)=
	\begin{cases}
		1,\ & \text{if}\ 1\leqslant x\leqslant 2,\\
		0,\ & \text{if}\ x\leqslant\frac{1}{2}\ \text{or}\ x\geqslant \frac{5}{2}.
	\end{cases}
	$$
	Moreover, for any non-negative integer $k\geqslant 0$, the following holds:
	$$|\psi^{(k+1)}(x)|=
	\begin{cases}
		O\left(16^k (k!)^2\right)\!\ & \text{for any}\ x,\\
		0,\ & \text{if}\ \ x\in (-\infty, \frac{1}{2}]\cup [1, 2]\cup [\frac{5}{2}, +\infty),
	\end{cases}
	$$
	where the constant in the symbol $O$ is absolute.
	\begin{lemma}\label{L5.4}
		Uniformly for all real $\lambda$, we have
		$$\hat{\psi}(\lambda)\ll e^{-\frac{\sqrt{|\lambda}|}{2}}.$$
	\end{lemma}
	
	\begin{proof}
		Let $ \lambda_0>1$ be sufficiently large and $|\lambda| \geqslant \lambda_0$. Using the fact that for $k\geqslant 0$ it holds 
		$$\psi^{(k)}(1/2) = \psi^{(k)}(5/2) = 0,$$ by integration by parts, we find
		\begin{multline*}
			\hat{\psi}(\lambda) = \int_{-\infty}^{+\infty}\psi(x)e^{-2\pi i x \lambda} dx = \int_{1/2}^{5/2}\psi(x)e^{-2\pi i x \lambda} dx \\
			= -\dfrac{1}{2\pi i\lambda} \int_{1/2}^{5/2}\psi'(x)e^{-2\pi i x \lambda} dx \\
			= \cdots = \dfrac{(-1)^{k+1}}{(2\pi i \lambda)^{k+1}}\int_{1/2}^{5/2}\psi^{(k+1)}(x)e^{-2\pi i x \lambda} dx\ll \dfrac{1}{|\lambda|}\dfrac{(4^k k!)^2}{(2\pi |\lambda|)^k}.
		\end{multline*}
		
		Using Stirling's formula, we obtain:
		$$\dfrac{(4^k k!)^2}{(2\pi |\lambda|)^k}\leqslant \dfrac{16^k \cdot 2\pi k (k/e)^{2k}e^{\frac{1}{6k}}}{(2\pi |\lambda|)^k} = 2\pi k\left(\dfrac{4k}{e\sqrt{2\pi|\lambda|}}\right)^{2k}e^{\frac{1}{6k}}.$$
		
		Let us choose $k = \lfloor\sqrt{|\lambda|}/4\rfloor\geqslant 1$; then 
		we obtain
		$$2\pi k\left(\dfrac{4k}{e\sqrt{2\pi |\lambda|}}\right)^{2k}\leqslant 2\pi k\left(\dfrac{1}{e\sqrt{2\pi}}\right)^{2k}\ll \sqrt{|\lambda|}\left(\dfrac{1}{e\sqrt{2\pi}}\right)^{{\sqrt{|\lambda|}}/{2}}\ll e^{-{\sqrt{|\lambda|}}/{2}}.$$
		
		Thus, for $|\lambda|\geqslant \lambda_0$, the estimate
		$$\hat{\psi}(\lambda)\ll e^{-{\sqrt{|\lambda|}}/{2}}$$
		holds. Note that the same estimate is also valid for $|\lambda|\leqslant \lambda_0$. Indeed,
		$$\hat{\psi}(\lambda) = \int_{1/2}^{5/2}\psi(x)e^{-2\pi i x \lambda} dx \ll 1\ll_{\lambda_0} e^{-{\sqrt{|\lambda|}}/{2}}.$$
		
		The lemma is proved.
	\end{proof}
\begin{lemma}\label{L5.5}
	Uniformly for all real $x$ and for any integer $k\geqslant 0$, we have
	$$\hat{\psi}^{(k)}(x) \ll_k \min\left(1, \dfrac{1}{|x|^k}\right)\!.$$
\end{lemma}

\begin{proof}
	We have
	$$\hat{\psi}(x) = \int\limits_{-\infty}^{+\infty}\psi(t)e^{-2\pi i t x} dt = \int\limits_{-1/2}^{5/2}\psi(t)e^{-2\pi i t x} dt.$$
	Differentiating under the integral sign $k\geqslant 0$ times, we obtain
	$$\hat{\psi}^{(k)}(x) = (-2\pi i)^k\int\limits_{-1/2}^{5/2}t^k\psi(t)e^{-2\pi i t x} dt.$$
	From this, we find
	$$\hat{\psi}^{(k)}(x)\ll (2\pi)^k \int\limits_{-1/2}^{5/2}t^k dt \ll_k 1.$$
	
	Furthermore, for any $0\leqslant j\leqslant k$, we have 
	$$P_j\left({1}/{2}\right) = P_j\left({5}/{2}\right) = 0,$$
	where
	$$P_j(t) = \dfrac{d^j}{dt^j}(t^k\psi(t)).$$ 
	By repeated integration by parts, for any $0\leqslant j\leqslant k$, we find
	$$\hat{\psi}^{(k)}(x) = \dfrac{(-1)^k(2\pi i)^{k-j}}{x^j}\int\limits_{1/2}^{5/2}P_j(t)e^{-2\pi i xt} dt.$$
	
	In particular, for $j = k$, we obtain
	$$\hat{\psi}^{(k)}(x) = \dfrac{(-1)^k}{x^k}\int\limits_{1/2}^{5/2}P_k(t)e^{-2\pi i xt} dt \ll_k \dfrac{1}{|x|^k}.$$
	
	The lemma is proved.
\end{proof}
	
	For $0<\delta<\frac{1}{2}$, we define
	\begin{equation}\label{f_delta(x)}
		f_\delta(x) = \sigma\left(\frac{3\delta}{2}, 1+\frac{\delta}{2}, \delta; x\right)\!, 
	\end{equation}
	where the function $\sigma(x)$ is defined in Lemma \ref{L5}. Then for any $x$, it holds $0\leqslant f_\delta(x)\leqslant 1$, specifically 
	$$
	f_\delta(x)=
	\begin{cases}
		1,\ & \text{if}\ 2\delta\leqslant x\leqslant 1,\\
		0,\ & \text{if}\ x\leqslant\delta\ \text{or}\ x\geqslant 1+\delta.
	\end{cases}
	$$
	Moreover, for any non-negative integer $\nu\geqslant 0$, the following holds:
	$$|f_\delta^{(\nu+1)}(x)|=
	\begin{cases}
		O\left( \left(\dfrac{8}{\delta}\right)^{\nu+1}(\nu!)^2\right)\!\ & \text{for any}\ x,\\
		0,\ & \text{if}\ \ x\in (-\infty, \delta] \cup [2\delta, 1] \cup [1+\delta, +\infty), 
	\end{cases}
	$$
	where the constant in the symbol $O$ is absolute.
	
	Let us denote by $F_\delta$ the Mellin transform of the function $f_\delta$:
	\begin{equation}\label{F_d}
		F_\delta(s) = \int_0^{+\infty}f_\delta(u)u^{s-1}du.
	\end{equation}
	\begin{lemma}\label{L8}
		Let $0<\delta <1/2$. Then for any real $t$, we have
		$$|F_\delta(it)|\ll \ln\left(\dfrac{1}{\delta}\right)\!.$$
		Moreover, if $|t|\geqslant 1$, then 
		$$|F_\delta(it)|\ll \dfrac{1}{\delta t^2}.$$
	\end{lemma}
	
	\begin{proof}
		For an arbitrary $t$, we have
		$$F_\delta(it) = \int_\delta^{1+\delta}f_\delta(u)u^{it-1}du \ll \int_\delta^{1+\delta}\dfrac{du}{u} = \ln\left(1+\dfrac{1}{\delta} \right)\ll \ln\left(\dfrac{1}{\delta} \right)\!.$$
		
		Now let $|t|\geqslant 1$. Since 
		$$f^{(\nu)}_\delta(\delta) = f^{(\nu)}_\delta(1+\delta) = 0,\ \ (\nu \geqslant 0),$$ 
		we have
		\begin{multline*}
			F_\delta(it) = \int_\delta^{1+\delta}f_\delta(u)\dfrac{d{u^{it}}}{it} = -\dfrac{1}{it}\int_\delta^{1+\delta}f_\delta'(u)u^{it}du \\
			= -\dfrac{1}{it}\int_\delta^{1+\delta}f_\delta'(u)\dfrac{d u^{it+1}}{it+1} = \dfrac{1}{it(it+1)}\int_\delta^{1+\delta}f_\delta''(u)u^{1+it}du \\
			\ll \dfrac{1}{t^2}\left( \int_\delta^{2\delta}+\int_1^{1+\delta}\right)\dfrac{du}{\delta^2}\ll \dfrac{1}{\delta t^2}.
		\end{multline*}
		
		The lemma is proved.
	\end{proof}
\begin{lemma}\label{L6}
	Let $f \in C^\infty(\mathbb{R})$, and suppose that for any $n\geqslant  1$ the following holds:
	\begin{equation}\label{Cond_Poisson}
		\lim\limits_{x\to \pm\infty}|x|^nf(x) = 0.
	\end{equation}
	Then for any $x$ and $H>0$, we have
	$$\sum_{n=-\infty}^{+\infty}f\left(\dfrac{n+x}{H}\right) = H\sum_{m=-\infty}^{+\infty}\hat{f}(Hm)e^{2\pi i m x},$$
	where $\hat{f}$ denotes the Fourier transform of the function $f$.
\end{lemma}

\begin{proof}
	For the class of functions $f \in C^\infty(\mathbb{R})$ satisfying the condition \eqref{Cond_Poisson}, the Poisson summation formula 
	$$\sum_{n=-\infty}^{+\infty}f(n+x) = \sum_{m=-\infty}^{+\infty}\hat{f}(m)e^{2\pi i m x}$$
	holds. By applying the Poisson summation formula to the function 
	$$g(y) = f\left(y/H \right),$$ 
	which belongs to the same class, we obtain the required result.
\end{proof}
	\begin{lemma}\label{L7}
		Let the function $\psi$ be defined in \eqref{psi(x)}. Then for $q\geqslant 2$, $0\leqslant a<q$, $M> \exp(\sqrt{2\ln q})$, and $H\geqslant (5q/M)(\ln M)^4$, the following formulas hold:
		\begin{equation}\label{eq1}
			\sum_{m\equiv a \!\!\!\!\!\pmod q}\psi\left(\dfrac{m}{M}\right) = \dfrac{M\hat{\psi}(0)}{q} + \dfrac{M}{q}\sum_{1\leqslant |m|\leqslant H}\hat{\psi}\left(\dfrac{mM}{q}\right)e^{2\pi i\frac{ma}{q}} + O\left(\dfrac{1}{qM}\right)\!,
		\end{equation}
		\begin{equation}\label{eq2}
			\sum_{(m, q) = 1}\psi\left(\dfrac{m}{M}\right) = \dfrac{\varphi(q)M\hat{\psi}(0)}{q} + O\left(\tau(q)(\ln M)^2\right)\!,
		\end{equation}
		where $\tau(q)$ is the divisor function.
	\end{lemma}
	\begin{proof}
		We have
		\begin{equation*}
			W_1 = \sum_{m\equiv a\!\!\!\!\!\pmod{q}}\psi\left(\dfrac{m}{M}\right)  =  \sum_{n=-\infty}^{+\infty}\psi\left( \dfrac{n+{a}/{q}}{{M}/{q}}\right)\!.
		\end{equation*}
		 Thus using Lemma \ref{L6} for each $H\geqslant (5q/M)(\ln M)^4$ we find that
		\begin{multline}\label{W1}
			W_1 = \dfrac{M}{q}\sum_{m=-\infty}^{+\infty}\hat{\psi}\left(\dfrac{Mm}{q}\right)e^{2\pi i \frac{ma}{q}}  = \dfrac{M}{q}\hat{\psi}(0)\\
			+	\dfrac{M}{q}\sum_{1\leqslant |m|\leqslant H}\hat{\psi}\left(\dfrac{Mm}{q}\right)e^{2\pi i \frac{ma}{q}}
			+\dfrac{M}{q}\sum_{|m|> H}\hat{\psi}\left(\dfrac{Mm}{q}\right)e^{2\pi i \frac{ma}{q}}.
		\end{multline}
		
		Using Lemma \ref{L5.4}, we estimate the last sum in \eqref{W1}. We have
		\begin{equation*}
			R = \dfrac{M}{q}\sum_{|m|> H}\hat{\psi}\left(\dfrac{Mm}{q}\right)e^{2\pi i \frac{ma}{q}}\ll \dfrac{M}{q}e^{-\frac{1}{2}\sqrt{\frac{HM}{q}}}+\dfrac{M}{q}\int_{H}^{+\infty} e^{-\frac{1}{2}\sqrt{\frac{xM}{q}}}dx.
		\end{equation*}
		By integrating the resulting integral by parts, we obtain:
		\begin{equation*}
			\int_{H}^{+\infty} e^{-\sqrt{\frac{xM}{4q}}}dx = \dfrac{8q}{M}\int_{\sqrt{\frac{HM}{4q}}}^{+\infty} e^{-w}wdw = \dfrac{8q}{M}e^{-\sqrt{\frac{HM}{4q}}}\left(\sqrt{\dfrac{HM}{4q}}+1 \right)\!.
		\end{equation*}
		Thus, taking into account the estimates $MH/q\geqslant 5 (\ln M)^4$ and $M>\exp\left(\sqrt {2\ln q}\right)$, which are satisfied by assumption, we obtain
		\begin{multline*}
			R\ll \left( \dfrac{M}{q}+\sqrt{\dfrac{HM}{4q}}+1\right)e^{-\frac{1}{2}\sqrt{\frac{HM}{q}}}\\
			\ll \left(M+(\ln M)^2 +1\right) e^{-\sqrt{\frac{5}{4}}(\ln M)^2}\ll e^{-(\ln M)^2}. 
		\end{multline*}
		Let $M\geq 10$, then $\ln M\leqslant (\ln M)^2/2$ and
		$$e^{(\ln M)^2-\ln M}\geqslant e^{\frac{(\ln M)^2}{2}}>q.$$
		Therefore, for $M\geq 10$ we have
		$$R\ll e^{-(\ln M)^2} \ll \dfrac{1}{qM}.$$
		If $M<10$, then the last estimate is obvious. Thus, the equality \eqref{eq1} is proven.
		
		Now let us proceed to prove the equality \eqref{eq2}. Take $H = 5q (\ln M)^4$, then using the M\"obius function summation property and the  equality \eqref{eq2}, we obtain
		\begin{multline*}
			W_2 = \sum_{(m,q) = 1}\psi\left(\dfrac{m}{M}\right) \\= \sum_{m=-\infty}^{+\infty}\psi\left(\dfrac{m}{M}\right)\sum_{d|(m,q)}\mu(d)
			= \sum_{d|q} \mu(d)\sum_{m \equiv 0\!\!\!\!\!\pmod{d}}\psi\left(\dfrac{m}{M}\right) \\
			= \sum_{d|q} \mu(d)\left(\dfrac{M}{d}\hat{\psi}(0)+\dfrac{M}{d}\sum_{1\leqslant |m|\leqslant H}\hat{\psi}\left(\dfrac{mM}{d}\right)+O\left(\dfrac{1}{dM} \right)  \right). 
		\end{multline*}
		Taking into account the inequality $e^{-\frac{1}{2}\sqrt{\lambda}}\leqslant \frac{10}{\lambda}$, which holds for $\lambda>0$, we obtain
		\begin{multline*}
			W_2 = M\hat{\psi}(0)\sum_{d|q}\dfrac{\mu(d)}{d}+O\left(\sum_{d|q}\dfrac{M}{d}\sum_{1\leqslant m\leqslant H}\dfrac{d}{mM}\right)+O\left(\dfrac{\tau(q)}{M}\right)\\
			= M\hat{\psi}(0)\dfrac{\varphi(q)}{q} + O\left(\tau(q)\ln H\right) \\
			= M\hat{\psi}(0)\dfrac{\varphi(q)}{q} + O\left(\tau(q)(\ln M)^2\right).
		\end{multline*}
		The lemma is proved.
	\end{proof}

	Let $\Delta >1$, and let the numbers $\Delta_1, \Delta_2$ be coprime, with $\Delta_1, \Delta_2 | \Delta^\infty$. We also set
	$$P = \mathcal{D}\mathcal{D}_1,\ Q = \mathcal{D}_2,$$
	where
	$$\mathcal{D} = \prod\limits_{\substack{p^\nu || \Delta \\ p\nmid \Delta_1\Delta_2}} p^\nu,\ \ \mathcal{D}_1 = \prod\limits_{\substack{p^\nu || \Delta\Delta_1 \\ p| \Delta_1}} p^\nu,\ \ \mathcal{D}_2 =\prod\limits_{\substack{p^\nu || \Delta\Delta_2 \\ p| \Delta_2}} p^\nu.$$
	It is easy to verify that for $i=1, 2$ the following holds:
	\begin{equation}\label{Div1}
		\mathcal{D}\mathcal{D}_i | \Delta\Delta_i,
	\end{equation}
	as well as
	\begin{equation}\label{Div2}
		PQ = \Delta\Delta_1\Delta_2
	\end{equation} and
	\begin{equation}\label{Div3}
		\Delta_1 | P, \Delta_2 | Q.
	\end{equation} 
	For positive integers $a$ and $q$ with the condition $(a, q) = 1$, we denote by $a^*_q$ the multiplicative inverse of $a$ modulo $q$.
	
	We will prove the following lemma, which generalizes the Chinese remainder theorem to the case where the moduli are not coprime.
	\begin{lemma}\label{L9}
		Let $a$ and $b$ be integers. Then the system of congruences 
		\begin{equation}\label{sys1}
			\begin{cases}
				m\equiv a \!\!\!\!\!\pmod{\Delta\Delta_1}\\
				m\equiv b \!\!\!\!\!\pmod{\Delta\Delta_2}
			\end{cases}
		\end{equation}
		is equivalent to the system
		\begin{equation}\label{sys2}\begin{cases}
				m\equiv \lambda \!\!\!\!\!\pmod{\Delta\Delta_1\Delta_2}\\
				a\equiv b \!\!\!\!\!\pmod{\Delta},
			\end{cases}
		\end{equation}
		where 
		\begin{equation}\label{lambda}
			\lambda = \alpha Q Q_P^* + \beta P P^*_Q
		\end{equation}
		and
		\begin{equation}\label{alphabeta}
			\alpha \equiv a \!\!\!\! \pmod {\Delta\Delta_1},\ \ \beta \equiv b \!\!\!\! \pmod {\Delta\Delta_2}.
		\end{equation}
	\end{lemma}
	\begin{proof}
		Let $a\equiv b\pmod{\Delta}$ (otherwise, \eqref{sys1} is not satisfied). From \eqref{Div1} and \eqref{sys1}, it follows that
		\begin{equation}\label{sys3}
			\begin{cases}
				m\equiv a \!\!\!\!\pmod{P}\\
				m\equiv b \!\!\!\! \pmod{Q}.
			\end{cases}
		\end{equation}
		Since $(P, Q) = 1$, by the Chinese remainder theorem and \eqref{Div2}, the system \eqref{sys3} is equivalent to the congruence
		\begin{equation}\label{sys4}
			m \equiv \mu \!\!\!\!\pmod{\Delta \Delta_1\Delta_2},
		\end{equation}
		where $\mu = aQQ^*_P + bPP^*_Q.$
		
		Now we will show that the converse is also true: the congruence \eqref{sys4} implies \eqref{sys1}. For some integers $s, t$, and $r$, we have
		$$QQ^*_P = 1 + sP,\ \ PP^*_Q = 1+ tQ,\ \ b = a + r\Delta.$$
		Since 
		$s \equiv -P^*_Q\pmod{Q}$, there exists some integer $w$ such that
		$s+P^*_Q = wQ$. Then we have
		\begin{equation*}
			QQ^*_P = 1+sP = 1+ (wQ-P^*_Q)P
			=1+wPQ-1-tQ
			= Q(wP-t),
		\end{equation*}
		and thus $t+Q^*_P = wP$. From this, we find
		\begin{multline*}
			\mu-a = a(QQ^*_P-1)+bPP^*_Q\\
			=asP+(a+r\Delta)PP^*_Q
			=aP(s+P^*_Q) + r(\Delta P)P_Q^*\\
			=awPQ + r(\Delta P)P_Q^*.
		\end{multline*}
		Taking into account \eqref{Div2} and \eqref{Div3}, we obtain that $\mu\equiv a\pmod{\Delta\Delta_1}$. Similarly, we find $\mu \equiv b \pmod{\Delta\Delta_2}$. Thus, the congruence \eqref{sys4}, and therefore the system \eqref{sys3} under the condition $a\equiv b\pmod{\Delta}$, are equivalent to the original system \eqref{sys1}.
		
		Now let $\alpha$ and $\beta$ be defined by the equalities \eqref{alphabeta}, then for some integers $u$ and $v$, we have
		$\alpha = a + u\Delta\Delta_1$ and $\beta = b+v\Delta\Delta_2.$ Consequently,
		$$\lambda - \mu = u\Delta\Delta_1QQ^*_P + v\Delta\Delta_2PP^*_Q,$$
		and therefore, due to \eqref{Div3}:
		$$
		\begin{cases*}
			\lambda\equiv\mu\equiv a \!\!\!\!\pmod{\Delta\Delta_1} \\
			\lambda\equiv\mu\equiv b \!\!\!\!\pmod{\Delta\Delta_2}.
		\end{cases*}
		$$
		Thus, since \eqref{sys1} and \eqref{sys4} are equivalent under the condition $a\equiv b\pmod{\Delta}$, we obtain $\lambda \equiv \mu \pmod{\Delta\Delta_1\Delta_2}$. The lemma is proved.
		\end{proof}
	\begin{lemma}\label{L10}
		Let $a, b, c$, and $d$ be positive integers such that $a, b$, and $cd$ are pairwise coprime. Let $(cd, e) = 1$, then the following congruence holds:
		$$\dfrac{(acd)^*_b}{b} + \dfrac{(abe)^*_c}{c} \equiv \dfrac{1}{abcd}-\dfrac{(bcd)^*_a}{a} + \dfrac{(d-e)(abe)^*_{cd}}{cd}\!\!\!\!\pmod{1}.$$
	\end{lemma}
	\begin{proof}
		First, we will show that the following holds:
		\begin{equation}\label{cong_1}
			\dfrac{(acd)^*_b}{b} \equiv \dfrac{1}{abcd}-\dfrac{(bcd)^*_a}{a} - \dfrac{(ab)^*_{cd}}{cd}\!\!\!\!\pmod{1}.
		\end{equation}
		Indeed, by multiplying by $abcd$, we obtain
		$$acd(acd)^*_b \equiv 1 - bcd(bcd)^*_a-ab(ab)^*_{cd} \!\!\!\!\pmod{abcd}.$$ 
		The resulting congruence holds modulo $a, b$, and $cd$, and thus it also holds modulo $abcd$.
		
		Now we will show that the following holds:
		\begin{equation}\label{cong_2}
			\dfrac{(abe)^*_c}{c} - \dfrac{(ab)^*_{cd}}{cd} \equiv \dfrac{(d-e)(abe)^*_{cd}}{cd}\!\!\!\!\pmod{1}. 
		\end{equation}
		Indeed, by multiplying by $cd$, we obtain
		$$d(abe)^*_c - (ab)^*_{cd} \equiv (d-e)(abe)^*_{cd} \!\!\!\!\pmod{cd}.$$
		Since $(ab, cd) = 1$, the last congruence is equivalent to the following:
		$$abd(abe)^*_c - 1 \equiv (d-e)e^*_{cd} \!\!\!\!\pmod{cd}.$$
		Expanding the brackets, adding one to both sides of the congruence, and dividing by $d$, we get 
		$$ab(abe)^*_{c}\equiv e^*_{cd}\!\!\!\! \pmod{c}.$$
		Taking into account that
		$$e^*_{cd}\equiv e^*_{c}\!\!\!\!\pmod{c},$$
		the congruence \eqref{cong_2} is proved.
		
		By applying \eqref{cong_1} and then \eqref{cong_2}, we obtain the required result.
	\end{proof}
	
	\begin{lemma}\label{L10.5}
		Let 
		$$F(x_1, x_2, x_3) = g\left(\dfrac{ax_1}{x_2 x_3}\right)\!,$$
		where $a$ is some constant and $g$ is a thrice continuously differentiable function. Then the following equalities hold:
		$$\dfrac{\partial F}{\partial x_i} = (-1)^{\mathbb{I}_{i=1}+1}F_1\left(\dfrac{ax_1}{x_2x_3}\right)\dfrac{1}{x_i},\ \ (1\leqslant i \leqslant 3),$$
		$$\dfrac{\partial^2 F}{\partial x_i\partial x_j} = (-1)^{\mathbb{I}_{i=1}}F_2\left(\dfrac{ax_1}{x_2x_3}\right)\dfrac{1}{x_ix_j},\ \ (1\leqslant i<j \leqslant 3)$$
		and
		$$\dfrac{\partial^3 F}{\partial x_1\partial x_2\partial x_3} = F_3\left(\dfrac{ax_1}{x_2x_3}\right)\dfrac{1}{x_1x_2x_3},$$
		where $F_1(z) = zg'(z)$, $F_2(z) = zg'(z) + z^2g''(z)$, and 
		$$F_3(z) = zg'(z) + 3z^2 g''(z) + z^3g'''(z).$$
	\end{lemma}
The following lemma is a multidimensional analogue of partial summation.
\begin{lemma}\label{L11} 
	Let the function $f(\vec{x}) = f(x_1, x_2, \ldots, x_m)$ be continuously differentiable with respect to each variable, and let $c(\vec{\ell}) = c(\ell_1, \ell_2, \ldots, \ell_m)$ be some multidimensional sequence. Then for the sum 
	$$S_m = \sum\limits_{\substack{K_i<\ell_i\leqslant L_i \\ 1\leqslant i\leqslant m}}c(\vec{\ell})f(\vec{\ell})$$
	the following equality holds:
	\begin{multline*}
		S_m = C(\vec{L})f(\vec{L})\\
		+\sum_{s=1}^m(-1)^s \!\!\!\!\!\sum_{1\leqslant i_1<\cdots<i_s\leqslant m}\int_{K_{i_1}}^{L_{i_1}}\cdots \int_{K_{i_s}}^{L_{i_s}}C(\vec{b}_{i_1, \ldots, i_s})\dfrac{\partial^s{f}(\vec{b}_{i_1, \ldots, i_s})}{\partial{x_{i_1}}\cdots \partial{x_{i_s}}}dx_{i_1}\cdots dx_{i_s},
	\end{multline*}
	where 
	$$\vec{L} = (L_1, \ldots, L_m),\ \ C(\vec{x}) = \sum\limits_{\substack{K_i<\ell_i\leqslant x_i \\ 1\leqslant i\leqslant m}}c(\vec{\ell}),$$
	and 
	$\vec{b}_{i_1, \ldots, i_s} = (b_1, \ldots, b_m)$, where
	$$b_k = 
	\begin{cases}
		x_k, & \text{if}\ k\in\left\{i_1, \ldots, i_s\right\},\\
		L_k, & \text{if}\ k\not\in\left\{i_1, \ldots, i_s\right\}.
	\end{cases}
	$$
\end{lemma}
\begin{proof}
	We will prove the statement by induction on the size $m$. For $m=1$, the statement of the lemma is true by partial summation. Assume the statement is true for $m=q\geqslant 1$, and we will prove it for $m = q+1$. We have
	$$S_{q+1} = \sum_{K_{q+1}<\ell_{q+1}\leqslant L_{q+1}}\sum\limits_{\substack{K_i<\ell_i\leqslant L_i \\ 1\leqslant i\leqslant q}}c(\vec{\ell}_q, \ell_{q+1})f(\vec{\ell}_q, \ell_{q+1}),$$
	where $\vec{\ell}_q = (\ell_1, \ldots, \ell_q)$.
	Applying the induction hypothesis to the inner sum and expanding the brackets, we obtain:
	\begin{multline*}
		S_{q+1} = \sum_{K_{q+1}<\ell_{q+1}\leqslant L_{q+1}}C(\vec{L}_q, \ell_{q+1})f(\vec{L}_q, \ell_{q+1})+\!\!
		\sum_{K_{q+1}<\ell_{q+1}\leqslant L_{q+1}}\sum_{s=1}^q(-1)^s\\
		\times\!\!\sum_{1\leqslant i_1<\cdots<i_s\leqslant q}\int_{K_{i_1}}^{L_{i_1}}\cdots\int_{K_{i_s}}^{L_{i_s}}\!\! C(\vec{a}_{i_1, \ldots, i_s}, \ell_{q+1})\dfrac{\partial^s{f}(\vec{a}_{i_1, \ldots, i_s}, \ell_{q+1})}{\partial{x_{i_1}}\cdots \partial{x_{i_s}}}dx_{i_1}\cdots dx_{i_s}\\
		=S_{q+1}^{(1)} + S_{q+1}^{(2)}.
	\end{multline*}
	Here $\vec{L}_q = (L_1, \ldots, L_q)$ and $\vec{a}_{i_1, \ldots, i_s} = (a_1, \ldots, a_q),$ where
	$$a_k = 
	\begin{cases}
		x_k, &  \text{if}\ k\in\left\{i_1, \ldots, i_s\right\}\!,\\
		L_k, &  \text{if}\ k\not\in\left\{i_1, \ldots, i_s\right\}\!.
	\end{cases}
	$$
	 Using the partial summation, for $S_{q+1}^{(1)}$ we get
	\begin{multline}\label{S_{q+1}^{(1)}}
		S_{q+1}^{(1)} = C(\vec{L}_q, L_{q+1})f(\vec{L}_q, L_{q+1})-\int_{K_{q+1}}^{L_{q+1}}C(\vec{L}_q,x_{q+1})\dfrac{\partial{f}(\vec{L}_q,x_{q+1})}{\partial{x_{q+1}}}dx_{q+1}\\
		= C(\vec{L})f(\vec{L}) - \int_{K_{q+1}}^{L_{q+1}}C(\vec{b}_{q+1})\dfrac{\partial{f}(\vec{b}_{q+1})}{\partial{x_{q+1}}}dx_{q+1},
	\end{multline}
	where the definition of $\vec{L}$ and the tuple $\vec{b}$ with indices is given in the statement.\\
	For the sum $ S_{q+1}^{(2)} $, we have
	$$S_{q+1}^{(2)} = \sum_{s=1}^q(-1)^s\sum_{1\leqslant i_1<\cdots<i_s\leqslant q}\int_{K_{i_1}}^{L_{i_1}}\cdots\int_{K_{i_s}}^{L_{i_s}}S_{i_1, \ldots, i_s}dx_{i_1}\cdots dx_{i_s},$$
	where
	$$S_{i_1, \ldots, i_s} = \sum_{K_{q+1}<\ell_{q+1}\leqslant L_{q+1}}C(\vec{a}_{i_1, \ldots, i_s}, \ell_{q+1})\dfrac{\partial^s{f}(\vec{a}_{i_1, \ldots, i_s}, \ell_{q+1})}{\partial{x_{i_1}}\cdots \partial{x_{i_s}}}.$$
Using partial summation for the sum $ S_{i_1, \ldots, i_s} $, we obtain:
	\begin{multline*}
		S_{i_1, \ldots, i_s} = C(\vec{a}_{i_1, \ldots, i_s}, L_{q+1})\dfrac{\partial^s{f}(\vec{a}_{i_1, \ldots, i_s}, L_{q+1})}{\partial{x_{i_1}}\cdots \partial{x_{i_s}}}\\
		-\int_{K_{q+1}}^{L_{q+1}}C(\vec{a}_{i_1, \ldots, i_s}, x_{q+1})\dfrac{\partial^{s+1}{f}(\vec{a}_{i_1, \ldots, i_s}, x_{q+1})}{\partial{x_{i_1}}\cdots \partial{x_{i_s}}\partial{x_{q+1}}}dx_{q+1}\\
		= C(\vec{b}_{i_1, \ldots, i_s})\dfrac{\partial^s{f}(\vec{b}_{i_1, \ldots, i_s})}{\partial{x_{i_1}}\cdots \partial{x_{i_s}}}\\
		-\int_{K_{q+1}}^{L_{q+1}}C(\vec{b}_{i_1, \ldots, i_s, q+1})\dfrac{\partial^{s+1}{f}(\vec{b}_{i_1, \ldots, i_s, q+1})}{\partial{x_{i_1}}\cdots \partial{x_{i_s}}\partial{x_{q+1}}}dx_{q+1}.
	\end{multline*}
	From this, we obtain that $S_{q+1}^{(2)} = S_{q+1}^{(3)} + S_{q+1}^{(4)}$, where
	\begin{multline*}
		S_{q+1}^{(3)} = \sum_{s=1}^q(-1)^s\sum_{1\leqslant i_1<\cdots<i_s\leqslant q}\int_{K_{i_1}}^{L_{i_1}}\cdots\int_{K_{i_s}}^{L_{i_s}}C(\vec{b}_{i_1, \ldots, i_s})\\
		\times\dfrac{\partial^s{f}(\vec{b}_{i_1, \ldots, i_s})}{\partial{x_{i_1}}\cdots \partial{x_{i_s}}}dx_{i_1}\cdots dx_{i_s},
	\end{multline*}
	\begin{multline*}
		S_{q+1}^{(4)} =  \sum_{s=1}^q(-1)^{s+1}\sum_{1\leqslant i_1<\cdots<i_s\leqslant q}\int_{K_{i_1}}^{L_{i_1}}\cdots\int_{K_{i_s}}^{L_{i_s}}\int_{K_{q+1}}^{L_{q+1}}C(\vec{b}_{i_1, \ldots, i_s, q+1})\\
		\times\dfrac{\partial^{s+1}{f}(\vec{b}_{i_1, \ldots, i_s, q+1})}{\partial{x_{i_1}}\cdots \partial{x_{i_s}}\partial{x_{q+1}}}dx_{i_1}\cdots dx_{i_s}dx_{q+1}.
	\end{multline*}
	By isolating the terms with $ s = 1 $ in $ S_{q+1}^{(3)} $ and making the substitution $ s = t - 1 $ in the sum $ S_{q+1}^{(4)} $, we obtain
	\begin{multline*}
		S_{q+1}  =  C(\vec{L})f(\vec{L}) - \sum_{1\leqslant i\leqslant q+1}\int_{K_{i}}^{L_{i}}C(\vec{b}_{i})\dfrac{\partial{f}(\vec{b}_{i})}{\partial{x_{i}}}dx_{i}\\
		+\sum_{s=2}^q(-1)^s\sum_{1\leqslant i_1<\cdots<i_s\leqslant q}\int_{K_{i_1}}^{L_{i_1}}\cdots\int_{K_{i_s}}^{L_{i_s}}C(\vec{b}_{i_1, \ldots, i_s})\dfrac{\partial^s{f}(\vec{b}_{i_1, \ldots, i_s})}{\partial{x_{i_1}}\cdots \partial{x_{i_s}}}dx_{i_1}\cdots dx_{i_s}\\
		+ \sum_{t = 2}^{q+1}(-1)^{t}\sum_{1\leqslant i_1<\cdots<i_{t-1}\leqslant q}\int_{K_{i_1}}^{L_{i_1}}\cdots\int_{K_{i_{t-1}}}^{L_{i_{t-1}}}\int_{K_{q+1}}^{L_{q+1}}C(\vec{b}_{i_1, \ldots, i_{t-1}, q+1})\\
		\times\dfrac{\partial^{t}{f}(\vec{b}_{i_1, \ldots, i_{t-1}, q+1})}{\partial{x_{i_1}}\cdots \partial{x_{i_{t-1}}}\partial{x_{q+1}}}dx_{i_1}\cdots dx_{i_{t-1}}dx_{q+1}.
	\end{multline*}
	By combining the last two sums into one, we obtain
	\begin{multline*}
		S_{q+1} = C(\vec{L})f(\vec{L}) - \sum_{1\leqslant i\leqslant q+1}\int_{K_{i}}^{L_{i}}C(\vec{b}_{i})\dfrac{\partial{f}(\vec{b}_{i})}{\partial{x_{i}}}dx_{i}\\
		+\sum_{s=2}^{q+1}(-1)^s\sum_{1\leqslant i_1<\cdots<i_s\leqslant q+1}\int_{K_{i_1}}^{L_{i_1}}\cdots\int_{K_{i_s}}^{L_{i_s}}C(\vec{b}_{i_1, \ldots, i_s})\\
		\times\dfrac{\partial^s{f}(\vec{b}_{i_1, \ldots, i_s})}{\partial{x_{i_1}}\cdots \partial{x_{i_s}}}dx_{i_1}\cdots dx_{i_s}.
	\end{multline*}
	Therefore, by induction, the lemma follows.
\end{proof}
We will need the following important theorem by Bettin and Chandee \cite[Theorem 1]{Bettin Chandee}, which provides an estimate for the trilinear Kloosterman sum.
\begin{theorem}\label{BettinChandee}
For the numbers $ M, N, $ and $ A $, we define
	$$\mathcal{B}(M, N, A) = \sum\limits_{\substack{a \in\mathcal{A}}}\sum_{m \in\mathcal{M}}\sum\limits_{\substack{n \in\mathcal{N} \\ (n, m) = 1}}\alpha_m\beta_n\nu_a \exp\left(2\pi i\cdot\vartheta \dfrac{a m^*_n}{n} \right)\!,$$
	where $ \nu_a, \alpha_m, $ and $ \beta_n $ are arbitrary sequences of complex numbers with supports $$\mathcal{A} = [A/2, A],\ \mathcal{M} = [M/2, M],\  \mathcal{N} = [N/2, N]$$ respectively, and $\vartheta$ is non-zero integer. Then we have
	\begin{multline*}
		\mathcal{B}(M, N, A)\ll_\varepsilon \left\|\alpha\right\|_2\left\|\beta\right\|_2\left\|\nu\right\|_2\left(1+ \dfrac{|\vartheta| A}{M N}\right)^{\frac{1}{2}}\\
		\times \left( \left(AMN\right)^{\frac{7}{20}+\varepsilon}(M+N)^{\frac{1}{4}} + \left(AMN\right)^{\frac{3}{8}+\varepsilon}(AM+AN)^{\frac{1}{8}}\right)\!,
	\end{multline*}
	where $ \left\|\cdot\right\|_2 $ denotes the $ \ell_2 $-norm of a sequence, and $ \varepsilon > 0 $ is an arbitrary fixed number.
\end{theorem}


\section{Transformation of the sum $ Q(x) $}
Let us consider an integer number $ A \geqslant 2 $. For brevity, we will write $ \mathcal{L} = \ln x $. Next, we define $ h(n) = \frac{r(n)}{4} $; then, using the equality $ h(n) = \sum_{d|n} \chi_4(d) $, we obtain
\begin{multline} Q(x) = \mathop{{\sum_{n\leqslant x}}'} \dfrac{h(n)}{h(n+1)} \\= \mathop{{\sum_{n\leqslant x}}'}\dfrac{1}{h(n+1)}\left( \sum\limits_{\substack{d|n \\ d\leqslant \sqrt{x}\mathcal{L}^{-A}}}\chi_4(d) + \sum\limits_{\substack{d|n \\ \sqrt{x}\mathcal{L}^{-A}<d\leqslant \sqrt{x}\mathcal{L}^{A}}}\chi_4(d) + \sum\limits_{\substack{d|n \\  d>\sqrt{x}\mathcal{L}^{A}}}\chi_4(d) \right)\\
	=Q_1+Q_2+Q_3.
\end{multline}
We transform each of the sums $ Q_i $. For the sum $ Q_1 $, we have
\begin{equation*}
	Q_1 = \sum_{d\leqslant \sqrt{x}\mathcal{L}^{-A}}\chi_4(d){\sum\limits_{\substack{n \leqslant x \\ n\equiv 1\,(\text{mod}\, d)}}}'\dfrac{1}{h(n)} + O(\sqrt{x}) .
\end{equation*}
Applying the orthogonality of characters, we obtain
\begin{multline}\label{Q^{MT_1}}
	Q_1 = \sum_{d\leqslant \sqrt{x}\mathcal{L}^{-A}}\dfrac{\chi_4(d)}{\varphi(d)}\sum_{\chi\,\text{mod}\, d}\,{\sum_{n\leqslant x}}'\ \dfrac{\chi(n)}{h(n)} + O(\sqrt{x})\\
	= \sum_{d\leqslant \sqrt{x}\mathcal{L}^{-A}}\dfrac{\chi_4(d)}{\varphi(d)}{\sum\limits_{\substack{n \leqslant x \\ (n, d) = 1}}}'\dfrac{1}{h(n)} + O\left( \sum_{d\leqslant \sqrt{x}\mathcal{L}^{-A}}\dfrac{1}{\varphi(d)}\sum\limits_{\substack{\chi\,\text{mod}\, d \\ \chi \neq \chi_0}}\left|{\sum_{n\leqslant x}}'\,\dfrac{\chi(n)}{h(n)} \right| \right) + O(\sqrt{x})\\
	= Q^{MT_1} + R_1 + O(\sqrt{x}),
\end{multline}
where the meanings of the notations $ Q^{MT_1} $ and $ R_1 $ are evident, and $ \chi_0 $ denotes the principal character modulo $ d $.

Now we transform the sum $ Q_2 $. We have
\begin{multline*}
	Q_2 =  \sum_{\sqrt{x}\mathcal{L}^{-A} < d \leqslant \sqrt{x}\mathcal{L}^{A}}{\chi_4(d)}{\sum\limits_{\substack{n \leqslant x \\ n\equiv 1\,(\text{mod}\, d)}}}'\dfrac{1}{h(n)} + O\left(\sqrt{x}\mathcal{L}^A \right)\\
	= Q^{MT_2} + Q^{Err_2} + O(\sqrt{x}\mathcal{L}^A),
\end{multline*}
where
\begin{equation}\label{Q^{MT_2}}
	Q^{MT_2} = \sum_{\sqrt{x}\mathcal{L}^{-A} < d \leqslant \sqrt{x}\mathcal{L}^{A}}\dfrac{\chi_4(d)}{\varphi(d)}{\sum\limits_{\substack{n \leqslant x \\ (n, d) = 1}}}'\dfrac{1}{h(n)}
\end{equation}
and
\begin{equation}\label{qerr2} Q^{Err_2} = \sum_{\sqrt{x}\mathcal{L}^{-A} < d \leqslant \sqrt{x}\mathcal{L}^{A}}{\chi_4(d)}\left( {\sum\limits_{\substack{n \leqslant x \\ n\equiv 1\,(\text{mod}\, d)}}}'\dfrac{1}{h(n)} - \dfrac{1}{\varphi(d)}{\sum\limits_{\substack{n \leqslant x \\ (n, d) = 1}}}'\dfrac{1}{h(n)} \right)\!.
\end{equation}

Finally, let us move on to the sum $ Q_3 $. Let $ y = \sqrt{x} \mathcal{L}^A $; then
\begin{multline*}
	Q_3 = 
	\mathop{{\sum_{n\leqslant x}}'}\dfrac{1}{h(n+1)}\sum\limits_{\substack{d|n \\  d<{n}/{\sqrt{x}\mathcal{L}^{A}}}}\chi_4\left( \dfrac{n}{d}\right)\\ =\sum_{d<\sqrt{x}\mathcal{L}^{-A}}\mathop{{\sum_{y<k\leqslant x/d}}'}\dfrac{\chi_4(k)}{h(kd+1)}
	=Q_{3}^{(1)}-Q_{3}^{(2)},
\end{multline*}
where $$
Q_{3}^{(1)} = \sum_{d<\sqrt{x}\mathcal{L}^{-A}}\mathop{{\sum_{k\leqslant x/d}}'}\dfrac{\chi_4(k)}{h(kd+1)}
$$
and
$$Q_{3}^{(2)} = \sum_{d<\sqrt{x}\mathcal{L}^{-A}}\mathop{{\sum_{k\leqslant y}}'}\dfrac{\chi_4(k)}{h(kd+1)}.
$$
First, we transform the sum $ Q_{3}^{(1)} $. We have
$$Q_{3}^{(1)} = \sum_{d<\sqrt{x}\mathcal{L}^{-A}}\left( {\sum\limits_{\substack{k\leqslant x/d \\ k \equiv\,1 (\text{mod}\, 4) }}}'\dfrac{1}{h(kd+1)} - {\sum\limits_{\substack{k\leqslant x/d \\ k \equiv\,3 (\text{mod}\, 4) }}}'\dfrac{1}{h(kd+1)}\right)\!.$$
By introducing the notation $ n = kd + 1 $ in the inner sums, we obtain
\begin{equation}\label{Q3}
	Q_{3}^{(1)} = \sum_{d<\sqrt{x}\mathcal{L}^{-A}}\left( {\sum\limits_{\substack{n\leqslant x \\ n \equiv\,1+d\, (\text{mod}\, 4d) }}}'\dfrac{1}{h(n)} - {\sum\limits_{\substack{n\leqslant x \\ n \equiv\,1+3d\, (\text{mod}\, 4d) }}}'\dfrac{1}{h(n)}\right) + O\left(\dfrac{\sqrt{x}}{\mathcal{L}^{A}}\right)\!.
\end{equation}

Set $\Delta_1 = (1+d, 4d)$ and $\Delta_2 = (1+3d, 4d)$. Since $$\Delta_1| 4(1+d)-4d = 4,\ \ \Delta_2 | 4(1+3d)-12d = 4,$$ it follows that $\Delta_1$ and $\Delta_2$ take one of the values: $1, 2$ or $ 4.$ Now, we note that if $ d $ is even, then $ 1 + d $ and $ 1 + 3d $ are odd, and therefore $ \Delta_1 = \Delta_2 = 1 $. On the other hand, if $ d \equiv 1 \, (\text{mod} \, 4) $, then $ (d + 1)/2 $ is odd. Hence $\Delta_1 = 2\left( (d+1)/2, 2d\right) = 2$ and
$$\Delta_2 
= 2((d+1)/2+d, 2d) = 4(((d+1)/2+d)/2, d) = 4.$$
Similarly, for $ d \equiv 3 \, (\text{mod} \, 4) $, we obtain $ \Delta_1 = 4 $ and $ \Delta_2 = 2 $.
Therefore, by splitting the sum \eqref{Q3} according to the remainder of $ d $ when divided by $ 4 $ and using the fact that $ h(k) = h(2k) $, we obtain
\begin{multline*}
	Q_{3}^{(1)} = \sum\limits_{\substack{d<\sqrt{x}\mathcal{L}^{-A}\\
			d\text{ even}}}\left( {\sum\limits_{\substack{n\leqslant x \\ n \equiv\,1+d\, (\text{mod}\, 4d) }}}'\dfrac{1}{h(n)} - {\sum\limits_{\substack{n\leqslant x \\ n \equiv\,1+3d\, (\text{mod}\, 4d) }}}'\dfrac{1}{h(n)}\right) \\
	+ \sum\limits_{\substack{d<\sqrt{x}\mathcal{L}^{-A}\\
			d \equiv\,1\,(\text{mod}\,4)}}\left( {\sum\limits_{\substack{k\leqslant x/2 \\ k \equiv\,(1+d)/2\, (\text{mod}\, 2d) }}}'\dfrac{1}{h(k)} - {\sum\limits_{\substack{k\leqslant x/4 \\ k \equiv\,(1+3d)/4\, (\text{mod}\, d) }}}'\dfrac{1}{h(k)}\right)\\
	+\sum\limits_{\substack{d<\sqrt{x}\mathcal{L}^{-A}\\
			d \equiv\,3\,(\text{mod}\,4)}}\left( {\sum\limits_{\substack{k\leqslant x/4 \\ k \equiv\,(1+d)/4\, (\text{mod}\, d) }}}'\dfrac{1}{h(k)} - {\sum\limits_{\substack{k\leqslant x/2 \\ k \equiv\,(1+3d)/2\, (\text{mod}\, 2d) }}}'\dfrac{1}{h(k)}\right)
	+ O\left(\dfrac{\sqrt{x}}{\mathcal{L}^{A}}\right)\!.
\end{multline*}
As before, we apply the orthogonality of characters and denote the contribution from the principal character by $ Q^{MT_{3, 1}} $, while the contribution from the remaining characters is denoted by $ R_3^{(1)} $. Then, for $ Q^{MT_{3, 1}} $, we have
\begin{multline}\label{Q^{MT_{3, 1}}}
	Q^{MT_{3, 1}} = \sum\limits_{\substack{d<\sqrt{x}\mathcal{L}^{-A}\\
			d \equiv\,1\,(\text{mod}\,4)}}\dfrac{1}{\varphi(2d)}{\sum\limits_{\substack{k\leqslant x/2 \\ (k, 2d) =1 }}}'\dfrac{1}{h(k)} 
	- \sum\limits_{\substack{d<\sqrt{x}\mathcal{L}^{-A}\\
			d \equiv\,3\,(\text{mod}\,4)}}\dfrac{1}{\varphi(2d)}{\sum\limits_{\substack{k\leqslant x/2 \\ (k, 2d) =1 }}}'\dfrac{1}{h(k)}\\
	-\left(  \sum\limits_{\substack{d<\sqrt{x}\mathcal{L}^{-A}\\
			d \equiv\,1\,(\text{mod}\,4)}}\dfrac{1}{\varphi(d)}{\sum\limits_{\substack{k\leqslant x/4 \\ (k, d) =1 }}}'\dfrac{1}{h(k)}\right.
	- \left.\sum\limits_{\substack{d<\sqrt{x}\mathcal{L}^{-A}\\
			d \equiv\,3\,(\text{mod}\,4)}}\dfrac{1}{\varphi(d)}{\sum\limits_{\substack{k\leqslant x/4 \\ (k, d) =1 }}}'\dfrac{1}{h(k)}\right)\\
	=\sum_{d<\sqrt{x}\mathcal{L}^{-A}}\dfrac{\chi_4(d)}{\varphi(d)}\left({\sum\limits_{\substack{k\leqslant x/2 \\ (k, 2d) =1 }}}'\dfrac{1}{h(k)} - {\sum\limits_{\substack{k\leqslant x/4 \\ (k, d) =1 }}}'\dfrac{1}{h(k)}  \right)\!. 
\end{multline}
For $R_3^{(1)}$ we find that
$$R_3^{(1)} \ll \sum_{d\leqslant 4\sqrt{x} \mathcal{L}^{-A}}\dfrac{1}{\varphi(d)}\sum\limits_{\substack{\chi\,\text{mod}\, d \\ \chi \neq \chi_0}}\left|{\sum_{n\leqslant X}}'\,\dfrac{\chi(n)}{h(n)} \right|\!,$$
where $ X $ is one of the numbers: $ x, {x}/{2} $, or $ {x}/{4} $.

Similarly, for $ Q_{3}^{(2)} $, we have
$$Q_{3}^{(2)} = Q^{MT_{3, 2}}+ R_3^{(2)},$$
where
\begin{equation}\label{Q^{MT_{3, 2}}}
	Q^{MT_{3, 2}} = \sum_{d<\sqrt{x}\mathcal{L}^{-A}}\dfrac{\chi_4(d)}{\varphi(d)}\left({\sum\limits_{\substack{k\leqslant dy/2 \\ (k, 2d) =1 }}}'\dfrac{1}{h(k)} - {\sum\limits_{\substack{k\leqslant dy/4 \\ (k, d) =1 }}}'\dfrac{1}{h(k)}  \right)
\end{equation}
and
$$R_3^{(2)} \ll \sum_{d\leqslant 4\sqrt{x} \mathcal{L}^{-A}}\dfrac{1}{\varphi(d)}\sum\limits_{\substack{\chi\,\text{mod}\, d \\ \chi \neq \chi_0}}\left|{\sum_{n\leqslant dY}}'\,\dfrac{\chi(n)}{h(n)} \right|\!,\ \ Y\in \{y, y/2, y/4\}.$$

Thus, we get
\begin{equation}\label{QFinal}
	Q(x) = Q^{MT} + Q^{Err},
\end{equation}
where
\begin{equation}\label{Main}
	Q^{MT} = Q^{MT_1} + Q^{MT_2} + Q^{MT_{3, 1}} - Q^{MT_{3, 2}},
\end{equation}
\begin{equation}\label{Error}
	Q^{Err} \ll_A 
	|Q^{Err_1}|+|Q^{Err_2}|+ \sqrt{x}\mathcal{L}^A,
\end{equation}
$$Q^{Err_1} = \sum_{d\leqslant 4\sqrt{x} \mathcal{L}^{-A}}\dfrac{1}{\varphi(d)}\sum\limits_{\substack{\chi\,\text{mod}\, d \\ \chi \neq \chi_0}}\left( \left|{\sum_{n\leqslant dY}}'\,\dfrac{\chi(n)}{h(n)} \right| + \left|{\sum_{n\leqslant X}}'\,\dfrac{\chi(n)}{h(n)} \right|\right)\!,
$$
and $Q^{Err_2}$ is defined in \eqref{qerr2}.

\section{Estimation of $ Q^{Err_1} $}\label{Korol}

In this section, using the contour integration method, we will show that
$ Q^{Err_1} \ll_B {x}{\mathcal{L}^{-B}} $, where $ B \geqslant 2 $ is an arbitrary fixed integer.

Let $ N $ be one of the numbers $ X $ or $ dY $, such that $ N \leqslant 4x $. Next, let $ Q = \sqrt{x} \mathcal{L}^{-A} $. Fix $ d \leqslant 4Q $ and a non-principal character $ \chi $ modulo $ d $. We will estimate the sum
$$S_\chi(N) = {\sum_{n\leqslant N}}'\dfrac{\chi(n)}{h(n)}.$$
To do this, we consider the Dirichlet generating series
$$\mathcal{F}(s) = {\sum_{n\geqslant 1}}'\dfrac{\chi(n)}{h(n)}n^{-s}, \ \ \Real s>1.$$
By expanding this series into an Euler product, we obtain
\begin{multline*}
	\mathcal{F}(s) = \left(1-\dfrac{\chi(2)}{2^s} \right)^{-1}\prod_{p \equiv 1 (\text{mod}\, 4)}\left(1+\dfrac{\chi(p)}{2p^s}+\dfrac{\chi^2(p)}{3p^{2s}}+\cdots \right)\\
	\times\prod_{p \equiv 3 (\text{mod}\, 4)}\left(1-\dfrac{\chi^2(p)}{p^{2s}}\right)^{-1}\!.    
\end{multline*}
Set $$f_1(x) = \left(1+\dfrac{x}{2}+\dfrac{x^2}{3}+\cdots\right)\left(1-x\right)^{\frac{1}{2}}(1-x^2)^{-\frac{1}{24}},$$
where $|x|<1$. Then, taking into account the equalities
$$(1-x)^{\frac{1}{2}} = 1 - \frac{x}{2}-\frac{x^2}{8}-\frac{x^3}{16}-\dfrac{5x^4}{128}+\cdots,$$
$$(1-x^2)^{-\frac{1}{24}} = 1 +\dfrac{x^2}{24} + \dfrac{25x^4}{1152} + \cdots,$$
by multiplying sequentially, we find
$$f_1(x) = 1 - \dfrac{x^3}{24} - \dfrac{49x^4}{2880}+ \cdots.$$
From here, we obtain
\begin{multline*}\mathcal{F}(s) = \prod_{p \equiv 1 (\text{mod}\, 4)}\left( 1 - \dfrac{\chi(p)}{p^s}\right)^{-\frac{1}{2}} \prod_{p \equiv 1 (\text{mod}\, 4)}\left( 1 -\dfrac{\chi^2(p)}{p^{2s}}\right)^{\frac{1}{24}}\\
	\times \prod_{p \equiv 1 (\text{mod}\, 4)}f_1\left( \dfrac{\chi(p)}{p^s}\right) \prod_{p \equiv 3 (\text{mod}\, 4)}\left(1-\dfrac{\chi^2(p)}{p^{2s}} \right)^{-1} \left(1-\dfrac{\chi(2)}{2^s} \right)^{-1}\!. 
\end{multline*}
Next, since
$$L(s, \chi) = \left( 1 - \dfrac{\chi(2)}{2^s}\right)^{-1}\prod_{p \equiv 1 (\text{mod}\, 4)}\left( 1 - \dfrac{\chi(p)}{p^s}\right)^{-1}\prod_{p \equiv 3 (\text{mod}\, 4)}\left( 1 - \dfrac{\chi(p)}{p^s}\right)^{-1}$$
and
$$L(s, \chi\chi_4) = \prod_{p \equiv 1 (\text{mod}\, 4)}\left( 1 - \dfrac{\chi(p)}{p^s}\right)^{-1}\prod_{p \equiv 3 (\text{mod}\, 4)}\left( 1 + \dfrac{\chi(p)}{p^s}\right)^{-1},$$
it follows that
\begin{multline*}\prod_{p \equiv 1 (\text{mod}\, 4)}\left( 1 - \dfrac{\chi(p)}{p^s}\right)^{-\frac{1}{2}}\\
	= \sqrt[4]{\left(1-\dfrac{\chi(2)}{2^s} \right)L(s, \chi)L(s, \chi\chi_4) \prod_{p \equiv 3 (\text{mod}\, 4)}\left( 1 - \dfrac{\chi^2(p)}{p^{2s}}\right)}.
\end{multline*}
From which we find
$$\mathcal{F}(s) = \sqrt[4]{L(s, \chi)L(s, \chi\chi_4)}\mathcal{F}_1(s)\mathcal{F}_2(s),$$
where
$$\mathcal{F}_1(s) = \left(\left( 1 - \dfrac{\chi(2)}{2^s}\right)\prod_{p \equiv 3 (\text{mod}\, 4)}\left( 1 - \dfrac{\chi^2(p)}{p^{2s}}\right) \right)^{-\frac{3}{4}}\prod_{p \equiv 1 (\text{mod}\, 4)}\left( 1 -\dfrac{\chi^2(p)}{p^{2s}}\right)^{\frac{1}{24}}$$
and
$$
\mathcal{F}_2(s) = \prod_{p \equiv 1 (\text{mod}\, 4)}f_1\left( \dfrac{\chi(p)}{p^s}\right)\!.
$$
Let $ T = T(d, \chi) \geqslant 10 $ be a certain number, different from the ordinates of the zeros of the functions $ L(s, \chi) $ and $ L(s, \chi\chi_4) $, the exact value of which will be chosen later. We also set $ N_1 = \lfloor N \rfloor + \frac{1}{2} $, $ \sigma_1 = \frac{1}{2} + \frac{1}{\mathcal{L}} $, and $ \sigma_2 = 2\sigma_1 $. Then, Perron's formula gives:
\begin{equation}\label{S_chi}
	S_\chi(N) = \dfrac{1}{2\pi i}\int_{\sigma_2-iT}^{\sigma_2+iT}\mathcal{F}(s)\dfrac{N_1^s}{s}ds + O\left(\dfrac{N\mathcal{L}}{T}\right)\!.
\end{equation}
Let $ \rho = \beta + i\gamma $ denote a non-trivial zero of $ L(s, \chi)L(s, \chi\chi_4) $. Consider the contour $ \Gamma $, which is a rectangle with vertices at $ \sigma_1 \pm iT $ and $ \sigma_2 \pm iT $, in which a horizontal cut is made from the side $ \Real s = \sigma_1 $ to each zero $ \rho $ under the condition that $ |\gamma| \leqslant T $. 
Using Cauchy's residue theorem, we find
$$I = -(I_1+ I_2+I_3)-\sum_{|\gamma|\leqslant T}I(\rho),$$
where
$$I = \dfrac{1}{2\pi i}\int_{\sigma_2-iT}^{\sigma_2+iT}\mathcal{G}(s)ds,\ \ \mathcal{G}(s) = \mathcal{F}(s)\dfrac{N_1^s}{s},$$
$$I_1 = -\dfrac{1}{2\pi i}\int_{\sigma_1+iT}^{\sigma_2+iT}\mathcal{G}(s)ds,\ \ I_2 = \dfrac{1}{2\pi i}\int_{\sigma_1-iT}^{\sigma_2-iT}\mathcal{G}(s)ds,$$
$$I_3 = - \text{p.v.} \dfrac{1}{2\pi i}\int_{\sigma_1-iT}^{\sigma_1+iT}\mathcal{G}(s)ds$$
and
$$I(\rho) = \dfrac{1}{2\pi i}\left( \int_{\sigma_1+i\gamma+i0}^{\beta+i\gamma+i0} + \int_{\beta+i\gamma-i0}^{\sigma_1+i\gamma-i0}\right) \mathcal{G}(s)ds. $$
We will estimate the function $ \mathcal{F}(s) $ along the contour $ \Gamma $. Since
$$\left|1-\dfrac{\chi^2(p)}{p^{2s}}\right|^{-1} = \left|1+\dfrac{\chi^2(p)}{p^{2s}}+\cdots\right|\leqslant \dfrac{1}{1-p^{-2\sigma}},$$
it follows that for $\sigma\geqslant \sigma_1$, we have
$$|\mathcal{F}_1(s)|\leqslant \prod_{p}\left(1-\dfrac{1}{p^{2\sigma}} \right)^{-3/4} \ll \mathcal{L}^{3/4}.$$
From the easily verifiable inequality $ \mathcal{F}_2(s) \ll 1 $, we obtain
$$\mathcal{F}(s) \ll |L(s, \chi)L(s, \chi\chi_4)|^{1/4} \mathcal{L}^{3/4}\ll \left( |L(s, \chi)|^{1/2}+|L(s, \chi\chi_4)|^{1/2}\right)  \mathcal{L}^{3/4}. $$
Now we will estimate each of the integrals $ I_1, I_2, I_3 $, and $ I(\rho) $. For $ I_1 $ and $ I_2 $, we have
\begin{multline*}
	\max\left(|I_1|, |I_2|\right)\ll \int_{\sigma_1}^{\sigma_2}\left( |\mathcal{F}(\sigma+iT)| + |\mathcal{F}(\sigma-iT)|\right) \dfrac{N_1^{\sigma}}{\sqrt{\sigma^2+T^2}}d\sigma\\
	\ll \dfrac{x \mathcal{L}^{3/4}}{T}\left( \int_{\sigma_1}^{\sigma_2}|L(\sigma+iT, \chi)|^{1/2}d\sigma + \int_{\sigma_1}^{\sigma_2}|L(\sigma+iT, \chi\chi_4)|^{1/2}d\sigma\right.\\
	\left. +  \int_{\sigma_1}^{\sigma_2}|L(\sigma+iT, \overline{\chi})|^{1/2}d\sigma + \int_{\sigma_1}^{\sigma_2}|L(\sigma+iT, \overline{\chi}\chi_4)|^{1/2}d\sigma\right)\!. 
\end{multline*}
For the integral $ I_3 $, we obtain
\begin{multline*}
	|I_3|\ll N^{\sigma_1}\int_{-T}^{T}\dfrac{|\mathcal{F}(\sigma_1+it)|}{\sqrt{\sigma^2+t^2}}dt \\
	\ll \sqrt{x}\mathcal{L}^{3/4}\left(\int_{-T}^T\dfrac{|L(\sigma_1+it,\chi)|^{1/2}}{\sqrt{\sigma_1^2+t^2}}dt + \int_{-T}^T\dfrac{|L(\sigma_1+it,\chi\chi_4)|^{1/2}}{\sqrt{\sigma_1^2+t^2}}dt \right)\\
	\ll \sqrt{x}\mathcal{L}^{3/4} \left(\int_{0}^T\left( {|L(\sigma_1+it,\chi)|^{1/2}}+{|L(\sigma_1+it,\overline{\chi})|^{1/2}}\right) \dfrac{dt}{t+1}\right. \\
	+ \left.\int_{0}^T\left( {|L(\sigma_1+it,\chi\chi_4)|^{1/2}}+{|L(\sigma_1+it,\overline{\chi}\chi_4)|^{1/2}}\right) \dfrac{dt}{t+1}\right)\!.
\end{multline*}
Finally, for the integral $ I(\rho) $, we have
$$|I(\rho)| \ll \dfrac{\mathcal{L}^{3/4}}{|\gamma|+1}\int_{\sigma_1}^{\beta}\left(|L(\sigma+i\gamma, \chi)|^{1/2} + |L(\sigma+i\gamma, \chi\chi_4)|^{1/2} \right) x^{\sigma} d\sigma.$$
Thus, for $ Q^{Err_1} $, we obtain the estimate
\begin{equation*}
	Q^{Err_1} = \sum_{d\leqslant 4Q}\dfrac{1}{\varphi(d)}\sum_{\chi \neq \chi_0}|S_\chi(N)|\ll R_1+R_2+R_3+R_4,
\end{equation*}
where $ R_1 $ denotes the contribution to $ R $ from the remainder in the formula \eqref{S_chi}, $ R_2 $ and $ R_3 $ are the contributions from $ \max\left(|I_1|, |I_2|\right) $ and $ I_3 $, respectively, and $ R_4 $ is the contribution from the sum $ \sum_{|\gamma| \leqslant T} I(\rho) $. 

Since $N\leqslant x$, we get
$$R_1 \ll x\mathcal{L} \sum_{d\leqslant 4Q}\dfrac{1}{\varphi(d)}\sum\limits_{\substack{\chi\,\text{mod}\, d \\ \chi \neq \chi_0}}\dfrac{1}{T} = x\mathcal{L} S_1(4Q),$$
where the meaning of the notation $ S_1(4Q) $ is obvious.
Using Lemma \ref{L1} and the fact that $ \chi $ and $ \overline{\chi} $ simultaneously run through all characters modulo $ d $, we find
\begin{multline*}R_2 \ll x\mathcal{L}^{3/4}\sum_{d\leqslant 4Q}\dfrac{1}{\varphi(d)}\sum\limits_{\substack{\chi\,\text{mod}\, d \\ \chi \neq \chi_0}}\dfrac{1}{T}\left(\int_{\sigma_1}^{\sigma_2}|L(\sigma+iT, \chi)|^{1/2}d\sigma\right.\\
	+ \left. \int_{\sigma_1}^{\sigma_2}|L(\sigma+iT, \chi\chi_4)|^{1/2}d\sigma\right)
	\ll x\mathcal{L}^{3/4} S_2(16Q),
\end{multline*}
where
$$ S_2(16Q) = \sum_{d\leqslant 16Q}\dfrac{1}{\varphi(d)}\sum\limits_{\substack{\chi\,\text{mod}\, d \\ \chi \neq \chi_0}}\dfrac{1}{T}\left(\int_{\sigma_1}^{\sigma_2}|L(\sigma+iT, \chi)|^{1/2}d\sigma\right)\!. 
$$
Similarly, we obtain
\begin{multline*}
	R_3 \ll \sqrt{x} \mathcal{L}^{3/4}\sum_{d\leqslant 4Q}\dfrac{1}{\varphi(d)}\sum\limits_{\substack{\chi\,\text{mod}\, d \\ \chi \neq \chi_0}} \left(\int_{0}^T {|L(\sigma_1+it,\chi)|^{1/2}}\dfrac{dt}{t+1}\right.\\
	+\left.\int_{0}^T {|L(\sigma_1+it,\chi\chi_4)|^{1/2}}\dfrac{dt}{t+1} \right)\ll\sqrt{x} \mathcal{L}^{3/4} S_3(16Q),
\end{multline*}
where
$$S_3(16Q) = \sum_{d\leqslant 16Q}\dfrac{1}{\varphi(d)}\sum\limits_{\substack{\chi\,\text{mod}\, d \\ \chi \neq \chi_0}} \int_{0}^T {|L(\sigma_1+it,\chi)|^{1/2}}\dfrac{dt}{t+1}.$$
Finally, we have
\begin{equation*}
	R_4 \ll \mathcal{L}^{3/4}\sum_{d \leqslant 4Q}\dfrac{1}{\varphi(d)}\sum\limits_{\substack{\chi\,\text{mod}\, d \\ \chi \neq \chi_0}}\sum_{|\gamma|\leqslant T}\dfrac{j(\gamma)}{|\gamma|+1},
\end{equation*}
$$j(\gamma) = \int_{\sigma_1}^{\beta}\left( |L(\sigma+i\gamma, \chi) |^{1/2}+|L(\sigma+i\gamma, \chi\chi_4)|^{1/2}\right) x^\sigma d\sigma,$$
where the summation in the inner sum is taken over the non-trivial zeros $ \rho = \beta + i\gamma $ of the function $ L(s, \chi)L(s, \chi\chi_4) $, satisfying the conditions

\begin{equation}\label{cond_zero}
	\sigma_1<\beta\leqslant 1,\ |\gamma|\leqslant T.
\end{equation}

Now let us choose the value $ T = T(d, \chi) $. Suppose the character $ \chi $ modulo $ d $ is induced by a primitive character $ \chi_1 $ modulo $ d_1 $. In this case, for some $ r \geqslant 1 $, we have $ d = d_1 r $ and the equality 
$$
\chi = \chi_1\chi_0,
$$
holds, where $ \chi_0 $ is the principal character modulo $ d $. 
Since for any $ t \leqslant x $ on the interval $ [t, t+1] $ there are at most $ O(\mathcal{L}) $ ordinates of the zeros of $ L(s, \chi)L(s, \chi\chi_4) $, there exists $ h = h(t) $ such that $ 0 \leqslant h \leqslant 1 $ and $ t + h $ does not coincide with any of the ordinates of the zeros. We will divide $ (1, 4Q] $ into intervals of the form $ \left(M, M_1\right] $, where $ M = \max\left(1, \frac{4Q}{2^{\nu+1}} \right) $, $ M_1 = \frac{4Q}{2^\nu} $, and $ \nu \geqslant 0 $. From this, we have $ M_1 \leqslant 2M $. If for some $ M $ it holds that $ M < d_1 \leqslant M_1 $, we set $ T(d, \chi) = M \mathcal{L}^C + h $, where $ h = h(M\mathcal{L}^C) $ is chosen as above, and $ C < A $ will be chosen later. It is clear that with this choice, $ T $ depends only on $ d_1 $ and $ \chi_1 $, and moreover, $ M \mathcal{L}^C \leqslant T \leqslant 2M \mathcal{L}^C $. Now we will estimate each of the sums $ R_j $.



In the work \cite{Korol2010}, during the proof of Lemma 13, the following estimates were obtained
$$S_1(Q) \ll \mathcal{L}^{2-C},\ \  S_2(Q) \ll \mathcal{L}^{11/4-3C/4}(\ln \mathcal{L})^2,\ \ S_3(Q)\ll Q\mathcal{L}^{(C+11)/4}(\ln \mathcal{L})^2.$$
Therefore the same estimates hold for the sums $$S_1(4Q),\  S_2(16Q),\  S_3(16Q).$$
Thus,
\begin{equation}\label{R123}
R_1 \ll x \mathcal{L}^{3-C},\ \ R_2 \ll x \mathcal{L}^{7/2-3C/4}(\ln \mathcal{L})^2,\ \ R_3 \ll x \mathcal{L}^{7/2 + C/4-A} (\ln \mathcal{L})^2.
\end{equation}
Consider the sum $ R_4 $. We set
$$S_4(Q) = \sum_{d \leqslant Q}\dfrac{1}{\varphi(d)}\sum\limits_{\substack{\chi\,\text{mod}\, d \\ \chi \neq \chi_0}}\sum_{|\gamma|\leqslant T}\dfrac{1}{|\gamma|+1} \int_{\sigma_1}^{\beta} |L(\sigma+i\gamma, \chi) |^{1/2} x^\sigma d\sigma,$$
where the summation in the inner sum is taken over all non-trivial zeros $ \rho = \beta + i\gamma $ of the function $ L(s, \chi) $ under the condition \eqref{cond_zero}. In \cite{Korol2010}, during the proof of Lemma 13, an estimate was obtained
$$S_4(Q)\ll x\mathcal{L}^{(C-A)/4+15}(\ln \mathcal{L})^2.$$
We can obtain the same result for $R_4$ by repeating the calculations verbatim. The only difference is that we will apply Lemma 6 from \cite{Korol2010} along with Lemma \ref{L3} from the present work. We have
\begin{equation}\label{R_4}
R_4 \ll \mathcal{L}^{3/4} \cdot x\mathcal{L}^{(C-A)/4+15}(\ln \mathcal{L})^2 = x\mathcal{L}^{(C-A)/4+63/4}(\ln \mathcal{L})^2.
\end{equation}
For an arbitrary $ B \geqslant 2 $, let us take
\begin{equation}\label{CA}
C= 2B+5,\ \ A = 6B+70;
\end{equation} 
from this, given the estimates \eqref{R123} and \eqref{R_4}, we obtain
$$Q^{Err_1}\ll_B x \mathcal{L}^{-B}.$$
The claim follows. 
\section{Transformation of  $Q^{Err_2}$}\label{Fouvry-Radziwill}
Let $0<\varepsilon< 1/10000$ and $x \geqslant x_0(\varepsilon)$. Set
$\mathcal{J} =  \left(\right.\!\!\exp\left(\mathcal{L}^\varepsilon\right), x^\varepsilon\!\!\left.\right].$ Let $\mathcal{A}$ denote the set of all integers $n \leqslant x$ that are divisible by at least one prime $p \in \mathcal{J}$, and let $\mathcal{B}$ be its complement,
$$\mathcal{B} = \left\{n\leqslant x: p|n\Rightarrow p\not\in\mathcal{J} \right\}\!.$$ 
In the sum $Q^{Err_2}$, we partition the interval $\left(\right.\!\!\sqrt{x}\mathcal{L}^{-A}, \sqrt{x}\mathcal{L}^{A}\!\!\left.\right]$ into dyadic intervals $(D, D_1]$, where
$$D = \max\left( \sqrt{x}\mathcal{L}^{-A}, \dfrac{1}{2^{\nu+1}}\sqrt{x}\mathcal{L}^{A}\right),\  D_1 = \dfrac{1}{2^{\nu}}\sqrt{x}\mathcal{L}^{A},\ \  \nu \geqslant 0.$$
For brevity, we write $d\sim D$ if $D<d\leqslant D_1$. Summation over the left endpoints of these intervals will be denoted by double primes. In this case, the sum $Q^{Err_2}$ can be written as follows
$$Q^{Err_2} = {\sum_{D}}''\sum_{d\sim D}\chi_4(d)\left( {\sum\limits_{\substack{n \leqslant x \\ n\equiv 1\,(\text{mod}\, d)}}}'\dfrac{1}{h(n)} - \dfrac{1}{\varphi(d)}{\sum\limits_{\substack{n \leqslant x \\ (n, d) = 1}}}'\dfrac{1}{h(n)} \right)\!.$$
We will show that the contribution of those $n$ for which $n\in \mathcal{B}$ can be neglected. Indeed, using Shiu's theorem and the fact that the condition $p\mid n \Rightarrow p \notin \mathcal{J}$ is multiplicative, we find:
\begin{multline*}
{\sum\limits_{\substack{n \in \mathcal{B} \\ n\equiv 1\!\!\!\!\!\pmod{d}}}}'\dfrac{1}{h(n)} = {\sum\limits_{\substack{n \leqslant x \\ n\equiv 1\!\!\!\!\!\pmod{d} \\ p|n \Rightarrow p \not\in \mathcal{J} }}}'\dfrac{1}{h(n)}\ll \dfrac{x}{\varphi(d)\mathcal{L}}\exp\left({\sum\limits_{\substack{p \leqslant x \\ p\not\in\mathcal{J}}}}'\dfrac{1}{ph(p)} \right)\\
\ll \dfrac{x}{\varphi(d)\mathcal{L}}\exp\left( {\sum\limits_{\substack{p \leqslant\exp\left(\mathcal{L}^\varepsilon\right)  \\ p\equiv 1\!\!\!\!\!\pmod{4}}}}\dfrac{1}{2p} + {\sum\limits_{\substack{x^\varepsilon<p\leqslant x \\ p\equiv 1\!\!\!\!\!\pmod{4}}}}\dfrac{1}{2p}\right)\\
\ll \dfrac{x}{\varphi(d)\mathcal{L}}\exp\left( \dfrac{1}{4}\varepsilon\ln\mathcal{L} +O(\varepsilon)\right)\ll_\varepsilon \dfrac{x}{\varphi(d)\mathcal{L}^{1-\frac{\varepsilon}{4}}}. 
\end{multline*}
In the same way, from the Richert–Halberstam inequality we obtain the estimate
$${\sum\limits_{\substack{n \in \mathcal{B} \\ (n, d) = 1}}}'\dfrac{1}{h(n)}\ll \dfrac{x}{\mathcal{L}^{1-\frac{\varepsilon}{4}}}.$$
From this, we obtain that the contribution of $n\in\mathcal{B}$ to the sum $Q^{Err_2}$ does not exceed
$$O_\varepsilon\left( {\sum_{D}}''\sum_{d\sim D}\dfrac{x}{\varphi(d)\mathcal{L}^{1-\frac{\varepsilon}{4}}}\right) =O_\varepsilon\left(  \dfrac{x\ln\mathcal{L}}{\mathcal{L}^{1-\frac{\varepsilon}{4}}}\right)\!. $$
Let $\mathcal{A}^*$ denote the set of those $n \in \mathcal{A}$ for which each prime divisor $p\in\mathcal{J}$ appears to the first power:
$$\mathcal{A}^* = \left\{n\in \mathcal{A}: p|n,\, p\in\mathcal{J} \Rightarrow p^2 \nmid n\right\}.$$
Let us estimate the contribution of $n\in\mathcal{A}\setminus\mathcal{A}^*$ to the sum $Q^{Err_2}$. Note that if $n\in\mathcal{A}\setminus\mathcal{A}^*$, then $n$ has the form $n = p^2 m$, where $p \in \mathcal{J}$. In the subsequent estimates, the quantity $p^2$ will appear in the denominator. Taking into account that $p$ is large, we obtain that the contribution of $n\in\mathcal{A}\setminus\mathcal{A}^*$ will be small.

First, let us show that if $h(n)\neq 0$, then the inequality $h(n)\geqslant h(m)$ holds. Indeed, let $n = p^\nu m'$, where $\nu\geq 2$ and $(m', p)=1$. Since $x\geqslant x_0(\varepsilon)$, we have $p>2$, and therefore either $p\equiv 1 \pmod{4}$ or $p\equiv 3\pmod{4}$. Taking into account the equality 
\begin{equation}\label{h(p^nu)}
h(p^\nu)=
\begin{cases}
\nu + 1, & \text{if}\ \ p\equiv 1\!\!\!\!\!\pmod{4},\\
\mathbb{I}_{2|\nu}, & \text{if}\ \ p\equiv 3\!\!\!\!\!\pmod{4},
\end{cases}
\end{equation}
in the first case we find
$$h(n) = h(p^\nu m') = (\nu+1)h(m')>(\nu-1)h(m')=h(p^{\nu-2}m')=h(m).$$
Let us consider the second case. Since $h(n)\neq 0$, it follows that $\nu$ is even. Hence, we have $h(p^\nu) = 1$ and
$$h(n) = h(p^\nu)h(m') = h(p^{\nu-2})h(m') = h(m).$$ The required inequality is proved. From this, using Shiu's theorem, we obtain
\begin{multline*}
{\sum\limits_{\substack{n\in \mathcal{A}\setminus\mathcal{A}^* \\ n\equiv 1\!\!\!\!\!\pmod{d}}}}'\dfrac{1}{h(n)}\leqslant \underset{mp^2\equiv 1\!\!\!\!\!\pmod{d}} {\underset{mp^2\leqslant x,\, p\in\mathcal{J}}{{\sum\sum}'}}\dfrac{1}{h(mp^2)}\\
\leqslant \sum_{p\in\mathcal{J}}{\sum\limits_{\substack{m\leqslant x/p^2 \\ m\equiv p^{-2}\!\!\!\!\!\pmod{d}}}}'\dfrac{1}{h(m)}\ll\sum_{p\in\mathcal{J}}\dfrac{x}{p^2\varphi(d)\ln\left({x}/{p^2}\right)}\exp\left({\sum\limits_{\substack{q\leqslant x}}}'\dfrac{1}{qh(q)} \right)\!.
\end{multline*}
Since $${\sum\limits_{q\leqslant x}}'\dfrac{1}{qh(q)} = \dfrac{1}{2}+\dfrac{1}{2}\!\!\!\!\!\sum\limits_{\substack{q\leqslant x \\ q\equiv 1\!\!\!\!\!\pmod{4}}}\dfrac{1}{q} = \dfrac{1}{4}\ln \mathcal{L} + O(1),$$
we get
\begin{equation}
\label{great_estimate}
{\sum\limits_{\substack{n\in \mathcal{A}\setminus\mathcal{A}^* \\ n\equiv 1\!\!\!\!\!\pmod{d}}}}'\dfrac{1}{h(n)}\ll\dfrac{x}{\varphi(d)\mathcal{L}}\cdot\mathcal{L}^{1/4}\sum_{p>\exp\left(\mathcal{L}^\varepsilon\right) }\dfrac{1}{p^2}\ll \dfrac{x}{\varphi(d)\mathcal{L}^{3/4}\exp\left(\mathcal{L}^\varepsilon\right)}.
\end{equation}
In the same way, using the Richert~--~Halberstam inequality, we obtain
\begin{equation}\label{second_estimate}
{\sum\limits_{\substack{n\in \mathcal{A}\setminus\mathcal{A}^* \\ (n,d)=1}}}'\dfrac{1}{h(n)}\leqslant {\underset{mp^2\leqslant x,\, p\in\mathcal{J}}{{\sum\sum}'}}\dfrac{1}{h(mp^2)} \ll \dfrac{x}{\mathcal{L}^{3/4}\exp\left(\mathcal{L}^\varepsilon\right)}.
\end{equation}
From this, the desired contribution of $n\in\mathcal{A}\setminus\mathcal{A}^*$ can be estimated as follows
$${\sum_{D}}''\sum_{d\sim D}\dfrac{x}{\varphi(d)\mathcal{L}^{3/4}\exp\left(\mathcal{L}^\varepsilon\right)}\ll \dfrac{x\ln\mathcal{L}}{\mathcal{L}^{3/4}\exp\left(\mathcal{L}^\varepsilon\right)}\ll \dfrac{x\ln\mathcal{L}}{\mathcal{L}^{1-\frac{\varepsilon}{4}}}.$$
Thus, we obtain
\begin{equation*}Q^{Err_2} = {\sum_{D}}''\sum_{d\sim D}\chi_4(d)\left( {\sum\limits_{\substack{n \in \mathcal{A}^* \\ n\equiv 1\,(\text{mod}\, d)}}}'\dfrac{1}{h(n)} - \dfrac{1}{\varphi(d)}{\sum\limits_{\substack{n \in \mathcal{A}^* \\ (n, d) = 1}}}'\dfrac{1}{h(n)} \right) + R_1,
\end{equation*}
where
$$R_1 \ll_\varepsilon  \dfrac{x\ln\mathcal{L}}{\mathcal{L}^{1-\frac{\varepsilon}{4}}}.$$
Let $\mathcal{A}_k$, where $k\geq 1$, denote the set of all elements from $\mathcal{A}^*$ having exactly $k$ prime divisors from $\mathcal{J}$. Also, let $\mathcal{A}_0 = \mathcal{B}$.

For each $n\in\mathcal{A}_k$, there are exactly $k$ representations of the form
$$
n = pm,
$$
where $p\in\mathcal{J}$, $m\in\mathcal{A}_{k-1}$, $m\leqslant x/p$, and $p\nmid m$. Since every prime $p$ dividing $n\in\mathcal{A}_k$ appears in $n$ to the first power, we have $p\equiv 1 \pmod{4}$ and
$$
h(n) = h(mp) = 2h(m).
$$
Since for every prime $p\in\mathcal{J}$ the inequality
$$
p > \exp\left(\mathcal{L}^\varepsilon\right)
$$
holds, the largest possible value of $k$ is less than $\mathcal{L}^{1-\varepsilon}$. Thus, we have
\begin{multline}\label{some_link}
{\sum\limits_{\substack{n \in \mathcal{A}^* \\ n\equiv 1\,(\text{mod}\, d)}}}'\dfrac{1}{h(n)} = \sum_{1\leqslant k< \mathcal{L}^{1-\varepsilon}}\dfrac{1}{k}\sum\limits_{\substack{p\in\J \\ p\equiv 1\!\!\!\!\!\pmod{4}}}{\sum\limits_{\substack{m\leqslant x/p \\ m\in\A_{k-1} \\ mp\equiv 1\!\!\!\!\!\pmod{d} \\ (m, p)=1}}}'\dfrac{1}{h(mp)}\\
=  \sum_{1\leqslant k< \mathcal{L}^{1-\varepsilon}}\dfrac{1}{2k}\sum\limits_{\substack{p\in\J \\ p\equiv 1\!\!\!\!\!\pmod{4} \\ p\nmid d}}{\sum\limits_{\substack{m\leqslant x/p \\ m\in\A_{k-1}, \\ m\equiv p^{-1}\!\!\!\!\!\pmod{d} \\ (m, p) = 1}}}'\dfrac{1}{h(m)}.
\end{multline}
Similarly, we find that
$${\sum\limits_{\substack{n \in \mathcal{A}^* \\ (n, d) =1}}}'\dfrac{1}{h(n)} = \sum_{1\leqslant k< \mathcal{L}^{1-\varepsilon}}\dfrac{1}{2k}\sum\limits_{\substack{p\in\J \\ p\equiv 1\!\!\!\!\!\pmod{4} \\ p\nmid d}}{\sum\limits_{\substack{m\leqslant x/p \\ m\in\A_{k-1} \\ (m, dp) = 1}}}'\dfrac{1}{h(m)}.$$
By virtue of the estimates \eqref{great_estimate} and \eqref{second_estimate}, omitting the condition $(m, p) = 1$ in both sums when estimating $Q^{Err_2}$ leads to an error of order
$$
O\left(x(\ln\LL)^2\LL^{-3/4}\exp\left(-\LL^\varepsilon\right)\right).
$$
Indeed, if $p \mid m$, then $m = pm'$ for some $m'$. Thus, including such $p$, for example, in the sum \eqref{some_link} when estimating $Q^{Err_2}$ leads to an error not exceeding (in order)
\begin{multline*}
{\sum_{{D}}}''\sum_{d\sim {D}}\sum_{k< \mathcal{L}^{1-\varepsilon}}\dfrac{1}{k}\sum\limits_{\substack{p \in \mathcal{J} \\ p\nmid d}}\sum\limits_{\substack{m'\leqslant x/p^2 \\ m'\equiv p^{-2}\!\!\!\!\!\pmod{d}}}\dfrac{1}{h(m')}\\
\ll {\sum_{{D}}}''\sum_{d\sim {D}} \dfrac{x\ln \mathcal{L}}{\varphi(d)\mathcal{L}^{3/4}\exp\left(\mathcal{L}^\varepsilon\right)}\ll \dfrac{x(\ln \mathcal{L})^2}{\mathcal{L}^{3/4}\exp\left(\mathcal{L}^\varepsilon\right)}.
\end{multline*}
Since $m \leqslant x/p$ and $p \geqslant \exp\left(\mathcal{L}^{\varepsilon}\right)$, we have $m \leqslant x\exp\left(\mathcal{L}^{-\varepsilon}\right)$. We partition the summation ranges for $p$ (the interval $\mathcal{J}$) and $m$ (the segment $[1,\, x\exp\left(\mathcal{L}^{-\varepsilon}\right)]$) into dyadic intervals $(N, N_1]$ and $(M, M_1]$, as was done above for the variable $d$.

Denoting summation over the left endpoints of such intervals by double primes, we obtain
\begin{multline}\label{QErr2_}
Q^{Err_2} = \dfrac{1}{2}\sum_{k<\LL^{1-\varepsilon}}\dfrac{1}{k}{\sum\limits_{\substack{D, N, M}}}''\sum_{d\sim D}\chi_4(d)\sum\limits_{\substack{p \sim N \\ p\equiv 1\!\!\!\!\!\pmod{4} \\ p\nmid d}}H(M; k, d, p)
+ R_2,
\end{multline}
where
$$H(M; k, d, p) = {\sum\limits_{\substack{m \sim M \\ mp\leqslant x\\  m\in\A_{k-1} \\m\equiv p^{-1}\!\!\!\!\!\pmod{d} }}}'\dfrac{1}{h(m)}-\dfrac{1}{\varphi(d)}{\sum\limits_{\substack{m \sim M \\ mp\leqslant x\\ m\in\A_{k-1}\\ (m, d) = 1}}}'\dfrac{1}{h(m)}$$
and
$$R_2 \ll_\varepsilon  \dfrac{x\ln\mathcal{L}}{\mathcal{L}^{1-\frac{\varepsilon}{4}}}.$$
\section{Preparation for the dispersion method}
Let $A_1 \geq A+3$ be a sufficiently large number, which will be chosen later. In the sum \eqref{QErr2_}, we retain only those terms for which
$$
x\LL^{-A_1} < NM \leqslant x.
$$
The inequality $NM > x$ is false, since otherwise
$$
x \geqslant pm \geqslant NM > x.
$$
Let us estimate the contribution from the terms for which $NM \leqslant x\LL^{-A_1}$. We have
\begin{multline}\label{link1}
\sum\limits_{\substack{p \sim N \\ p\equiv 1\!\!\!\!\!\pmod{4} \\ p\nmid d}}H(M; k, d, p) \ll \sum\limits_{\substack{p \sim N \\ p\equiv 1\!\!\!\!\!\pmod{4} \\ p\nmid d}}{\sum\limits_{\substack{m \sim M \\  m\in\A_{k-1} \\m\equiv p^{-1}\!\!\!\!\!\pmod{d}}}} 1+\sum\limits_{\substack{p \sim N \\ p\equiv 1\!\!\!\!\!\pmod{4} \\ p\nmid d}}\dfrac{1}{\varphi(d)}{\sum\limits_{\substack{m \sim M \\  m\in\A_{k-1}\\ (m, d) = 1}}}1\\
\leqslant k\sum\limits_{\substack{n\leqslant 4NM \\ n \equiv 1\!\!\!\!\!\pmod{d}}}1+\dfrac{NM}{\varphi(d)}\\
\ll k\left(\dfrac{NM}{d}+1 \right)+ \dfrac{NM}{\varphi(d)}
\ll \dfrac{kNM}{\varphi(d)}+k.
\end{multline}
Thus, the desired contribution to the sum $Q^{Err_2}$ does not exceed
\begin{equation}\label{link2}
\sum_{k<\LL^{1-\varepsilon}}\dfrac{1}{k}{\sum_{D, N, M}}''\sum_{d\sim D}\left(\dfrac{kNM}{\varphi(d)}+k\right)\ll x\LL^{3-A_1}+\sqrt{x}\LL^{A+3}\ll_A x\LL^{3-A_1}.
\end{equation}
So we get
\begin{equation*}
Q^{Err_2} = \dfrac{1}{2}\sum_{k<\LL^{1-\varepsilon}}\dfrac{1}{k}{\sum\limits_{\substack{D, N, M \\ x\mathcal{L}^{-A_1}<NM\leqslant x}}}''\sum_{d\sim D}\chi_4(d)\sum\limits_{\substack{p \sim N \\ p\equiv 1\!\!\!\!\!\pmod{4} \\ p\nmid d}}H(M; k, d, p) + R_3,
\end{equation*}
where
$$R_3 \ll_\varepsilon \dfrac{x\ln\mathcal{L}}{\mathcal{L}^{1-\frac{\varepsilon}{4}}}+\dfrac{x}{\LL^{A_1-3}} \ll_{\varepsilon, A_1} \dfrac{x\ln\mathcal{L}}{\mathcal{L}^{1-\frac{\varepsilon}{4}}}.$$
Put
$$\alpha(m)=
\begin{cases}
\dfrac{1}{h(m)}, & \text{if\ } h(m) \neq 0,\ m \leq x\exp\left(-\LL^\varepsilon\right), m \in \mathcal{A}_{k-1},\\
0, & \text{otherwise,}
\end{cases}
$$
\begin{equation}\label{beta(n)}
\beta(n)=
\begin{cases}
1, &\text{if\ } n = p \in\mathcal{J}, p\equiv 1\!\!\!\!\!\pmod{4},\\
0, & \text{otherwise.}
\end{cases}
\end{equation}
Then we can write
$$Q^{Err_2} = \dfrac{1}{2}\sum_{k<\LL^{1-\varepsilon}}\dfrac{1}{k}{\sum\limits_{\substack{D, N, M \\ x\mathcal{L}^{-A_1}<NM\leqslant x}}}''\sum_{d\sim D}\chi_4(d)U_k(N, M, d) + R_3,$$
where
$$U_k(N, M, d) = \sum\limits_{\substack{n \sim N \\ m\sim M \\nm\leqslant x\\ nm\equiv 1\!\!\!\!\!\pmod{d}}}\alpha(m)\beta(n) - \dfrac{1}{\varphi(d)}\sum\limits_{\substack{n \sim N \\ m\sim M \\ nm\leqslant x\\(nm, d) = 1}}\alpha(m)\beta(n).$$
Let us remove the condition $nm\leqslant x$. To do this, we approximate $U_k(N, M, d)$ by the sum
$$U_f(N, M, d) = \sum\limits_{\substack{n \sim N \\ m\sim M \\ nm\equiv 1\!\!\!\!\!\pmod{d}}}\alpha(m)\beta(n)f_\delta\left(\dfrac{mn}{x} \right)  - \dfrac{1}{\varphi(d)}\sum\limits_{\substack{n \sim N \\ m\sim M \\ (nm, d) = 1}}\alpha(m)\beta(n)f_\delta\left(\dfrac{mn}{x} \right),$$
where the function $f_\delta(x)$ is defined in \eqref{f_delta(x)}, $\delta = \LL^{-A_2}$, and the constant $A_2 > A_1$ will be chosen later.

For the variables $n$ and $m$ appearing in the summation in $U_k$ and $U_f$, the following alternative is possible:
\begin{enumerate} 
\item $mn/x\leqslant \delta,$
\item $ mn/x \geqslant 1+\delta,$
\item $\delta < mn/x< 1+\delta.$
\end{enumerate}
Note that in the first case, $n$ and $m$ appear in each of the sums $U_f$ and $U_k$ with zero weight, so their contribution to the difference $U_f - U_k$ is zero.
The same holds for the second case. The third case is equivalent to the fulfillment of one of the following inequalities
\begin{equation}\label{ineq1}
x\LL^{-A_2}<mn\leqslant x \LL^{-A_1},\ \ x<mn< x(1+\LL^{-A_2})
\end{equation} or
\begin{equation}\label{ineq0}
x\LL^{-A_1}<mn\leqslant x.
\end{equation}
Proceeding as in the derivation of the estimates \eqref{link1} and \eqref{link2}, we obtain that the contribution to $Q^{Err_2}$ from those $m$ and $n$ which satisfy the inequalities \eqref{ineq1} does not exceed
$$
x\LL^{3-A_1} + x\LL^{3-A_2} \ll x\LL^{3-A_1}.
$$
Finally, those $m$ and $n$ which satisfy the inequality \eqref{ineq0} appear with weight $1$ in each of the sums $U_k$ and $U_f$, so their contribution to the difference $U_f - U_k$ is zero.

Thus, the error from replacing $U_k$ by $U_f$ does not exceed, in order, $x \LL^{3-A_1}$. From this we obtain
\begin{equation}\label{QErr2}
Q^{Err_2} = \dfrac{1}{2}\sum_{k<\LL^{1-\varepsilon}}\dfrac{1}{k}{\sum\limits_{\substack{D, N, M \\ x\LL^{-A_1}<NM\leqslant x}}}''\sum_{d\sim D}\chi_4(d)U_f(N, M, d) + R_4,
\end{equation}
where $$R_4\ll_{\varepsilon, A, A_1, A_2}\dfrac{x\ln \LL}{\LL^{1-\frac{\varepsilon}{4}}}.$$
Let us now write the inverse Mellin transform for the function $F_\delta$, defined in \eqref{F_d},
\begin{equation*}
f_\delta(u) = \dfrac{1}{2\pi i}\int_{0-\infty i}^{0+\infty i}F_\delta(s)u^{-s}ds = \dfrac{1}{2\pi} \int_{-\infty}^{+\infty }F_\delta(it)u^{-it} dt.
\end{equation*}
Let $T = \LL^{A_3}$ for some $A_3 > A_2 + 4$. Then, by virtue of Lemma~\ref{L8}, we have
$$f_\delta\left(\dfrac{mn}{x}\right) = \dfrac{1}{2\pi}\int_{-T}^T F_\delta(it)\left(\dfrac{mn}{x}\right)^{-it} dt + O\left(\dfrac{1}{\delta T}\right).$$
Substituting this expression into \eqref{QErr2}, we see that the contribution of the term $O\left(1/(\delta T)\right)$ to the sum $Q^{Err_2}$ does not exceed (in order)
\begin{multline*}\sum_{k\leqslant \LL^{1-\varepsilon}}{{\sum\limits_{\substack{D, N, M \\ NM\leqslant x}}}''}\sum_{d \sim D}\left(\dfrac{NM}{d\delta T}+\dfrac{1}{\delta T}+\dfrac{NM}{\varphi(d)\delta T}\right)\\
\ll (\ln \LL)\LL^{3-\varepsilon} \dfrac{x}{\delta T}+\dfrac{\sqrt{x}\LL^{3+A-\varepsilon}}{T}\ll_A \dfrac{x \ln \mathcal{L}}{\mathcal{L}^{A_3-A_2-3+\varepsilon}}\ll \dfrac{x}{\mathcal{L}}.
\end{multline*}
Changing the order of integration and summation over $d$, we obtain
\begin{equation}\label{Qerr2}
Q^{Err_2} = \dfrac{1}{4\pi}\sum_{k<\LL^{1-\varepsilon}}\dfrac{1}{k}{\sum\limits_{\substack{D, N, M \\ x\LL^{-A_1}<NM\leqslant x}}}''\int_{-T}^{T}F_{\delta}(it)\tilde{U}(D, N, M, t)x^{it} dt + R_5,
\end{equation}
where 
\begin{equation}\label{Utilde}
\tilde{U}(D, N, M, t) = \sum_{d \sim D}\chi_4(d)\left(\sum\limits_{\substack{n \sim N \\ m\sim M \\ nm\equiv 1\!\!\!\!\!\pmod{d}}}a(m)b(n) - \dfrac{1}{\varphi(d)}\sum\limits_{\substack{n \sim N \\ m\sim M \\ (nm, d) = 1}}a(m)b(n)\right)\!,
\end{equation}
the values $a(m)$ and $b(n)$ are defined by the equalities
\begin{equation}\label{ab}
a(m) = \alpha(m)m^{-it}, b(n) = \beta(n)n^{-it},
\end{equation}
and
$$R_5 \ll_{\varepsilon, A, A_1, A_2, A_3}\dfrac{x\ln \LL}{\LL^{1-\frac{\varepsilon}{4}}}.$$

\section{The dispersion method of Fouvry and Radziwi\l\l}
Changing the order of summation, we have
$$\tilde{U}(D, N, M, t) = \sum\limits_{\substack{m\sim M }}a(m)\left(\sum_{d \sim D}\chi_4(d)\sum\limits_{\substack{n \sim N \\ nm\equiv 1\!\!\!\!\!\pmod{d}}}b(n) - \sum\limits_{\substack{d \sim D}}\dfrac{\chi_4(d)}{\varphi(d)}\sum\limits_{\substack{n \sim N \\ (nm,d) = 1}}b(n)\right)\!.$$
Squaring the moduli of both sides of the equality and applying the Cauchy~--~Schwarz inequality, we obtain
\begin{multline*}
|\tilde{U}(D, N, M, t)|^2 \leqslant \sum\limits_{\substack{m\sim M}}|a(m)|^2\\
\times \sum\limits_{\substack{m\sim M}}\left|\sum_{d \sim D}\chi_4(d)\sum\limits_{\substack{n \sim N \\ nm\equiv 1\!\!\!\!\!\pmod{d}}}b(n) - \sum_{d \sim D}\dfrac{\chi_4(d)}{\varphi(d)}\sum\limits_{\substack{n \sim N \\ (nm,d) = 1}}b(n)\right|^2\\
\leqslant \left\|a\right\|_2^2\sum\limits_{\substack{m = -\infty}}^{+\infty}\psi\left(\dfrac{m}{M}\right) \left|\sum_{d \sim D}\chi_4(d)\sum\limits_{\substack{n \sim N \\ nm\equiv 1\!\!\!\!\!\pmod{d}}}b(n) - \sum_{d \sim D}\dfrac{\chi_4(d)}{\varphi(d)}\sum\limits_{\substack{n \sim N \\ (nm,d) = 1}}b(n)\right|^2\!,
\end{multline*}
where $\left\|a\right\|_2$ denotes the $\ell_2$-norm of the sequence $a(m)$,
\begin{equation}\label{a(m)}
\left\|a\right\|_2 = \left( \sum\limits_{\substack{m\sim M}}|a(m)|^2\right)^{\frac{1}{2}}\ll \sqrt{M}.
\end{equation}
Using the identity
$$|s_1-s_2|^2 = |s_1|^2-2\Real s_1\bar{s}_2+|s_2|^2,\ \ (s_1, s_2 \in \mathbb{C}),$$
we get
\begin{equation}\label{Uest}
|\tilde{U}(D, N, M, t)|^2 \leqslant \left\|a\right\|_2^2\left(W - 2\Real V + U \right)\!, 
\end{equation}
where  $$W = \sum_{m=-\infty}^{+\infty}\psi\left(\dfrac{m}{M}\right)\sum_{d_1, d_2\sim D}\chi_4(d_1)\chi_4(d_2)\sum\limits_{\substack{n_1, n_2\sim N \\ n_1m\equiv 1\!\!\!\!\!\pmod {d_1} \\ n_2m\equiv 1\!\!\!\!\!\pmod {d_2}}}b(n_1)\overline{b}(n_2),$$
$$V= \sum_{m=-\infty}^{+\infty}\psi\left(\dfrac{m}{M}\right)\sum_{d_1, d_2\sim D}\dfrac{\chi_4(d_1)\chi_4(d_2)}{\varphi(d_2)}\sum\limits_{\substack{n_1, n_2\sim N\\ n_1m\equiv 1\!\!\!\!\!\pmod {d_1} \\ (n_2m, d_2)=1}}b(n_1)\overline{b}(n_2),$$
$$U= \sum_{m=-\infty}^{+\infty}\psi\left(\dfrac{m}{M}\right)\sum_{d_1, d_2\sim D}\dfrac{\chi_4(d_1)\chi_4(d_2)}{\varphi(d_1)\varphi(d_2)}\sum\limits_{\substack{n_1, n_2\sim N \\ (n_1m, d_1)=1 \\ (n_2m, d_2)=1}}b(n_1)\overline{b}(n_2).$$
In the remaining part of this section, we transform the sums $U$ and $V$. Changing the order of summation in $U$, we obtain
$$U = \sum_{d_1, d_2\sim D}\dfrac{\chi_4(d_1)\chi_4(d_2)}{\varphi(d_1)\varphi(d_2)}\sum\limits_{\substack{n_1, n_2\sim N\\ (n_1, d_1)=1 \\ (n_2, d_2)=1}}b(n_1)\overline{b}(n_2)\sum\limits_{\substack{m = -\infty \\ (m, d_1d_2) = 1}}^{+\infty}\psi\left(\dfrac{m}{M}\right).$$
We apply Lemma~\ref{L7} to the inner sum. To do this, let us verify the inequality
$M>\exp\left(\sqrt{2\ln(d_1 d_2)}\right)\!.$ Let $M\leqslant \exp\left(2\sqrt{\LL}\right)$; then
$$NM\leqslant \exp\left(2\sqrt{\LL}\right) x^{\varepsilon}< x\LL^{-A_1},$$
which is false for sufficiently large $x$. Therefore,
$$M>\exp\left(2\sqrt{\LL}\right)\geqslant \exp\left(\sqrt{2\ln(d_1 d_2)}\right)\!,$$
and the conditions of Lemma~\ref{L7} are satisfied. Using the obvious inequalities $$\tau(d_1 d_2)\leqslant \tau(d_1)\tau(d_2)$$ and $M\leqslant x$, we obtain
\begin{equation}\label{psi}\sum\limits_{\substack{m = -\infty \\ (m,  d_1d_2) = 1}}^{+\infty}\psi\left(\dfrac{m}{M}\right) = \dfrac{\varphi(d_1 d_2)}{d_1 d_2} M\hat{\psi}(0) + O\left(\tau(d_1)\tau(d_2)\LL^2\right)\!.
\end{equation}
Taking into account the rough estimate $$\sum_{d \sim D}\dfrac{\tau(d)}{\varphi(d)}\ll \LL^2,$$ we obtain that the contribution of the $O$-term in \eqref{psi} to the sum $U$ does not exceed
$$\LL^2\sum_{d_1, d_2}\dfrac{\tau(d_1)\tau(d_2)}{\varphi(d_1)\varphi(d_2)}\sum\limits_{\substack{n_1, n_2 \sim N \\ n_1, n_2\text{\ both prime}}} 1 \ll \left(\dfrac{N}{\ln N} \right)^2 \LL^6\ll N^2\LL^6.$$
Therefore,
$$U = M\hat{\psi}(0)\sum_{d_1, d_2\sim D}\dfrac{\chi_4(d_1)\chi_4(d_2)}{\varphi(d_1)\varphi(d_2)}\dfrac{\varphi(d_1 d_2)}{d_1 d_2}\sum\limits_{\substack{n_1, n_2\sim N\\ (n_1, d_1)=1 \\ (n_2, d_2)=1}}b(n_1)\overline{b}(n_2) + O\left(N^2\LL^6\right).$$
Put $\Delta = (d_1, d_2),\ d_1 = \Delta k_1,\ d_2 = \Delta k_2,\ (k_1, k_2) = 1,$ then 
$$\varphi(d_1 d_2) = \dfrac{\varphi(d_1)\varphi(d_2)\Delta}{\varphi(\Delta)}.$$Since $\Delta \leqslant d_1\leqslant D_1$, where $D_1\leqslant 2D$, we have
\begin{multline}\label{U}
U = M\hat{\psi}(0)\sum\limits_{\substack{\Delta \leqslant D_1 \\ \Delta \text{\, odd}}}\dfrac{\Delta}{\varphi(\Delta)}\sum\limits_{\substack{d_1, d_2\sim D \\ (d_1, d_2) = \Delta}}\dfrac{\chi_4(d_1)\chi_4(d_2)}{d_1 d_2}\\
\times\sum\limits_{\substack{n_1, n_2\sim N\\ (n_1, d_1)=1 \\ (n_2, d_2)=1}}b(n_1)\overline{b}(n_2) + O\left(N^2\LL^6\right)\\
= U^{MT} + O\left(N^2\LL^6\right)\!,
\end{multline}
where
\begin{multline}\label{U^MT}
U^{MT} = M\hat{\psi}(0)\sum\limits_{\substack{\Delta \leqslant D_1 \\ \Delta \text{\, odd}}}\dfrac{1}{\Delta \varphi(\Delta)}\sum\limits_{\substack{k_1, k_2\sim \frac{D}{\Delta} \\ (k_1, k_2) = 1}}\dfrac{\chi_4(k_1)\chi_4(k_2)}{k_1 k_2}\\
\times\sum\limits_{\substack{n_1 \sim N \\ (n_1, \Delta k_1) = 1}}b(n_1)\sum\limits_{\substack{n_2 \sim N \\ (n_2, \Delta k_2) = 1}}\overline{b}(n_2).
\end{multline}
Now let us transform the sum $V$. Changing the order of summation, we obtain
\begin{equation}\label{V}
V = \sum_{d_1, d_2\sim D}\dfrac{\chi_4(d_1)\chi_4(d_2)}{\varphi(d_2)}\sum\limits_{\substack{n_1, n_2\sim N\\ (n_1, d_1)=1 \\ (n_2, d_2)=1}}b(n_1)\overline{b}(n_2)\sum\limits_{\substack{m = -\infty \\ mn_1 \equiv 1\!\!\!\!\!\pmod{d_1} \\ (m, d_2) = 1}}^{+\infty}\psi\left(\dfrac{m}{M}\right).
\end{equation}
Denoting the inner sum in \eqref{V} by $\vartheta$ and then using the M\"obius function summation property, we obtain
\begin{equation*}
\vartheta = \sum\limits_{\substack{m = -\infty \\ mn_1 \equiv 1\!\!\!\!\!\pmod{d_1} \\ (m, d_2) = 1}}^{+\infty}\psi\left(\dfrac{m}{M}\right) 
= \sum_{d|d_2}\mu(d) \sum\limits_{\substack{m = -\infty \\ mn_1 \equiv 1\!\!\!\!\!\pmod{d_1} \\ m \equiv 0\!\!\!\!\!\pmod{d}}}^{+\infty}\psi\left(\dfrac{m}{M}\right).
\end{equation*}
Note that $(d, d_1) = 1$, then
\begin{multline}\label{theta}
\vartheta = \sum\limits_{\substack{d | d_2 \\ (d, d_1) = 1}}\mu(d)\sum\limits_{\substack{m = -\infty \\ mn_1 \equiv 1\!\!\!\!\!\pmod{d_1} \\ m \equiv 0\!\!\!\!\!\pmod{d}}}^{+\infty}\psi\left(\dfrac{m}{M}\right)\\
=\sum\limits_{\substack{d | d_2 \\ (d, d_1) = 1}}\mu(d)\sum\limits_{\substack{\ell = -\infty \\ \ell \equiv (n_1 d)^{-1}\!\!\!\!\!\pmod{d_1} }}^{+\infty}\psi\left(\dfrac{\ell}{M/d}\right).
\end{multline}
Let us show that $M/d > \exp\left(\sqrt{2\ln d_1}\right)$. Indeed, if
$$
M \leqslant \sqrt{x}\LL^A \exp(\sqrt{2\LL}),
$$
then for sufficiently large $x$ we have
$$
MN \leqslant \sqrt{x}\LL^A \exp(\sqrt{2\LL}) x^\varepsilon < x\LL^{-A_1},
$$
which is false. Therefore,
$$
M > \sqrt{x}\LL^A \exp\left(\sqrt{2\LL}\right) \geqslant d\,\exp\left(\sqrt{2\ln d_1}\right).
$$
Now set
$$
H = \frac{5dd_1}{M} \left( \ln \frac{M}{d} \right)^4,
$$
then, by Lemma~\ref{L7}, we have
\begin{multline*}\sum\limits_{\substack{\ell = -\infty \\ \ell \equiv (n_1 d)^{-1}\!\!\!\!\!\pmod{d_1} }}^{+\infty}\psi\left(\dfrac{\ell}{M/d}\right) = \dfrac{M}{d d_1}\hat{\psi}(0)\\
+ \dfrac{M}{dd_1}\sum_{1\leqslant |m|\leqslant H}\hat{\psi}\left(\dfrac{mM}{d d_1}\right)e^{2\pi i\frac{m(dn_1)^*_{d_1}}{d_1}} + O\left(\dfrac{d}{d_1 M}\right)\\
= \dfrac{M}{d d_1}\hat{\psi}(0) + O\left(\dfrac{MH}{d d_1}\right) + O(1)
= \dfrac{M}{d d_1}\hat{\psi}(0) + O\left(\LL^4\right)\!.
\end{multline*}
Therefore, using \eqref{theta}, we find that
\begin{equation*}
\vartheta = \sum\limits_{\substack{d | d_2 \\ (d, d_1) = 1}}\mu(d)\left(\dfrac{M}{d d_1}\hat{\psi}(0) + O\left(\LL^4\right)\right)
=\dfrac{M}{d_1}\hat{\psi}(0)\prod\limits_{\substack{p | d_2 \\ p\nmid d_1}}\left(1-\dfrac{1}{p}\right) + O\left(\tau(d_2)\LL^4\right)\!.
\end{equation*}
Since
$$\prod\limits_{\substack{p | d_2 \\ p\nmid d_1}}\left(1-\dfrac{1}{p}\right) = \prod\limits_{\substack{p | \Delta k_2 \\ p\nmid \Delta k_1}}\left(1-\dfrac{1}{p}\right) = \prod\limits_{\substack{p | k_2}}\left(1-\dfrac{1}{p}\right) =\dfrac{\varphi(k_2)}{k_2},$$
it follows that
\begin{equation}
\label{theta1}
\vartheta = \dfrac{M}{\Delta k_1}\dfrac{\varphi(k_2)}{k_2}\hat{\psi}(0) + O\left( \tau(\Delta)\tau(k_2)\LL^4\right).
\end{equation}
The contribution of the $O$-term to the sum \eqref{V} does not exceed
$$\LL^4\sum_{\Delta \leqslant D_1}\dfrac{\tau(\Delta)}{\varphi(\Delta)}\sum_{k_1, k_2 \sim \frac{D}{\Delta}}\dfrac{\tau(k_2)}{\varphi(k_2)}\left(\dfrac{N}{\ln N}\right)^2\ll \LL^6 N^2 D.$$
The contribution of the main term \eqref{theta1} to the sum \eqref{V} is
\begin{multline*}
M\hat{\psi}(0)\sum\limits_{\substack{\Delta \leqslant D_1 \\ \Delta \text{\, odd}}}\dfrac{1}{ \varphi(\Delta)}\sum\limits_{\substack{k_1, k_2\sim \frac{D}{\Delta} \\ (k_1, k_2) = 1}}\dfrac{\chi_4(k_1)\chi_4(k_2)}{\varphi(k_2)}\\
\times\sum\limits_{\substack{n_1 \sim N \\ (n_1, \Delta k_1) = 1}}b(n_1)\sum\limits_{\substack{n_2 \sim N \\ (n_2, \Delta k_2) = 1}}\overline{b}(n_2)\dfrac{\varphi(k_2)}{\Delta k_1 k_2} = U^{MT}. 
\end{multline*}
Thus, 
$$V = U^{MT} + O(N^2 \LL^6 D).$$
Since $k_1$ and $k_2$ enter the sum $U^{MT}$ symmetrically, we have $U^{MT} = \overline{U^{MT}}$, and therefore $U^{MT}$ is a real number.

From this, we find that
\begin{equation}\label{W - 2Re U + V}
W- 2\Real V + U = W - U^{MT} + O\left(N^2 \LL^6 D\right)\!.
\end{equation}
\section{Contribution to $W$ from large $\Delta$, $\Delta_1$, and $\Delta_2$}
We have
$$W = \sum_{d_1, d_2\sim D}\chi_4(d_1)\chi_4(d_2)\sum\limits_{\substack{n_1, n_2\sim N \\ (n_1, d_1) = 1\\ \ (n_2, d_2) = 1 }}b(n_1)\overline{b}(n_2)\sum\limits_{\substack{m=-\infty \\ mn_1\equiv 1\!\!\!\!\!\pmod{d_1} \\ mn_1\equiv 1\!\!\!\!\!\pmod{d_2}}}^{+\infty}\psi\left(\dfrac{m}{M}\right)\!.$$
Below we will use the fact that for each pair $d_1$ and $d_2$ in the sum $W$, there is a unique representation of the form
\begin{equation}\label{ddk} d_1 = \Delta k_1 = \Delta \Delta_1 \kappa_1,\ \ d_2 = \Delta k_2 = \Delta\Delta_2\kappa_2,
\end{equation}
where
$$\Delta = (d_1, d_2),\ \ \Delta_1, \Delta_2 | \Delta^\infty,$$
$$k_1 = \dfrac{d_1}{\Delta},\ \ k_2 = \dfrac{d_2}{\Delta},\ \ (k_1, k_2) = 1,$$
$$(\kappa_1\Delta_1, \kappa_2\Delta_2) = 1, \ \ (\kappa_1, \Delta) = (\kappa_2, \Delta) = 1.$$
Let $X$ be a certain value such that $\mathcal{L}^{10}<X\leqslant \sqrt{N}.$
Denote by $R_k$, where $1 \leqslant k \leqslant 4$, the contribution of those terms in $W$ for which the following conditions hold, respectively
\begin{enumerate}
\item $\Delta = (d_1, d_2) > X;$
\item $\Delta_1 > X;$
\item $\Delta_2 > X;$
\item $n_1 = n_2 = p > X.$
\end{enumerate}

First, let us estimate $R_1$. We have 
\begin{equation*}
|R_1| \leqslant \sum_{X<\Delta\leqslant 2D}\sum\limits_{\substack{k_1, k_2 \sim \frac{D}{\Delta} \\ (k_1, k_2) = 1}}\sum\limits_{\substack{p_1,\, p_2 \sim N \\ p_i\nmid k_i\Delta,\, i=1, 2\\ p_1\equiv p_2\!\!\!\!\! \pmod{\Delta}}}\sum\limits_{\substack{m=-\infty \\ mp_1\equiv 1\!\!\!\!\!\pmod{k_1\Delta} \\ mp_2\equiv 1\!\!\!\!\!\pmod{k_2\Delta} }}^{+\infty}\psi\left(\dfrac{m}{M}\right).
\end{equation*}
The inner sum in this case does not exceed
\begin{equation*}
\sum\limits_{\substack{\frac{M}{2}\leqslant
m\leqslant\frac{5M}{2} \\
mp_1\equiv 1\!\!\!\!\!\pmod{k_1\Delta} \\mp_2-1\equiv 0\!\!\!\!\!\pmod{k_2}}} 1.
\end{equation*}
From this, taking the summation over $k_2$ as the inner sum, we obtain
\begin{multline*}|R_1|\leqslant 
\sum_{X<\Delta\leqslant 2D}\sum\limits_{\substack{k_1\sim \frac{D}{\Delta}}}\sum\limits_{\substack{p_1,\, p_2 \sim N \\p_1\equiv\, p_2\!\!\!\!\!\!\pmod{\Delta} \\ p_1\nmid k_1\Delta,\ p_2\nmid \Delta}}\sum\limits_{\substack{ m\leqslant\frac{5M}{2} \\ m\,\equiv\,\overline{p}_1\!\!\!\!\!\!\pmod{k_1\Delta}}}\tau(mp_2-1)\\
\leqslant \sum_{X<\Delta\leqslant 2D}\sum\limits_{\substack{k_1\sim \frac{D}{\Delta}}}\sum\limits_{\substack{p_1,\, p_2 \sim N \\p_1\equiv\, p_2\!\!\!\!\!\pmod{\Delta} \\ p_1\nmid k_1\Delta,\ p_2\nmid \Delta}}\sum\limits_{\substack{ n\leqslant 5MN \\ n\,\equiv -1\!\!\!\!\!\pmod{p_2} \\ n\,\equiv\,p_2 \overline{p}_1-1\!\!\!\!\!\pmod{k_1\Delta}}}\tau(n)
= R_{1, 1} + R_{1, 2},
\end{multline*}
where $\overline{a}$ for $(a, q) = 1$, as well as the symbol $a_q^*$, denotes the inverse residue of $a$ modulo $q$,
and $R_{1,1}$ and $R_{1,2}$ denote the contributions of those terms for which $X < \Delta \leqslant 2N$ and $2N < \Delta \leqslant 2D$, respectively.

Consider the sum $R_{1,1}$. Denote by $R_{1,1}^{(1)}$ the contribution of those $p_2$ which do not divide $k_1$, and by $R_{1,1}^{(2)}$ the contribution of the remaining $p_2$. Let
$$
a = p_2 (p_1)^*_{k_1 \Delta} - 1, \quad q = p_2 k_1 \Delta,
$$
and $\delta = (a, q)$. First, we estimate the sum $R_{1,1}^{(1)}$. Using the Chinese remainder theorem, we find
$$T = \sum\limits_{\substack{ n\leqslant 5MN \\ n\,\equiv -1\!\!\!\!\!\pmod{p_2} \\ n\,\equiv\,p_2 \overline{p}_1-1\!\!\!\!\!\pmod{k_1\Delta}}}\tau(n) = \sum\limits_{\substack{n \leqslant 5MN \\ n\equiv a\!\!\!\!\!\pmod{q}}}\tau(n)\leqslant \tau(\delta)\sum\limits_{\substack{\ell \leqslant \frac{5MN}{\delta} \\ \ell\equiv \frac{a}{\delta}\!\!\!\!\!\pmod{\frac{q}{\delta}}}}\tau(\ell).$$
Using the inequality
$$\dfrac{2ND}{\delta^{\frac{1}{4}}}\leqslant x^{\frac{1}{2}+2\varepsilon}\leqslant (MN)^{\frac{3}{4}},$$
which holds for sufficiently large $x$, we get
$$\dfrac{q}{\delta}\leqslant \left( \dfrac{5MN}{\delta}\right)^{\frac{3}{4}}\!\!.$$
From this, by Shiu's Theorem,
$$T \ll \dfrac{\tau(\delta)}{\varphi(\frac{q}{\delta})}\dfrac{MN}{\delta\mathcal{L}}\exp\left(2\sum_{p\leqslant x}\dfrac{1}{p} \right)\ll \dfrac{\tau(\delta) MN\mathcal{L}}{\varphi(\frac{q}{\delta})\delta}. $$

Using the inequalities $\varphi(q) \leqslant \delta \varphi(q/\delta)$ and $\tau(\delta) \leqslant \tau(q)$, as well as the fact that for any $a$ and $b$ the following hold:
$$
\tau(ab) \leqslant \tau(a) \tau(b), \quad \varphi(ab) \geqslant \varphi(a) \varphi(b),
$$
we obtain
$$T\ll \dfrac{\tau(q)MN\mathcal{L}}{\varphi(q)}\ll \dfrac{\tau(k_1)\tau(\Delta)MN\mathcal{L}}{\varphi(k_1)\varphi(\Delta)(p_2-1)} \ll \dfrac{\tau(k_1)\tau(\Delta)M\mathcal{L}}{\varphi(k_1)\varphi(\Delta)}.$$
Note that for $\Delta\leqslant 2N$, we have
$$\sum\limits_{\substack{p_1,\, p_2 \sim N \\p_1\equiv\, p_2\!\!\!\!\!\pmod{\Delta}}} 1 \ll N\left( \dfrac{N}{\Delta}+1\right)\ll \dfrac{N^2}{\Delta}.$$
Hence,
\begin{multline*}
R_{1,1}^{(1)}\ll M\mathcal{L} \sum_{X<\Delta\leqslant 2N}\dfrac{\tau(\Delta)}{\varphi(\Delta)}\sum\limits_{\substack{k_1\sim  \frac{D}{\Delta}}}\dfrac{\tau(k_1)}{\varphi(k_1)}\sum\limits_{\substack{p_1,\, p_2 \sim N \\p_1\equiv\, p_2\!\!\!\!\!\pmod{\Delta}}} 1\\
\ll MN^2\mathcal{L}\sum_{X<\Delta\leqslant 2N}\dfrac{\tau(\Delta)}{\Delta\varphi(\Delta)}\sum\limits_{\substack{k_1\leqslant \frac{2D}{\Delta}}}\dfrac{\tau(k_1)}{\varphi(k_1)}\\
\ll MN^2\mathcal{L}^3\sum_{X<\Delta\leqslant 2N}\dfrac{\tau(\Delta)}{\Delta\varphi(\Delta)}. 
\end{multline*}
Since
\begin{multline*}\sum_{\Delta > X}\dfrac{\tau(\Delta)}{\Delta\varphi(\Delta)} =\sum_{\nu\geqslant 0}\sum_{2^\nu X< \Delta\leqslant  2^{\nu+1}X}\dfrac{\tau(\Delta)}{\Delta\varphi(\Delta)}\\
\ll \sum_{\nu\geqslant 0}\dfrac{1}{2^\nu X}\sum_{\Delta\leqslant 2^{\nu+1}X}\dfrac{\tau(\Delta)}{\varphi(\Delta)}
\ll\sum_{\nu\geqslant 0}\dfrac{(\ln(2^{\nu+1}X))^2}{2^\nu X}\ll \dfrac{\mathcal{L}^2}{X},
\end{multline*}
we get
$$R_{1, 1}^{(1)} \ll \dfrac{MN^2\mathcal{L}^5}{X}.$$
Now let us estimate $R_{1,1}^{(2)}$. Proceeding as above and using the condition $p_2 \mid k$, we obtain
\begin{equation*}
\sum\limits_{\substack{ n\leqslant 5MN \\ n\,\equiv -1\!\!\!\!\!\pmod{p_2} \\ n\,\equiv\,p_2 \overline{p}_1-1\!\!\!\!\!\pmod{k_1\Delta}}}\tau(n) = \sum\limits_{\substack{n \leqslant 5MN \\ n\equiv p_2 \overline{p}_1-1\!\!\!\!\!\pmod{k_1\Delta}}}\tau(n)\\
\ll \dfrac{\tau(k_1)\tau(\Delta)MN\mathcal{L}}{\varphi(k_1)\varphi(\Delta)}.
\end{equation*}

Therefore, we have
\begin{multline*}
R_{1,1}^{(2)}\ll MN\mathcal{L} \sum_{X<\Delta\leqslant 2N}\dfrac{\tau(\Delta)}{\varphi(\Delta)}\sum\limits_{\substack{k_1\sim  \frac{D}{\Delta}}}\dfrac{\tau(k_1)}{\varphi(k_1)}\sum\limits_{\substack{p_1,\, p_2 \sim N \\p_1\equiv\, p_2\!\!\!\!\!\pmod{\Delta} \\ p_2|k_1}} 1\\
\ll MN\mathcal{L}\sum_{X<\Delta\leqslant 2N}\dfrac{\tau(\Delta)}{\varphi(\Delta)}\sum_{p_2\sim N}\sum\limits_{\substack{k_1\leqslant \frac{2D}{\Delta} \\ k_1\equiv 0\!\!\!\!\!\pmod{p_2}}}\dfrac{\tau(k_1)}{\varphi(k_1)}\left(\dfrac{N}{\Delta} + 1\right)\\
\ll MN^2\mathcal{L}\sum_{X<\Delta\leqslant 2N}\dfrac{\tau(\Delta)}{\Delta\varphi(\Delta)}\sum_{p_2\sim N}\dfrac{1}{p_2-1}\sum_{\ell\leqslant \frac{2D}{p_2\Delta}}\dfrac{\tau(\ell)}{\varphi(\ell)}\\
\ll   MN^2\mathcal{L}^3\sum_{\Delta>X}\dfrac{\tau(\Delta)}{\Delta\varphi(\Delta)}\ll\dfrac{MN^2\mathcal{L}^5}{X}.
\end{multline*}
Thus,
$$R_{1, 1} = R_{1, 1}^{(1)}+R_{1, 1}^{(2)} \ll \dfrac{MN^2\mathcal{L}^5}{X}.$$
Now consider the sum $R_{1,2}$. Note that the conditions $\Delta > 2N$ and
$$
p_1 \equiv p_2 \pmod{\Delta}
$$
imply the equality $p_1 = p_2$. Indeed, without loss of generality, assuming $p_1 \geqslant p_2$, for some integer $s \geqslant 0$ we have
$$
2N \geqslant p_1 = p_2 + \Delta s \geqslant N + 2Ns,
$$
which is possible only if $s = 0$ and $p_1 = p_2$. From this we obtain
\begin{multline*}
R_{1, 2} \leqslant \sum_{2N<\Delta\leqslant 2D}\sum_{k_1\sim\frac{D}{\Delta}}\sum_{p\sim N}\sum\limits_{\substack{m\leqslant \frac{5}{2}M \\ mp-1\equiv 0\!\!\!\!\!\pmod{\Delta k_1}}}\tau(mp-1)\\
\leqslant \sum_{2N<\Delta\leqslant 2D}\sum_{p\sim N}\sum\limits_{\substack{m\leqslant \frac{5}{2}M \\ mp-1\equiv 0\!\!\!\!\!\pmod{\Delta} }}\tau(mp-1)\sum\limits_{\substack{k_1\sim\frac{D}{\Delta} \\  k_1|mp-1}} 1 \\
\leqslant  \sum_{2N<\Delta\leqslant 2D}\sum_{p\sim N}\sum\limits_{\substack{m\leqslant \frac{5}{2}M \\ mp-1\equiv 0\!\!\!\!\!\pmod{\Delta}}}\tau^2(mp-1)\\
\leqslant \sum_{p\sim N}\sum\limits_{\substack{m\leqslant \frac{5}{2}M}}\tau^3(mp-1)\leqslant \sum_{n\leqslant 5MN}\tau^3(n-1)\tau(n).
\end{multline*}
Applying the Cauchy~--~Schwarz inequality to the last sum, as well as the estimate (see \cite{Mardzh})
\begin{equation}\label{Mardzh}
\sum_{n\leqslant x}\tau^k(n)\ll_k x(\ln x)^{2^{k}-1},
\end{equation}
valid for any fixed $k$, we obtain
\begin{equation*}R_{1, 2}\leqslant \sqrt{\sum_{n\leqslant 5MN}\tau^6(n)} \sqrt{\sum_{n\leqslant 5MN}\tau^2(n)}
\ll \sqrt{MN \mathcal{L}^{63}}\cdot\sqrt{MN \mathcal{L}^{3}} =  MN \mathcal{L}^{33}.
\end{equation*}
Thus, we have
\begin{equation}\label{R1}
|R_1|\leqslant R_{1, 1}+ R_{1, 2}\ll \dfrac{MN^2\mathcal{L}^5}{X} + MN\mathcal{L}^{33} = MN^2\left(\dfrac{\mathcal{L}^5}{X} + \dfrac{\mathcal{L}^{33}}{N}\right).
\end{equation}

Let us proceed to estimate the sum $R_2$ (the sum $R_3$ can be estimated similarly). Without loss of generality, we may assume that in the sum $R_2$ the condition $\Delta = (d_1, d_2) \leqslant X$ holds.

We have
\begin{equation*}
R_2 = \sum\limits_{\substack{d_1, d_2\sim D \\ \Delta = (d_1, d_2)\leqslant X\\
d_1 = \Delta\Delta_1\kappa_1\\ \Delta_1|\Delta^\infty,\ \Delta_1>X\\ (\kappa_1, \Delta) = 1}} \sum\limits_{\substack{n_1, n_2 \sim N \\ n_1\equiv n_2\!\!\!\!\!\pmod{\Delta} \\ (n_i, d_i) = 1,\ i = 1, 2 }}b(n_1)\bar{b}(n_2)\sum\limits_{\substack{m=-\infty \\ mn_1\equiv 1\!\!\!\!\!\pmod{d_1} \\  mn_2\equiv 1\!\!\!\!\!\pmod{d_2} }}^{+\infty}\psi\left(\dfrac{m}{M}\right).
\end{equation*}
From here, we get
\begin{multline}\label{eq}
|R_2|\leqslant \sum_{\Delta\leqslant X}\sum\limits_{\substack{\Delta_1 | \Delta^\infty \\ \Delta_1>X}}\sum\limits_{\substack{\kappa_1\sim \frac{D}{\Delta\Delta_1} \\ (\kappa_1, \Delta) = 1}}\sum\limits_{\substack{p_1, p_2\sim N \\ p_1\equiv p_2\!\!\!\!\!\pmod{\Delta} }}\sum\limits_{\substack{m = -\infty \\ mp_1\equiv 1\!\!\!\!\!\pmod{\Delta\Delta_1\kappa_1}}}^{+\infty}\psi\left(\dfrac{m}{M}\right)\sum\limits_{\substack{k_2\sim \frac{D}{\Delta} \\ \Delta k_2|mp_2-1}} 1\\
\leqslant \sum_{\Delta\leqslant X}\sum\limits_{\substack{\Delta_1 | \Delta^\infty \\ X<\Delta_1\leqslant \frac{2D}{\Delta}}}\sum\limits_{\substack{\kappa_1\sim \frac{D}{\Delta\Delta_1} \\ (\kappa_1, \Delta) = 1}}\sum\limits_{\substack{p_1, p_2\sim N \\ p_1\equiv p_2\!\!\!\!\!\pmod{\Delta} \\ p_2\nmid \Delta}}\sum\limits_{\substack{\frac{M}{2}<m\leqslant \frac{5}{2}M \\ mp_1\equiv 1\!\!\!\!\!\pmod{\Delta\Delta_1\kappa_1}}}\tau(mp_2-1).
\end{multline}
Proceeding as above, we obtain that the sum over $m$ in \eqref{eq} does not exceed, in order,
$$\dfrac{\tau(\Delta)\tau(\Delta_1)\tau(\kappa_1)}{\varphi(\Delta)\varphi(\Delta_1)\varphi([\kappa_1, p_2])}MN\mathcal{L}.$$

Separating the contribution to $R_2$ of terms satisfying the conditions $p_2 \nmid \kappa_1$ and $p_2 \mid \kappa_1$, we obtain
$$
|R_2| \ll R_2^{(1)} + R_2^{(2)},
$$
where
$$R_2^{(1)} = \sum_{\Delta\leqslant X}\dfrac{\tau(\Delta)}{\varphi(\Delta)}\sum\limits_{\substack{\Delta_1 | \Delta^\infty \\ X<\Delta_1\leqslant \frac{2D}{\Delta}}}\dfrac{\tau(\Delta_1)}{\varphi(\Delta_1)}\sum\limits_{\substack{\kappa_1\sim \frac{D}{\Delta\Delta_1} \\ (\kappa_1, \Delta) = 1}}\dfrac{\tau(\kappa_1)}{\varphi(\kappa_1)}\sum\limits_{\substack{p_1, p_2\sim N \\ p_1\equiv p_2\!\!\!\!\!\pmod{\Delta} \\ p_2\nmid \kappa_1}}\dfrac{MN\mathcal{L}}{\varphi(p_2)}$$
and
$$R_2^{(2)} = \sum_{p_2\sim N}\sum_{\Delta\leqslant X}\dfrac{\tau(\Delta)}{\varphi(\Delta)}\sum\limits_{\substack{\Delta_1 | \Delta^\infty \\ X<\Delta_1\leqslant \frac{2D}{\Delta}}}\dfrac{\tau(\Delta_1)}{\varphi(\Delta_1)}\sum\limits_{\substack{\kappa_1\sim \frac{D}{\Delta\Delta_1} \\ \kappa_1\equiv 0 \!\!\!\!\!\pmod{p_2}}}\dfrac{\tau(\kappa_1)}{\varphi(\kappa_1)}\sum\limits_{\substack{p_1\sim N \\ p_1\equiv p_2\!\!\!\!\!\pmod{\Delta}}}{MN\mathcal{L}}.$$
The sums $R_2^{(1)}$ and $R_2^{(2)}$ are estimated in exactly the same way. Let us estimate, for example, $R_2^{(1)}$. Since $\Delta \leqslant X < N$, we have
\begin{multline*}
R_2^{(1)}\ll MN\mathcal{L}\sum_{\Delta\leqslant X}\dfrac{\tau(\Delta)}{\varphi(\Delta)}\sum\limits_{\substack{\Delta_1 | \Delta^\infty \\ X<\Delta_1\leqslant \frac{2D}{\Delta}}}\dfrac{\tau(\Delta_1)}{\varphi(\Delta_1)}\sum\limits_{\substack{\kappa_1\sim \frac{D}{\Delta\Delta_1} \\ (\kappa_1, \Delta) = 1}}\dfrac{\tau(\kappa_1)}{\varphi(\kappa_1)}\left( \dfrac{N}{\Delta}+1\right) \\
\ll MN^2\mathcal{L}^3\sum_{\Delta\leqslant X}\dfrac{\tau(\Delta)}{\Delta\varphi(\Delta)}\sum\limits_{\substack{\Delta_1 | \Delta^\infty \\ X<\Delta_1\leqslant \frac{2D}{\Delta}}}\dfrac{\tau(\Delta_1)}{\varphi(\Delta_1)}.
\end{multline*}
The sum over $\Delta_1$ does not exceed
\begin{multline*}
\sum\limits_{\substack{\Delta_1 | \Delta^\infty \\ \Delta_1>X}}\dfrac{\tau(\Delta_1)}{\varphi(\Delta_1)} \leqslant \dfrac{1}{\sqrt{X}}\sum_{\Delta_1|\Delta^\infty}\dfrac{\tau(\Delta_1)\sqrt{\Delta_1}}{\varphi(\Delta_1)} = \dfrac{1}{\sqrt{X}}\prod_{p|\Delta}\left(1 + \sum_{\nu = 1}^{+\infty}\dfrac{(\nu+1)\sqrt{p^\nu}}{p^\nu-p^{\nu-1}}\right)\\
=\dfrac{1}{\sqrt{X}}\prod_{p|\Delta}\left(1 + O\left( \dfrac{1}{\sqrt{p}}\right) \right)\ll \dfrac{\tau(\Delta)}{\sqrt{X}}.
\end{multline*}
Hence, we obtain
$$R_2^{(1)}\ll \dfrac{MN^2\mathcal{L}^3}{\sqrt{X}}\sum_{\Delta\leqslant X}\dfrac{\tau^2(\Delta)}{\Delta\varphi(\Delta)} \ll \dfrac{MN^2\mathcal{L}^3}{\sqrt{X}}.$$
Similarly, we find
$$R_2^{(2)} \ll \dfrac{MN^2\mathcal{L}^3}{\sqrt{X}}.$$
Thus, 
\begin{equation}\label{R23}|R_2|, |R_3|\ll R_2^{(1)} + R_2^{(2)}\ll  \dfrac{MN^2\mathcal{L}^3}{\sqrt{X}}.
\end{equation}
Finally, let us estimate the quantity $R_4$. Again, without loss of generality, assume that $\Delta \leqslant X$. We have
\begin{multline*}
|R_4|\leqslant \sum\limits_{\substack{d_1, d_2\sim D \\ \Delta = (d_1, d_2)\leqslant X}} \sum\limits_{\substack{p \sim N }}\sum\limits_{\substack{m=-\infty \\ mp\equiv 1\!\!\!\!\!\pmod{d_1} \\  mp\equiv 1\!\!\!\!\!\pmod{d_2} }}^{+\infty}\psi\left(\dfrac{m}{M}\right)\\
\leqslant \sum_{\Delta\leqslant X} \sum\limits_{\substack{k_1, k_2\sim \frac{D}{\Delta}\\ (k_1, k_2) = 1}}\sum\limits_{\substack{p \sim N }}\sum\limits_{\substack{\frac{M}{2}<m\leqslant \frac{5M}{2} \\ mp\equiv 1\!\!\!\!\!\pmod{\Delta k_1} \\  mp\equiv 1\!\!\!\!\!\pmod{\Delta k_2} }} 1
\leqslant \sum_{\Delta\leqslant X}\sum\limits_{\substack{p \sim N }}\sum_{m\leqslant \frac{5M}{2}}\tau^2(mp-1).
\end{multline*}
Introducing the notation $n = mp$ and applying the Cauchy~--~Schwarz inequality together with the estimate \eqref{Mardzh}, we obtain
\begin{equation}\label{R4}|R_4|\leqslant X\sum_{n\leqslant 5MN}\tau(n)\tau^{2}(n-1)\ll X MN\mathcal{L}^9.
\end{equation}
Combining the estimates \eqref{R1}, \eqref{R23}, and \eqref{R4}, we obtain
\begin{equation*}
|R_1|+ |R_2| + |R_3| + |R_4|
\ll MN^2\left( \dfrac{\mathcal{L}^5}{X} + \dfrac{\mathcal{L}^{33}}{N} +  \dfrac{\mathcal{L}^3}{\sqrt{X}} + \dfrac{X\mathcal{L}^9}{N} \right)\ll \dfrac{MN^2\mathcal{L}^{3}}{\sqrt{X}}.
\end{equation*}
Thus, for the sum $W$ we have 
\begin{equation}\label{W}
W = W(D; X) + O\left( \dfrac{MN^2\mathcal{L}^{3}}{\sqrt{X}}\right)\!,
\end{equation}
where $W(D; X)$ denotes the contribution of those terms in $W$ for which simultaneously $\Delta, \Delta_1, \Delta_2 \leqslant X$ and $n_1 \neq n_2$.
\section{Transformation of $W(D; X)$}
The values $n_1$, $n_2$, and $m$ involved in the sum $W$ satisfy the system of congruences
\begin{equation}\label{sysW}
\begin{cases*}
m\equiv \overline{n}_1\!\!\!\!\!\pmod{d_1}\\
m\equiv \overline{n}_2\!\!\!\!\!\pmod{d_2},
\end{cases*}
\end{equation}
which in turn is equivalent to the following system
$$
\begin{cases*}
m\equiv \overline{n}_1\!\!\!\!\!\pmod{\Delta\Delta_1}\\
m\equiv \overline{n}_2\!\!\!\!\!\pmod{\Delta\Delta_2}\\
m\equiv \overline{n}_1\!\!\!\!\!\pmod{\kappa_1}\\
m\equiv \overline{n}_2\!\!\!\!\!\pmod{\kappa_2}.
\end{cases*}
$$
Here the quantities $\Delta$, $\Delta_i$, and $\kappa_i$ are defined in \eqref{ddk}.  
From this, introducing the notation $\tau = \Delta \Delta_1 \Delta_2$, and using Lemma~\ref{L9} and the Chinese remainder theorem, we obtain that \eqref{sysW} is equivalent to the system
$$
\begin{cases*}
n_1\equiv n_2\!\!\!\!\!\pmod{\Delta}\\
m\equiv \nu \!\!\!\!\!\pmod{\tau\kappa_1\kappa_2},\\
\end{cases*}
$$
where
\begin{equation}\label{nu}
\nu = \lambda \kappa_1\kappa_2(\kappa_1\kappa_2)^*_{\tau} + \tau\kappa_2(n_1\tau\kappa_2)^*_{\kappa_1} + \tau\kappa_1(n_2\tau\kappa_1)^*_{\kappa_2},
\end{equation}
and $\lambda$ is defined in \eqref{lambda}.
Using the definition of $W(D; X)$, we obtain
\begin{multline}\label{W(D,X)}
W(D; X) = \!\!\!\!\!\sum\limits_{\substack{d_1, d_2\sim D \\ \Delta = (d_1, d_2) \leqslant X \\ d_j = \Delta\Delta_j\kappa_j,\ j=1, 2\\ \Delta_j |\Delta^\infty,\ (\kappa_j, \Delta) = 1,\ \Delta_j \leqslant X }}\!\!\!\!\!\chi_4(d_1 d_2)\\
\times\sum\limits_{\substack{n_1\neq n_2\\ n_1, n_2\sim N \\ n_1 \equiv n_2\!\!\!\!\!\pmod{\Delta} \\ (n_j, d_j) = 1,\ j=1, 2 }}b(n_1)\overline{b}(n_2)\sum\limits_{\substack{m=-\infty \\ m\equiv \nu\!\!\!\!\!\pmod{\tau\kappa_1\kappa_2}}}^{+\infty}\psi\left(\dfrac{m}{M}\right)\!.
\end{multline}
Set $q = \tau\kappa_1\kappa_2$ and 
\begin{equation}\label{HHH}
H = \left( 20D^2/M\right)(\ln M)^4.
\end{equation} Then
$$q  = \dfrac{d_1d_2}{\Delta}\leqslant 4D^2\ll x\mathcal{L}^{2A}.$$
Therefore, $H\geqslant (5q/M)(\ln M)^4$ and  $x\gg q\mathcal{L}^{-2A}$. From the last inequality, taking into account that $x$ is sufficiently large, as well as that $MN \geqslant x \mathcal{L}^{-A_1}$ and $q \geqslant \sqrt{x} \mathcal{L}^{-A}$, we obtain
$$M\geqslant x N^{-1}\mathcal{L}^{-A_1}> x^{1-2\varepsilon}\geqslant (q\mathcal{L}^{-2A})^{1-2\varepsilon}\geqslant q^{1-3\varepsilon}\geqslant e^{\sqrt{2\ln q}}.$$
Thus, the sum over $m$ in \eqref{W(D,X)} satisfies the conditions of Lemma \ref{L7}. Taking into account that $q = \tau \kappa_1 \kappa_2 = \Delta k_1 k_2$, we find
\begin{multline}\label{WWW}\sum\limits_{\substack{m=-\infty \\ m\equiv \nu\!\!\!\!\!\pmod{\tau\kappa_1\kappa_2}}}^{+\infty}\psi\left(\dfrac{m}{M}\right) = \dfrac{M\hat{\psi}(0)}{\Delta k_1 k_2}\\ + \dfrac{M}{\tau\kappa_1\kappa_2}\sum_{1\leqslant |h|\leqslant H}\hat{\psi}\left( \dfrac{hM}{\tau\kappa_1\kappa_2}\right)\exp\left( 2\pi i\dfrac{h\nu}{\tau\kappa_1\kappa_2}\right)  + O\left(\dfrac{1}{\Delta k_1 k_2 M}\right)\!.
\end{multline}
Denoting by $W^{MT}$, $W^{Err_1}$, and $W^{Err_2}$ the contributions to the sum $W(D; X)$ of the first, second, and third terms of the formula \eqref{WWW}, respectively, we obtain  
$$W^{MT} = M\hat{\psi}(0)\sum\limits_{\substack{\Delta\leqslant X \\ \Delta\text{\, odd}}}\dfrac{1}{\Delta}\!\!\!\!\!\sum\limits_{\substack{k_1, k_2\sim \frac{D}{\Delta} \\ (k_1, k_2) = 1 \\ k_j = \Delta_j \kappa_j,\ \Delta_j|\Delta^\infty,\ j=1, 2\\ \Delta_j\leqslant X,\ (\kappa_j, \Delta) = 1 }}\!\!\!\!\!\dfrac{\chi_4(k_1 k_2)}{k_1k_2}\!\!\!\!\!\sum\limits_{\substack{n_1, n_2\sim N \\ (n_j, \Delta k_j) = 1,\ j=1, 2\\ n_1 \equiv n_2\!\!\!\!\!\pmod{\Delta},\ n_1\neq n_2 }}\!\!\!\!\! b(n_1)\overline{b}(n_2),$$
\begin{multline}\label{Weer1}
W^{Err_1} = M  \sum\limits_{\substack{\Delta \leqslant X \\ \Delta \text{\, odd}}} \dfrac{1}{\Delta}\sum\limits_{\substack{\Delta_j\leqslant X\\ (\Delta_1, \Delta_2) = 1 \\ \Delta_j|\Delta^\infty,\ j=1, 2 }}\dfrac{\chi_4(\Delta_1\Delta_2)}{\Delta_1\Delta_2}\!\!\!\!\sum\limits_{\substack{\kappa_j \sim \frac{D}{\Delta\Delta_j} \\ (\kappa_1, \kappa_2) = 1 \\ (\kappa_j, \Delta) = 1,\  j=1,2}}\!\!\!\!\dfrac{\chi_4(\kappa_1\kappa_2)}{\kappa_1\kappa_2}\\
\times \sum_{1\leqslant |h|\leqslant H}\hat{\psi}\left(\dfrac{hM}{\tau\kappa_1\kappa_2}\right)\!\!\!\!\!\sum\limits_{\substack{n_1, n_2\sim N \\ (n_j, \Delta \kappa_j) = 1,\ j=1, 2\\ n_1 \equiv n_2\!\!\!\!\!\pmod{\Delta},\ n_1\neq n_2 }}\!\!\!\!\! b(n_1)\overline{b}(n_2)\exp\left(2\pi i \dfrac{h\nu}{\tau\kappa_1\kappa_2}\right)
\end{multline}
and
\begin{equation}\label{Werr2}
W^{Err_2} \ll \dfrac{1}{M}\sum_{\Delta\leqslant X}\dfrac{1}{\Delta}\sum_{k_1, k_2\sim \frac{D}{\Delta}}\dfrac{1}{k_1k_2}\sum_{p_1, p_2\sim N}1\ll \dfrac{N^2\ln X}{M}.
\end{equation}
In the remaining part of the paragraph, we transform the sum $W^{MT}$. We add to $W^{MT}$ the terms with $\Delta_1 > X$. The contribution of such terms does not exceed
\begin{multline*}M\sum_{\Delta\leqslant X}\dfrac{1}{\Delta}\sum\limits_{\substack{\Delta_1|\Delta^\infty \\ \Delta_1> X}}\dfrac{1}{\Delta_1}\sum\limits_{\substack{\kappa_1\sim \frac{D}{\Delta\Delta_1} \\ (\kappa_1, \Delta) = 1}}\dfrac{1}{\kappa_1}\sum_{k_2\sim \frac{D}{\Delta}}\dfrac{N^2}{k_2\Delta}\\
\ll\dfrac{MN^2}{\sqrt{X}}\sum_{\Delta\leqslant X}\dfrac{1}{\Delta^2}\prod_{p|\Delta}\left(1+\dfrac{1}{\sqrt{\Delta}}\right)\ll \dfrac{MN^2}{\sqrt{X}}. 
\end{multline*}
Similarly, the contribution of terms with the condition $\Delta_2 > X$ is estimated. We add to $W^{MT}$ the terms with the condition $n_1 = n_2$. The contribution of such terms does not exceed
$$M\sum_{\Delta\leqslant X}\dfrac{1}{\Delta}\sum_{k_1, k_2\sim \frac{D}{\Delta}}\dfrac{N}{k_1k_2}\ll MN\ln X.$$
Finally, we add to $W^{MT}$ the terms with the condition $X < \Delta \leqslant D_1$, where $D_1 \leqslant 2D$. The contribution of such terms does not exceed
\begin{equation*}
M\left( \sum_{X<\Delta\leqslant 2N} + \sum_{2N<\Delta\leqslant D_1}\right)\dfrac{1}{\Delta}\sum_{k_1, k_2 \sim \frac{D}{\Delta}}\dfrac{1}{k_1k_2}\sum\limits_{\substack{n_1, n_2\sim N \\ n_1 \equiv n_2\!\!\!\!\pmod{\Delta} }} 1\ll \dfrac{MN^2}{X} + MN\mathcal{L}.
\end{equation*}
Therefore, dropping the conditions $\Delta, \Delta_1, \Delta_2 \leqslant X$ and $n_1 \neq n_2$ in the sum $W^{MT}$ leads to an error of order
$$MN^2\left(\dfrac{1}{X} + \dfrac{\mathcal{L}}{N} + \dfrac{\ln X}{N} + \dfrac{1}{\sqrt{X}}\right)\ll \dfrac{MN^2}{\sqrt{X}}.$$
Thus, we have
\begin{multline}\label{W^MT}
W^{MT} =  M\hat{\psi}(0)\sum\limits_{\substack{\Delta \leqslant D_1 \\ \Delta \text{\, odd}}}\dfrac{1}{\Delta }\sum\limits_{\substack{k_1, k_2\sim \frac{D}{\Delta} \\ (k_1, k_2) = 1}}\dfrac{\chi_4(k_1)\chi_4(k_2)}{k_1 k_2}\!\!\!\!\!\sum\limits_{\substack{n_1, n_2\sim N \\ (n_j, \Delta k_j) = 1,\ j=1, 2\\ n_1 \equiv n_2\!\!\!\!\!\pmod{\Delta} }}\!\!\!\!\! b(n_1)\overline{b}(n_2)\\
+ O\left(\dfrac{MN^2}{\sqrt{X}}\right)
= M\hat{\psi}(0)\sum\limits_{\substack{\Delta \leqslant D_1 \\ \Delta \text{\, odd}}}\dfrac{1}{\Delta }\sum\limits_{\substack{k_1, k_2\sim \frac{D}{\Delta} \\ (k_1, k_2) = 1}}\dfrac{\chi_4(k_1)\chi_4(k_2)}{k_1 k_2}\\
\times\sum\limits_{\substack{\delta = 1 \\ (\delta, \Delta) = 1}}^\Delta \sum\limits_{\substack{n_1\sim N \\ n_1\equiv \delta\!\!\!\!\!\pmod{\Delta} \\ (n_1, k_1) = 1}}b(n_1) \sum\limits_{\substack{n_2\sim N \\ n_2\equiv \delta\!\!\!\!\!\pmod{\Delta} \\ (n_2, k_2) = 1}}\overline{b}(n_2) + O\left(\dfrac{MN^2}{\sqrt{X}}\right)\!.
\end{multline}
\section{Estimate of $W^{Err_1}$}
Dividing and multiplying the expression \eqref{Weer1} by $D^2$, and also changing the order of summation, we obtain
\begin{equation}
\label{werr}
W^{Err_1} = \dfrac{M}{D^2} \sum\limits_{\substack{\Delta \leqslant X \\ \Delta \text{\, odd}}} \dfrac{1}{\Delta}\sum\limits_{\substack{\Delta_j\leqslant X\\ (\Delta_1, \Delta_2) = 1 \\ \Delta_j|\Delta^\infty,\ j=1, 2 }}\dfrac{\chi_4(\Delta_1\Delta_2)}{\Delta_1\Delta_2}\cdot\mathcal{W}(\Delta, \Delta_1, \Delta_2),
\end{equation}
where 
$$\mathcal{W}(\Delta, \Delta_1, \Delta_2) = \!\!\!\!\!\sum\limits_{\substack{n_1, n_2\sim N \\ (n_j, \Delta) = 1,\ j=1, 2\\ n_1 \equiv n_2\!\!\!\!\!\pmod{\Delta},\ n_1\neq n_2 }}\!\!\!\!\! b(n_1)\overline{b}(n_2)\cdot \mathcal{S},$$
the sum $\mathcal{S} = \mathcal{S}(n_1, n_2, \Delta, \Delta_1, \Delta_2)$ is defined by the equality
$$\mathcal{S} = \sum\limits_{\substack{\kappa_j \sim \frac{D}{\Delta\Delta_j} \\ (\kappa_1, \kappa_2) = 1 \\ (n_j, \kappa_j)=(\kappa_j, \Delta) = 1,\  j=1,2}}\!\!\!\!{\chi_4(\kappa_1)\chi_4(\kappa_2)}\sum_{0< |h|\leqslant H}f(h, \kappa_1, \kappa_2)\exp\left(2\pi i \dfrac{h\nu}{\tau\kappa_1\kappa_2}\right)\!,$$
where
$$f(h, \kappa_1, \kappa_2) = \hat{\psi}\left(\dfrac{hM}{\tau\kappa_1\kappa_2}\right)\dfrac{D^2}{\kappa_1\kappa_2}.$$
Denote by $\mathcal{W}^{+}$ and $W^{Err_1}_{>0}$ the contributions from $h > 0$ in the sums $\mathcal{W}$ and $W^{Err_1}$, respectively. The notations $\mathcal{W}^{-}$ and $W^{Err_1}_{<0}$ have analogous meanings. Let us estimate the sum $\mathcal{W}^+$ (the sum $\mathcal{W}^-$ can be estimated similarly).

For brevity, let us set
$$\mathcal{D}_{\vec{i}} = \mathcal{D}_{i_1, i_2, i_3} = \dfrac{\partial^{i_1 + i_2 + i_3}}{(\partial x_1)^{i_1}(\partial x_2)^{i_2}(\partial x_3)^{i_3}}.$$
We need the following estimate
\begin{equation}
	\label{Df}
	\mathcal{D}_{i_1, i_2, i_3}f(x_1, x_2, x_3) \ll \dfrac{X^4}{x_1^{i_1}x_2^{i_2}x_3^{i_3}},\ \ (0\leqslant i_1, i_2, i_3\leqslant 1),
\end{equation}
which holds for $x_2 \asymp \kappa_1$ and $x_3 \asymp \kappa_2$.
To prove it, we estimate all derivatives of the function $\hat{\psi}\left(\frac{a x_1}{x_2 x_3}\right)$, where $a = M/\tau$. 
We use Lemma \ref{L10.5} for the case $g(z) = \hat{\psi}(z)$.  
By virtue of Lemma \ref{L5.5}, for the functions $F_1$, $F_2$, and $F_3$ defined in Lemma \ref{L10.5}, we have $F_1(z), F_2(z), F_3(z) \ll 1$. Therefore,
$$\mathcal{D}_{i_1, i_2, i_3}\hat{\psi}\left(\dfrac{ax_1}{x_2x_3}\right)\ll \dfrac{1}{x_1^{i_1}x_2^{i_2}x_3^{i_3}},\ \ (0\leqslant i_1, i_2, i_3\leqslant 1).$$
It is also easy to see that
$$\mathcal{D}_{i_1, i_2, i_3}\left(\dfrac{D^2}{x_2 x_3}\right)
\begin{cases}
	\ll \dfrac{D^2}{x_2 x_3}\cdot\dfrac{1}{x_2^{i_2}x_3^{i_3}}, & \text{if}\ \, i_1 = 0,\\
	=0, & \text{if}\ \, i_1 > 0.
\end{cases}$$
From here, we get
\begin{multline*}
	\mathcal{D}_{\vec{i}}f(x_1, x_2, x_3) \ll
	\sum\limits_{\substack{0 \leqslant j_s\leqslant i_s \\ 1\leqslant s\leqslant 3 \\ j_1 = i_1}}\mathcal{D}_{\vec{j}}\left( \hat{\psi}\left(\dfrac{ax_1}{x_2x_3}\right)\right) \cdot \mathcal{D}_{\vec{i}-\vec{j}}\left(\dfrac{D^2}{x_2 x_3}\right)\\
	\ll \sum\limits_{\substack{0 \leqslant j_s\leqslant i_s \\ 1\leqslant s\leqslant 3 \\ j_1 = i_1}}\dfrac{1}{x_1^{j_1}x_2^{j_2}x_3^{j_3}}\cdot\dfrac{D^2}{x_2x_3}\cdot\dfrac{1}{x_2^{i_2-j_2}x_3^{i_3-j_3}}
	\ll
	\dfrac{D^2}{x_2 x_3}\cdot\dfrac{1}{x_1^{i_1}x_2^{i_2}x_3^{i_3}}.
\end{multline*}
Taking into account that $\kappa_1 \sim D/(\Delta \Delta_1)$ and $\kappa_2 \sim D/(\Delta \Delta_2)$, we obtain
$$
\dfrac{D^2}{x_2 x_3} \ll \Delta^2 \Delta_1 \Delta_2 \ll X^4,
$$
which proves \eqref{Df}.  
Let us write the sum $\mathcal{W}^+$ in a more convenient form. To this end, set $\ell_1 = h$, $\ell_2 = \kappa_1$, and $\ell_3 = \kappa_2$. Also set
$$c(\vec{\ell}) = c_1(\vec{\ell})\exp\left(2\pi i\dfrac{\ell_1\nu}{\tau \ell_2\ell_3}\right)\!,$$
$$c_1(\vec{\ell}) = \chi_4(\ell_2\ell_3)\mathbb{I}_{(\ell_2, \ell_3) = 1} \mathbb{I}_{(\ell_2, \Delta) = (\ell_3, \Delta) = 1}\mathbb{I}_{(\ell_2, n_1) = (\ell_3, n_2) = 1},$$
$$K_1 = 0,\ \ L_1 = H,$$
$$K_{i+1} = \dfrac{D}{\Delta\Delta_i},\ \ L_{i+1} = \dfrac{D_1}{\Delta\Delta_i},\ \ (i = 1, 2).$$
From this, applying Lemma \ref{L11}, we obtain
\begin{multline*}
	\mathcal{W}^+ 
	= C_0(\vec{L})f(\vec{L})\\
	+ \sum_{s = 1}^3 (-1)^s \!\!\!\!\!\sum_{1\leqslant i_1<\cdots<i_s\leqslant 3}\int_{K_{i_1}}^{L_{i_1}}\cdots \int_{K_{i_s}}^{L_{i_s}}C_0(\vec{b}_{i_1, \ldots, i_s})\dfrac{\partial^s{f}(\vec{b}_{i_1, \ldots, i_s})}{\partial{x_{i_1}}\cdots \partial{x_{i_s}}}dx_{i_1}\cdots dx_{i_s},
\end{multline*}
where $\vec{b}_{i_1, \ldots, i_s}$ is defined in Lemma \ref{L11}, and
$$C_0(\vec{x}) = \sum\limits_{\substack{n_1, n_2\sim N \\ (n_j, \Delta ) = 1,\ j=1, 2\\ n_1 \equiv n_2\!\!\!\!\!\pmod{\Delta},\ n_1\neq n_2 }}\!\!\!\!\! b(n_1)\overline{b}(n_2)\sum\limits_{\substack{K_j<\ell_j\leqslant x_j \\ 1\leqslant j\leqslant 3}}c(\vec{\ell}).$$
Applying the bound \eqref{Df}, we get
$$|\mathcal{W}^+|\ll |C_0(\vec{x}_1)| X^4\left( 1+ \sum_{s = 1}^3  \sum_{1\leqslant i_1<\cdots<i_s\leqslant 3}\int_{K_{i_1}}^{L_{i_1}}\dfrac{dx_{i_1}}{x_{i_1}}\cdots \int_{K_{i_s}}^{L_{i_s}}\dfrac{dx_{i_s}}{x_{i_s}}\right)\!,$$
where the tuple 
$$
\vec{x}_1 = \vec{x}_1(\Delta, \Delta_1, \Delta_2) = (x_1', x_2', x_3'),
$$
satisfying the condition $K_j < x_j' \leqslant L_j$, is chosen so that the sum $C_0(\vec{x}_1)$ is maximal in absolute value.
Since
$$\int_{K_i}^{L_i}\dfrac{dx_i}{x_i}\ll
\begin{cases*}
	\ln H, & \text{if}\ \, i = 1,\\
	1, & \text{if}\ \, i = 2, 3,
\end{cases*}
$$
it follows that

\begin{equation}\label{W+}
	|\mathcal{W}^+|\ll |C_0(\vec{x}_1)| X^4\ln H.
\end{equation}
Now choose the tuple $\vec{x}_2 = (x_1'', x_2'', x_3'')$ so that
$$|C_0(\vec{x}_2)| = \max\limits_{\substack{\Delta, \Delta_1, \Delta_2\leqslant X \\
		(\Delta_1, \Delta_2) = 1\\
		\Delta_1, \Delta_2 |\Delta^{\infty}}}|C_0(\vec{x}_1(\Delta, \Delta_1, \Delta_2))|.$$  
It is clear that
$0<x_1''\leqslant H$
and
$$\dfrac{D}{X^2}<x_j''\leqslant 2D,\ \ 1< \dfrac{x''_j}{K_j}\leqslant 2\ \ (j = 2, 3).$$
Hence, from \eqref{werr} and $\eqref{W+}$ we get
\begin{multline}\label{W>0}
	W^{Err_1}_{>0}\ll \dfrac{MX^4\ln H}{D^2}\times|C_0(\vec{x}_2)|\sum_{\Delta,\Delta_1,\Delta_2\leqslant X}1\\
	\ll\dfrac{MX^7\ln H}{D^2}\times|C_0(\vec{x}_2)|.
\end{multline}
Now let us estimate the quantity $C_0(\vec{x}_2)$. We split this sum into parts depending on the residues that $n_1$, $n_2$, $\ell_2$, and $\ell_3$ give upon division by $\tau$,
\begin{equation}
	\label{C0}
	C_0(\vec{x}_2) = \sum\limits_{\substack{1\leqslant \alpha_j\leqslant \tau,\,  1\leqslant j\leqslant 4 \\ \alpha_1\equiv\alpha_2\!\!\!\!\pmod{\Delta}}} \sum\limits_{\substack{n_1, n_2\sim N\\ n_1\neq n_2 \\ n_1\equiv\alpha_1\!\!\!\!\pmod{\tau} \\ n_2\equiv\alpha_2\!\!\!\!\pmod{\tau} \\ (n_j, \Delta ) = 1,\ j=1, 2}}b(n_1)\bar{b}(n_2)\cdot S_0(\vec{x}_2, \vec{\alpha}),
\end{equation}
$$S_0(\vec{x}_2, \vec{\alpha}) = \sum\limits_{\substack{K_j< \ell_j\leqslant x_j'',\,  1\leqslant j\leqslant 3 \\ \ell_2\equiv\alpha_3\!\!\!\!\pmod{\tau} \\ \ell_3\equiv\alpha_4\!\!\!\!\pmod{\tau}}}c_1(\vec{\ell})\exp\left(2\pi i\dfrac{\ell_1\nu}{\tau \ell_2\ell_3}\right)\!.$$
Using \eqref{nu}, as well as the facts that $\ell_2 \equiv \alpha_3 \pmod{\tau}$ and $\ell_3 \equiv \alpha_4 \pmod{\tau}$, we obtain
\begin{multline*}\exp\left(2\pi i\dfrac{\ell_1\nu}{\tau \ell_2\ell_3}\right) = \exp\left(2\pi i \ell_1\dfrac{\lambda (\alpha_3\alpha_4)^*_\tau}{\tau} \right)\\
	\times \exp\left(2\pi i \ell_1\left( \dfrac{(\tau \ell_3 n_1)^*_{\ell_2}}{\ell_2} + \dfrac{(\tau \ell_2 n_2)^*_{\ell_3}}{\ell_3}\right)\right)\!.
\end{multline*}
Applying Lemma \ref{L10} to the second factor with $a = \tau$, $b = \ell_2$, $c = \ell_3$, and $d = n_1$, $e = n_2$, we find
\begin{multline*}
	\exp\left(2\pi i\dfrac{\ell_1\nu}{\tau \ell_2\ell_3}\right) = \exp\left(2\pi i \ell_1\dfrac{\lambda (\alpha_3\alpha_4)^*_\tau-(\alpha_1\alpha_3\alpha_4)^*_\tau}{\tau} \right)\\
	\times \exp\left(2\pi i \dfrac{\ell_1}{\tau \ell_2 \ell_3 n_1}\right)\exp\left(2\pi i \ell_1 \dfrac{(n_1-n_2)(\tau \ell_2 n_2)^*_{\ell_3n_1}}{\ell_3n_1}\right)\!.  
\end{multline*}
Note that $\lambda$ is determined by the numbers $\Delta, \Delta_1, \Delta_2$ and the residues $$(n_1)^*_{\Delta\Delta_1} = (\alpha_1)^*_{\Delta\Delta_1},\ \ (n_2)^*_{\Delta\Delta_2} = (\alpha_2)^*_{\Delta\Delta_2}.$$ 
Setting 
$$F(x_1, x_2, x_3) = \exp\left(\dfrac{ax_1}{x_2 x_3}\right)\!,\ a = \dfrac{2\pi i}{\tau n_1},$$ we have
$$S_0(\vec{x}_2, \vec{\alpha}) =  \sum\limits_{\substack{K_1<\ell_1\leqslant x_1''}}c(\ell_1)\sum\limits_{\substack{K_j<\ell_j\leqslant x_j'',\, j = 2, 3 \\ \ell_2 \equiv \alpha_3\!\!\!\!\pmod{\tau} \\ \ell_3 \equiv \alpha_4\!\!\!\!\pmod{\tau} }}F(\vec{\ell})c_1(\vec{\ell})\exp\left(2\pi i \ell_1 \dfrac{(n_1-n_2)(\tau \ell_2 n_2)^*_{\ell_3n_1}}{\ell_3n_1}\right)\!,$$
where
$$c(\ell_1) = \exp\left(2\pi i \ell_1\dfrac{\lambda (\alpha_3\alpha_4)^*_\tau-(\alpha_1\alpha_3\alpha_4)^*_\tau}{\tau} \right)\!.$$
Applying Lemma \ref{L10.5} with $g(z) = e^z$ for $1 \leqslant s \leqslant 3$ and
$$
1 \leqslant i_1 < \cdots < i_s \leqslant 3,
$$
we find
$$\dfrac{\partial^{s}F}{\partial x_{i_1}\cdots\partial x_{i_s}}\ll \left( \dfrac{ax_1}{x_2x_3}+\cdots+\left(\dfrac{ax_1}{x_2x_3}\right)^{s} \right)\dfrac{1}{x_{i_1} \cdots x_{i_s}}.$$
For $x_1 \leqslant H$ and $x_2 \asymp D/(\Delta \Delta_1)$, $x_3 \asymp D/(\Delta \Delta_2)$, by virtue of \eqref{HHH} we have
$$\dfrac{ax_1}{x_2 x_3}\ll \dfrac{H\Delta}{D^2 N}\ll \dfrac{(\ln M)^4X}{MN}\ll 1,$$
and therefore
$$\dfrac{\partial^{s}F}{\partial x_{i_1}\cdots\partial x_{i_s}}\ll \dfrac{1}{x_{i_1} \cdots x_{i_s}}.$$
Applying Lemma \ref{L11}, for some vector $\vec{x}_3 = (x_1^{(3)}, x_2^{(3)}, x_3^{(3)})$ satisfying the condition
$$
K_j < x_j^{(3)} \leqslant x_j'', \quad (1 \leqslant j \leqslant 3),
$$
we obtain
\begin{multline}\label{S0}
	|S_0(\vec{x}_2, \vec{\alpha})|\ll |C_1(\vec{x}_3)|\\
	\times \left(1 + \sum_{s=1}^3\sum_{1\leqslant i_1<\cdots<i_s\leqslant 3}\int_{K_{i_1}}^{L_{i_1}}\dfrac{dx_{i_1}}{x_{i_1}}\cdots \int_{K_{i_s}}^{L_{i_s}}\dfrac{dx_{i_s}}{x_{i_s}}\right)
	\ll |C_1(\vec{x}_3)|\ln H,
\end{multline}
where
$$C_1(\vec{x}_3)=\sum\limits_{\substack{K_1<\ell_1\leqslant x_1^{(3)}}}c(\ell_1)\sum\limits_{\substack{K_j<\ell_j\leqslant x_j^{(3)},\, j=2, 3\\ \ell_2 \equiv \alpha_3\!\!\!\!\pmod{\tau} \\ \ell_3 \equiv \alpha_4\!\!\!\!\pmod{\tau} }}c_1(\vec{\ell})\exp\left(2\pi i \ell_1 \dfrac{(n_1-n_2)(\tau \ell_2 n_2)^*_{\ell_3n_1}}{\ell_3n_1}\right)\!.$$
Finally, let us estimate the sum $C_1(\vec{x}_3)$. We split the sum over $\ell_1$ into dyadic intervals. 
Denote summation over the left endpoints of such intervals by a double prime. Also, we make a change of variables in the summation,
$$\mu = \ell_2\cdot\tau n_2,\ \ \eta = \ell_3\cdot n_1.$$
Then since $c_1(\vec{\ell}) = 0$ for $(\eta, \mu)>1$, we have
\begin{equation}\label{last_ineq}
	C_1(\vec{x}_3)
	\ll {\sum_{K\leqslant x_1^{(3)}}}''\left|\sum_{\ell_1\sim K}c(\ell_1)\sum_{\mu\sim K_2\tau n_2}a(\mu)\sum\limits_{\substack{\eta \sim K_3n_1 \\ (\eta,\,\mu) = 1}}b(\eta)\exp\left( 2\pi i \ell_1\dfrac{\vartheta\mu^*_\eta}{\eta}\right) \right|\!,
\end{equation}
where $\vartheta = n_1-n_2\neq 0,$ 
\begin{multline*}a(\mu) = \chi_4\left(\dfrac{\mu}{\tau n_2}\right)\mathbb{I}\left(\tau n_2 | \mu, \left(\frac{\mu}{\tau n_2},\ \Delta\right) = 1, \frac{\mu}{\tau n_2}\equiv \alpha_3\!\!\!\!\pmod{\tau}\right)\\
	\times\mathbb{I}\left( \frac{\mu}{\tau n_2}\in\left(K_2,x_2^{(3)}\right], \left(\frac{\mu}{\tau n_2}, n_1\right) = 1\right)\!.
\end{multline*}
and
\begin{multline*}b(\eta) = \chi_4\left(\dfrac{\eta}{n_1}\right)\mathbb{I}\left(n_1 | \eta, \left(\frac{\eta}{n_1},\ \Delta\right) = 1, \frac{\eta}{n_1}\equiv \alpha_4\!\!\!\!\pmod{\tau}\right)\\
	\times \mathbb{I}\left(\frac{\eta}{n_1}\in \left(K_3, x_3^{(3)}\right], \left(\frac{\eta}{n_1}, n_2\right) = 1\right)\!.
\end{multline*}
Now let us estimate the inner sum.  
Noting that
$$K_2K_3\tau n_1 n_2\ll \dfrac{D^2}{\Delta^2\Delta_1\Delta_2}\Delta\Delta_1\Delta_2 N^2 \ll D^2N^2,$$
$$K_2K_3\tau n_1 n_2\gg \dfrac{D^2 N^2}{\Delta}\gg \dfrac{D^2 N^2}{X},$$
$$\dfrac{|\vartheta|K}{K_2 K_3 \tau n_1 n_2}\ll \dfrac{H\Delta}{D^2 N}\ll\dfrac{(\ln M)^4 X}{M N}\ll 1,$$
$$K_2\tau n_2 + K_3 n_1\ll D N X,$$
and using Theorem \ref{BettinChandee}, taking into account the inequality $X < N \leqslant x^\varepsilon$, we verify that the inner sum in \eqref{last_ineq} does not exceed in order
\begin{multline*}
\sqrt{K}DN\left( \left( KD^2 N^2\right)^{\frac{7}{20}+\varepsilon}\left(DNX\right)^{\frac{1}{4}} + (KD^2 N^2)^{\frac{3}{8}+\varepsilon}K^{\frac{1}{8}} (DNX)^{\frac{1}{8}} \right)\\
\ll x^{5\varepsilon }\left(K^{\frac{17}{20}}D^{\frac{39}{20}}+ KD^{\frac{15}{8}}\right)\!.  
\end{multline*}
Summing over $K$ gives
$$C_1(\vec{x}_3)\ll x^{5\varepsilon }\left(H^{\frac{17}{20}}D^{\frac{39}{20}}+ HD^{\frac{15}{8}}\right)\!.$$
It follows from \eqref{S0} and \eqref{C0} that
$$C_0(\vec{x}_2)\ll N^2 (\ln H) x^{5\varepsilon}\left(H^{\frac{17}{20}}D^{\frac{39}{20}}+ HD^{\frac{15}{8}}\right)\!.$$
Finally, from \eqref{W>0} we get
\begin{multline*}W^{Err_1}_{>0}\ll \dfrac{M}{D^2}X^{7}(\ln H)^2N^2 x^{5\varepsilon}\\
\times\left(H^{\frac{17}{20}}D^{\frac{39}{20}}+ HD^{\frac{15}{8}}\right)
\ll M x^{20\varepsilon} \left( \dfrac{H^{\frac{17}{20}}}{D^{\frac{1}{20}}} + \dfrac{H}{D^{\frac{1}{8}}}\right)\!.
\end{multline*}
Similarly, the sum $W^{Err_1}_{<0}$ is estimated.\\  
Finally, we find
\begin{equation}\label{Werr1}
W^{Err_1}\ll M x^{20\varepsilon} \left( \dfrac{H^{\frac{17}{20}}}{D^{\frac{1}{20}}} + \dfrac{H}{D^{\frac{1}{8}}}\right)\ll x^{21\varepsilon }\left(D^{\frac{33}{20}}M^{\frac{3}{20}} + D^{\frac{15}{8}}\right)\!.
\end{equation}

\section{Estimation of $W^{MT}-U^{MT}$}
Set
$$
Z = W^{MT} - U^{MT}, 
$$
then from \eqref{U^MT} and \eqref{W^MT} we obtain
\begin{equation}\label{W-U}
Z = M\hat{\psi}(0)\sum\limits_{\substack{\Delta \leqslant D_1 \\ \Delta \text{\, odd}}}\dfrac{1}{\Delta }\sum\limits_{\substack{k_1, k_2\sim \frac{D}{\Delta} \\ (k_1, k_2) = 1}}\dfrac{\chi_4(k_1)\chi_4(k_2)}{k_1 k_2}\cdot S,
\end{equation}
where
\begin{equation*}
S =  \sum\limits_{\substack{\delta = 1 \\ (\delta, \Delta) = 1}}^\Delta \sum\limits_{\substack{n_1, n_2\sim N  \\ (n_j, \Delta k_j) = 1 \\ n_j\equiv \delta\!\!\!\!\!\pmod{\Delta},\, j=1, 2}}b(n_1)\overline{b}(n_2)-\dfrac{1}{\varphi(\Delta)} \sum\limits_{\substack{n_1, n_2\sim N  \\ (n_j, \Delta k_j) = 1,\, j=1, 2}}b(n_1)\overline{b}(n_2).
\end{equation*}
Using the M\"obius function summation property, we obtain
\begin{equation}
\label{Moebius}
\sum\limits_{\substack{k_1, k_2\sim \frac{D}{\Delta} \\ (k_1, k_2) = 1}}\dfrac{\chi_4(k_1)\chi_4(k_2)}{k_1 k_2} = \sum_{t\leqslant \frac{2D}{\Delta}}\mu(t)\sum\limits_{\substack{k_1, k_2\sim \frac{D}{\Delta} \\ k_j\equiv 0\!\!\!\!\pmod{t},\, j=1, 2 }}\dfrac{\chi_4(k_1)\chi_4(k_2)}{k_1 k_2}.
\end{equation}
Let $Y$ be a random variable taking each of the values
$$y(\delta) = \sum\limits_{\substack{k\sim \frac{D}{\Delta} \\ k\equiv 0 \!\!\!\!\pmod{ t}}}\dfrac{\chi_4(k)}{k}\sum\limits_{\substack{n \sim N\\ (n, \Delta k) = 1 \\ n\equiv \delta\!\!\!\!\!\pmod{\Delta}}}b(n),\ \ \  (1\leqslant \delta\leqslant \Delta,\ \ (\delta, \Delta) = 1),$$
with probability $1/\varphi(\Delta)$. Then \eqref{W-U} can be written in the form
\begin{equation}	Z	= M\hat{\psi}(0)\sum\limits_{\substack{\Delta \leqslant D_1 \\ \Delta \text{\, odd}}}\dfrac{1}{\Delta }\sum_{t\leqslant \frac{2D}{\Delta}}\mu(t)\varphi(\Delta)\left(\mathbb{E}(|Y|^2)-|\mathbb{E}(Y)|^2\right)\!,
\end{equation}
where $\mathbb{E}(\cdot)$ denotes the expectation of a random variable.\\  
Using the formula for the variance
$$
\mathbb{E}\left(|Y - \mathbb{E}(Y)|^2\right) = \mathbb{E}(|Y|^2) - |\mathbb{E}(Y)|^2,
$$
and again applying the equality \eqref{Moebius}, we obtain
\begin{equation*}Z = M\hat{\psi}(0)\sum\limits_{\substack{\Delta \leqslant D_1 \\ \Delta \text{\, odd}}}\dfrac{1}{\Delta }\sum\limits_{\substack{k_1, k_2\sim \frac{D}{\Delta} \\ (k_1, k_2) = 1}}\dfrac{\chi_4(k_1)\chi_4(k_2)}{k_1 k_2}\sum\limits_{\substack{\delta = 1 \\ (\delta, \Delta) = 1}}^\Delta E(\Delta, \delta, k_1)\bar{E}(\Delta, \delta, k_2),
\end{equation*}
where 
$$E(\Delta, \delta, k) = \sum\limits_{\substack{n \sim N \\ (n,  \Delta k) = 1 \\ n\equiv\delta\!\!\!\!\pmod{\Delta}}}b(n) - \dfrac{1}{\varphi(\Delta)}\sum\limits_{\substack{n \sim N \\ (n,  \Delta k) = 1}}b(n).$$
Passing to the estimates, we have
\begin{equation*}|Z|
\leqslant M|\hat{\psi}(0)|
\sum\limits_{\substack{\Delta \leqslant D_1 \\ \Delta \text{\, odd}}}\dfrac{1}{\Delta }
\sum\limits_{\substack{\delta = 1 \\ (\delta, \Delta) = 1}}^\Delta\left( \sum_{k\sim \frac{D}{\Delta}}
\dfrac{|E(\Delta, \delta, k)|}{k}\right)^2\!.
\end{equation*}
Using the Cauchy~--~Schwarz inequality, we find
\begin{multline}\label{estZ}
|Z|\leqslant M|\hat{\psi}(0)|\sum\limits_{\substack{\Delta \leqslant D_1 \\ \Delta \text{\, odd}}}\dfrac{1}{\Delta }\sum\limits_{\substack{\delta = 1 \\ (\delta, \Delta) = 1}}^\Delta\sum_{k'\sim\frac{D}{\Delta}}\dfrac{1}{k'}
\sum_{k\sim \frac{D}{\Delta}}\dfrac{|E(\Delta, \delta, k)|^2}{k}\\
\ll \dfrac{M}{D}\sum\limits_{\substack{\Delta \leqslant D_1 \\ \Delta \text{\, odd}}}\sum\limits_{\substack{\delta = 1 \\ (\delta, \Delta) = 1}}^\Delta \sum_{k\sim\frac{D}{\Delta}}|E(\Delta, \delta, k)|^2 = \dfrac{M \mathcal{E}}{D},
\end{multline}
where the meaning of the notation $\mathcal{E}$ is clear.\\
Using \eqref{beta(n)} and \eqref{ab}, as well as the fact that
$(N, N_1] \subseteq \mathcal{J}$, we can write
$$E(\Delta, \delta, k) = \sum\limits_{\substack{N<p\leqslant N_1,\, p \nmid k \\ p \equiv 1\!\!\!\!\pmod{4} \\ p\equiv \delta\!\!\!\!\pmod{\Delta} }}p^{-it}-\dfrac{1}{\varphi(\Delta)}\sum\limits_{\substack{N<p\leqslant N_1,\, p \nmid \Delta k \\ p \equiv 1\!\!\!\!\pmod{4} }}p^{-it}.$$
Let us introduce the notations
$$\pi_r(x) = \sum\limits_{\substack{N<p\leqslant x,\, p \nmid r \\ p \equiv 1\!\!\!\!\pmod{4} }} 1,\ \ \ \ \pi_r(x; q, a) = \sum\limits_{\substack{N<p\leqslant x,\, p \nmid r \\ p \equiv 1\!\!\!\!\pmod{4} \\ p \equiv a\!\!\!\!\pmod{q} }} 1.$$
We also set
$$C(x; \Delta, \delta) = \pi_k(x; \Delta, \delta) - \pi_k(N; \Delta, \delta),$$
$$C(x) = \pi_{\Delta k}(x) - \pi_{\Delta k}(N).$$
Using partial summation, we find
\begin{multline}E(\Delta, \delta, k) = N_1^{-it}\left(C(N_1;\Delta, \delta) - \dfrac{1}{\varphi(\Delta)}C(N_1)\right)\\ +it\int_N^{N_1}\left(C(u;\Delta, \delta) - \dfrac{1}{\varphi(\Delta)}C(u) \right)u^{-it-1}du.
\end{multline}

Taking into account the inequality $|t| \leqslant T$ 
for some $N < N_2 \leqslant N_1$, we obtain
$$
|E(\Delta, \delta, k)|^2 \ll T^2 \left| C(N_2; \Delta, \delta) - \dfrac{1}{\varphi(\Delta)} C(N_2) \right|^2\!.
$$
Let us write the last difference in the form
$$C(N_2;\Delta, \delta) - \dfrac{1}{\varphi(\Delta)}C(N_2) = R- \rho_1 + \dfrac{1}{\varphi(\Delta)}\rho_2,$$
where 
$$ R = R(N_2; \Delta, \delta) = \sum\limits_{\substack{N<p\leqslant N_2 \\ p \equiv 1\!\!\!\!\pmod{4} \\ p\equiv \delta\!\!\!\!\pmod{\Delta} }}1-\dfrac{1}{\varphi(\Delta)}\sum\limits_{\substack{N<p\leqslant N_2\\ p \equiv 1\!\!\!\!\pmod{4} }} 1$$
and
$$\rho_1 = \rho_1(N_2; \Delta, \delta) = \sum\limits_{\substack{p|k \\ N<p\leqslant N_1\\ p \equiv 1\!\!\!\!\pmod{4} \\ p\equiv \delta\!\!\!\!\pmod{\Delta} }}1,\ \ \ \ \rho_2 = \rho_2(N_2; \Delta, \delta) = \sum\limits_{\substack{p|\Delta k \\ N<p\leqslant N_1\\ p \equiv 1\!\!\!\!\pmod{4} }} 1.$$
Therefore, we have
$$\mathcal{E}\ll T^2\sum\limits_{\substack{\Delta \leqslant D_1 \\ \Delta \text{\, odd}}}\sum\limits_{\substack{\delta = 1 \\ (\delta, \Delta) = 1}}^\Delta \sum_{k\sim\frac{D}{\Delta}}\left(R^2 +\rho_1^2 + \dfrac{1}{\varphi^2(\Delta)}\rho_2^2 \right) = \mathcal{E}_1 + \mathcal{E}_2^{(1)} + \mathcal{E}_2^{(2)}.$$
For the sum $\mathcal{E}_2^{(1)}$ we have
\begin{multline*}\mathcal{E}_2^{(1)} \ll T^2 \sum\limits_{\substack{\Delta \leqslant 2D}}\sum\limits_{\substack{\delta = 1 \\ (\delta, \Delta) = 1}}^\Delta \sum_{k\sim\frac{D}{\Delta}}\sum\limits_{\substack{p_1, p_2 |k \\ N<p_1, p_2\leqslant 2N \\ p_1\equiv p_2\equiv \delta\!\!\!\!\pmod{\Delta} }} 1\\
\ll  T^2 \sum\limits_{\substack{\Delta \leqslant 2D}}\sum\limits_{\substack{ N<p_1, p_2\leqslant 2N \\ p_1\equiv p_2\!\!\!\!\pmod{\Delta} }} \sum\limits_{\substack{k \sim \frac{D}{\Delta} \\ k\equiv 0\!\!\!\!\pmod{[p_1, p_2]}}} 1 \leqslant T^2 \sum\limits_{\substack{\Delta \leqslant 2D}}\sum\limits_{\substack{ N<p_1, p_2\leqslant 2N}} \left\lfloor\dfrac{2D}{\Delta [p_1, p_2]}\right\rfloor\\
\ll T^2 D \sum\limits_{\substack{\Delta \leqslant 2D}}\dfrac{1}{\Delta}\left( \sum\limits_{\substack{ N<p\leqslant 2N }}\dfrac{1}{p} + \sum\limits_{\substack{N<p_1, p_2\leqslant 2N \\ p_1\neq p_2}}\dfrac{1}{p_1 p_2}\right)
\ll T^2 D\mathcal{L}.
\end{multline*}
For the sum $\mathcal{E}_2^{(2)}$ we have
$$\mathcal{E}_2^{(2)} \ll T^2 \sum\limits_{\substack{\Delta \leqslant 2D}}\sum\limits_{\substack{\delta = 1 \\ (\delta, \Delta) = 1}}^\Delta \sum_{k\sim\frac{D}{\Delta}}\dfrac{1}{\varphi^2(\Delta)} \left(\omega^2(\Delta)+\omega^2(k) \right)\!,$$
where $\omega(n) = \sum_{p \mid n} 1$ denotes the prime divisor function.

The contribution of the term $\omega^2(k)$ is estimated similarly to the sum $\mathcal{E}_2^{(1)}$ and is of order $O(T^2 D \mathcal{L})$. The contribution of the term $\omega^2(\Delta)$ does not exceed in order
$$T^2 D\sum_{\Delta \leqslant 2 D}\dfrac{\omega^2(\Delta)}{\Delta\varphi(\Delta)}\ll T^2 D.$$
Thus, $\mathcal{E}_2^{(2)}\ll T^2 D\mathcal{L}.$

Let us proceed to the estimate of $\mathcal{E}_1$. We have
\begin{equation*}\mathcal{E}_1 \ll T^2 D \sum\limits_{\substack{\Delta \leqslant D_1 \\ \Delta \text{\, odd}}}\dfrac{1}{\Delta} \sum\limits_{\substack{\delta = 1 \\ (\delta, \Delta) = 1}}^\Delta R^2(N_2; \Delta, \delta)
= T^2 D \left( \Sigma_1 + \Sigma_2 + \Sigma_3\right),
\end{equation*}
where $\Sigma_1$ denotes the contribution from $1 \leqslant \Delta \leqslant \mathcal{L}^{A_4}$ to the sum $\mathcal{E}_1$ for some $A_4 > 2 A_3$, $\Sigma_2$ denotes the contribution from $\mathcal{L}^{A_4} < \Delta \leqslant N/C$ for some constant $C \geqslant 10$, and finally, $\Sigma_3$ denotes the contribution from $N/C < \Delta \leqslant 2D$.

Let us estimate the contribution of the sum $\Sigma_1$. Since $N \geqslant e^{\mathcal{L}^\varepsilon}$, we have $\mathcal{L} \leqslant (\ln N)^{\varepsilon}$, and therefore $\Delta \leqslant \mathcal{L}^{A_4} \leqslant (\ln N)^{\frac{A_4}{\varepsilon}}$. Consequently, by the Siegel~--~Walfisz theorem (see \cite[Theorem 8.17]{Siegel-Walfisz})
$$
R(N_2; \Delta, \delta) \ll_{\varepsilon, A_4} N \exp\left(-c_1 \sqrt{\ln N}\right)
$$
for some $c_1 > 0$. From this we obtain that the contribution of $\Sigma_1$ to the sum $\mathcal{E}_1$ does not exceed in order
\begin{multline*}T^2 D \sum\limits_{\substack{\Delta \leqslant \mathcal{L}^{A_4}}}\dfrac{1}{\Delta} \sum\limits_{\substack{\delta = 1 \\ (\delta, \Delta) = 1}}^\Delta N^2 \exp\left(-2c_1\sqrt{\ln N}\right) \\
\ll T^2 N^2 D \mathcal{L}^{A_4}\exp\left(-2c_1\mathcal{L}^{\frac{\varepsilon}{2}}\right) \ll_{\varepsilon, A_3, A_4} N^2 D \exp\left(-c_1\mathcal{L}^{\frac{\varepsilon}{2}}\right)\!.
\end{multline*}
Now let us estimate the contribution from the sum $\Sigma_2$. Estimating the sum $R$ trivially, we have
$$|R(N_2;\Delta, \delta)|\ll \dfrac{N}{\Delta} + 1 + \dfrac{N}{\varphi(\Delta)} \ll\dfrac{N}{\varphi(\Delta)}.$$ From this, the desired contribution does not exceed
\begin{multline*}
T^2 D \sum_{\mathcal{L}^{A_4}<\Delta\leqslant \frac{N}{C}}\dfrac{1}{\Delta}\sum\limits_{\substack{\delta = 1 \\ (\delta, \Delta) = 1}}^\Delta \dfrac{N^2}{\varphi^2(\Delta)}\\
\ll N^2 T^2 D\sum_{\Delta>\mathcal{L}^{A_4}}\dfrac{1}{\Delta\varphi(\Delta)}\ll \dfrac{N^2 T^2 D}{\mathcal{L}^{A_4}} = N^2D \mathcal{L}^{2A_3-A_4}.
\end{multline*}
Finally, let us estimate the contribution of $\Sigma_3$ to the sum $\mathcal{E}_1$. Note that for $\Delta > N/C$, there exist at most $O_C(1)$ natural numbers $n \leqslant 2N$ lying in the progression $n \equiv \delta \pmod{\Delta}$. From this, for the quantity $R^2$ we obtain
$$|R(N_2;\Delta, \delta)|^2\ll \pi^2(2N; \Delta, \delta) + \dfrac{N^2}{\varphi^2(\Delta)} \ll_C \pi(2N; \Delta, \delta) + \dfrac{N^2}{\varphi^2(\Delta)}.$$
Consequently, the contribution of $\Sigma_3$ does not exceed in order
$$T^2 D \sum_{\frac{N}{C}<\Delta\leqslant 2D}\dfrac{1}{\Delta}\sum\limits_{\substack{\delta = 1 \\ (\delta, \Delta) = 1}}^\Delta\left( \pi(2N; \Delta, \delta) + \dfrac{N^2}{\varphi^2(\Delta)} \right)\ll T^2 D N.$$
Thus, we have
$$\mathcal{E}_1\ll N^2 D\left(\exp\left(-c_1\mathcal{L}^{\frac{\varepsilon}{2}}\right)+ \mathcal{L}^{2A_3-A_4}+\dfrac{T^2}{N}\right)\ll \dfrac{N^2 D}{\mathcal{L}^{A_4 - 2A_3}}.$$
Consequently, taking into account that $T = \mathcal{L}^{A_2}$ and $N \geqslant \exp\left(\mathcal{L}^\varepsilon\right)$, we have
$$\mathcal{E}\ll N^2 D\left(\dfrac{1}{\mathcal{L}^{A_4-2A_3}} + \dfrac{T^2 \mathcal{L}}{N^2}\right)\ll \dfrac{N^2 D}{\mathcal{L}^{A_4-2A_3}}.$$
Let us take $X = N^\varepsilon$, then from \eqref{estZ} we find
$$W^{MT} - U^{MT}\ll \dfrac{MN^2}{\mathcal{L}^{A_4-2A_3}}.$$
\section{Completion of the estimation of $Q^{Err_2}$}
From \eqref{W - 2Re U + V}, \eqref{W}, and the inequalities \eqref{Werr1} and \eqref{Werr2}, it follows that
\begin{multline*}
W-2\Real U + V =W^{MT} - U^{MT} + W^{Err_1} + W^{Err_2} + O\left(\dfrac{MN^2\mathcal{L}^3}{\sqrt{X}} + N^2\mathcal{L}^6 D\right)\\
\ll \dfrac{MN^2}{\mathcal{L}^{A_4-2A_3}} + x^{21\varepsilon }\left(D^{\frac{33}{20}}M^{\frac{3}{20}}
+ D^{\frac{15}{8}}\right)\\ + \dfrac{N^2\ln X}{M} + \dfrac{MN^2\mathcal{L}^3}{\sqrt{X}} + N^2\mathcal{L}^6 D.
\end{multline*}
In the sum \eqref{Qerr2}, we have $x \geqslant MN > x \mathcal{L}^{-A_1}$ and $N \leqslant x^\varepsilon$, and therefore for sufficiently large $x$ it holds that
$$
M \geqslant x^{1 - 2\varepsilon}.
$$
Taking into account that $D \leqslant x^{\frac{1}{2} + \varepsilon}$, $X = N^\varepsilon$, $\exp(\mathcal{L}^\varepsilon) \leqslant N$, and $0 < \varepsilon < \frac{1}{10000}$, we obtain
$$
W - 2 \Real U + V \ll \dfrac{M N^2}{\mathcal{L}^{A_4 - 2 A_3}}.
$$
For the quantity $\tilde{U}$ defined in \eqref{Utilde}, from \eqref{a(m)} and \eqref{Uest} we get
$$
\tilde{U}(D, N, M, t) \ll \dfrac{M N}{\mathcal{L}^{\frac{A_4}{2} - A_3}}.
$$
From this, using Lemma \ref{L8}, from \eqref{Qerr2} we find
\begin{multline*}Q^{Err_2}\ll_{\varepsilon, A, A_1, A_2} \sum_{k<\mathcal{L}^{1-\varepsilon}}\dfrac{1}{k}{\sum\limits_{\substack{D, N, M \\ x\LL^{-A_1}<NM\leqslant x}}}''{\mathcal{L}^{A_2}\ln \mathcal{L}\cdot\dfrac{x}{\mathcal{L}^{\frac{A_4}{2}-{A_3}}}} + \dfrac{x\ln \mathcal{L}}{\mathcal{L}^{1-\frac{\varepsilon}{4}}}\\
\ll \dfrac{x}{\mathcal{L}^{F}} + \dfrac{x}{\mathcal{L}^{1-\varepsilon}},
\end{multline*}
where $F = \frac{A_4}{2} - A_3 - A_2 - 3.$  
Choose $A_1 = A + 3$, $A_2 = A + 4$, $A_3 = A + 9$, and $A_4 = 6A + 32$. In the formula \eqref{CA}, take $B = 5$, then $A = 100$ and
$$
Q^{Err_2} \ll_{\varepsilon} \dfrac{x}{\mathcal{L}^{100}} + \dfrac{x}{\mathcal{L}^{1-\varepsilon}} \ll \dfrac{x}{\mathcal{L}^{1-\varepsilon}}.
$$
Hence, for the quantity $Q^{Err}$ defined in \eqref{Error}, we also obtain
\begin{equation}\label{Finalerr}
Q^{Err}\ll_\varepsilon \dfrac{x}{\mathcal{L}^{1-\varepsilon}}.
\end{equation}

\section{Evaluation of $Q^{MT}$}
Recall that $$y = \sqrt{x}\mathcal{L}^A = \sqrt{x}\mathcal{L}^{100}.$$

Let us estimate the quantity $Q^{MT_{3,2}}$ defined in \eqref{Q^{MT_{3, 2}}}. Changing the order of summation, we obtain
$$Q^{MT_{3,2}} = {\sum\limits_{\substack{k <\frac{x}{2} \\ (k, 2) = 1}}}'\dfrac{1}{h(k)}\sum\limits_{\frac{2k}{y}\leqslant d <\frac{\sqrt{x}}{\mathcal{L}^{100}}}\dfrac{\chi_4(d)}{\varphi(d)}.$$
Applying Lemma 10 from \cite[Chapter V, \S 3]{Hooley}, we find
\begin{equation*}\label{FinalQ3}
Q^{MT_{3,2}} \ll \sum_{k\leqslant \frac{x}{2}}\left( \dfrac{\tau(k)\mathcal{L}y}{k} + \dfrac{\tau(k)\mathcal{L}^{101}}{\sqrt{x}}\right)\ll \sqrt{x}\mathcal{L}^{103}.
\end{equation*}
Adding $Q^{MT_1}$ and $Q^{MT_2}$, defined in \eqref{Q^{MT_1}} and \eqref{Q^{MT_2}} respectively, changing the order of summation and again applying Lemma 10 from \cite{Hooley}, we obtain
\begin{multline}Q^{MT_1} + Q^{MT_2}
= \sum_{d\leqslant\sqrt{x}\mathcal{L}^{100}}\dfrac{\chi_4(d)}{\varphi(d)}{\sum\limits_{\substack{n \leqslant x \\ (n, d) = 1}}}'\dfrac{1}{h(n)}
={\sum_{n\leqslant x}}'\dfrac{1}{h(n)}\sum\limits_{\substack{d \leqslant  \sqrt{x}\mathcal{L}^{100} \\ (d, n) = 1}}\dfrac{\chi_4(d)}{\varphi(d)}\\
= {\sum_{n\leqslant x}}'\dfrac{1}{h(n)}\left( \dfrac{\pi c}{4} E(n) + O\left(\dfrac{\tau(n)\mathcal{L}}{\sqrt{x}\mathcal{L}^{100}}\right) \right)\\
= \dfrac{\pi c}{4} {\sum_{n\leqslant x}}'\dfrac{E(n)}{h(n)} + O(\sqrt{x}),
\end{multline}
where
\begin{equation}\label{const_c}
c = \prod_{p}\left(1+\dfrac{\chi_4(p)}{p(p-1)}\right)\!,
\end{equation}
$$E(n) = \prod\limits_{\substack{p | n \\ p\equiv 1\!\!\!\!\pmod{4}}}\dfrac{(p-1)^2}{p^2-p+1} \prod\limits_{\substack{p | n \\ p\equiv 3\!\!\!\!\pmod{4}}}\dfrac{p^2-1}{p^2-p-1}.
$$
Similarly, for the quantity $Q^{MT_{3,1}}$ defined in \eqref{Q^{MT_{3, 1}}}, we obtain
$$Q^{MT_{3,1}} = \dfrac{\pi c}{4}\left({\sum\limits_{\substack{k\leqslant \frac{x}{2} \\ (k,2) = 1}}}'\dfrac{E(k)}{h(k)} - {\sum\limits_{\substack{k\leqslant \frac{x}{4} }}}'\dfrac{E(k)}{h(k)} \right) + O(\sqrt{x}\mathcal{L}^{102}).$$
Thus, from \eqref{Main} we find
$$Q^{MT} = \dfrac{\pi c}{4}\left( H(x) + H_1\left(\dfrac{x}{2}\right) - H\left(\dfrac{x}{4}\right) \right) + O\left(\sqrt{x}\mathcal{L}^{103}\right)\!,$$
where
$$H(x) = {\sum_{n\leqslant x}}'\dfrac{E(n)}{h(n)},\ \ H_1(x) = {\sum\limits_{\substack{n\leqslant x \\ (n,2) = 1}}}'\dfrac{E(n)}{h(n)}.$$
Since for each integer $\ell\geqslant 1$ we have $E(2\ell) = E(\ell)$ and $h(2\ell) = h(\ell)$, it follows that
$$H_1(x) = H(x)-H\left(\dfrac{x}{2}\right)\!.$$
Hence, 
\begin{equation}\label{QMT}
Q^{MT} = \dfrac{\pi c}{4}\left( H(x) + H\left(\dfrac{x}{2}\right) - 2H\left(\dfrac{x}{4}\right) \right) + O\left(\sqrt{x}\mathcal{L}^{103}\right)\!.
\end{equation}
Now let us find the asymptotic formula for $H(x)$. Set
$$
\mathcal{H}(s) = {\sum_{n \geqslant 1}}' \dfrac{E(n)}{h(n)} n^{-s}, \quad s = \sigma + it, \quad \sigma > 1.
$$
Using \eqref{h(p^nu)}, as well as the fact that for any $\nu \geqslant 0$ we have
$$
E(2^\nu) = h(2^\nu) = 1,
$$
we find
\begin{multline}\mathcal{H}(s) = \left(1-2^{-s}\right)^{-1}\prod\limits_{p \equiv 1\!\!\!\!\pmod{4}}\left(1 + \dfrac{(p-1)^2}{p^2-p+1}\left(\dfrac{1}{2p^s}+\dfrac{1}{3p^{2s}}+\cdots \right) \right)\\
\times \prod_{p\equiv 3\!\!\!\!\pmod{4}}\left(1+\dfrac{p^2-1}{p^2-p-1}\cdot\dfrac{1}{p^{2s}-1}\right)\!.
\end{multline}
From the Euler product expansions of the Riemann zeta function $\zeta(s)$ and the Dirichlet $L$-function $L(s, \chi_4)$, it follows that
$$\prod_{p \equiv 1\!\!\!\!\!\pmod{4}}\left(1 - \dfrac{1}{p^{s}}\right)^{-\frac{1}{2}} = \sqrt[4]{\left(1-2^{-s}\right)\zeta(s)L(s, \chi_4)}\prod_{p \equiv 3\!\!\!\!\pmod{4}}\left(1-\dfrac{1}{p^{2s}}\right)^{\frac{1}{4}}\!,$$
where $\Real s > 1$ and the principal branch of the root is chosen.  
From this, we obtain
$$\mathcal{H}(s) = \sqrt[4]{\zeta(s)}\mathcal{G}(s),$$
where
$$\mathcal{G}(s) = \left(1-2^{-s}\right)^{-\frac{3}{4}}\sqrt[4]{L(s, \chi_4)}\mathcal{P}_1(s)\mathcal{P}_3(s),$$
\begin{equation}\label{P1}
\mathcal{P}_1(s) = \prod_{p\equiv 1\!\!\!\!\pmod{4}} \left(1-p^{-s}\right)^{\frac{1}{2}}\left(1 + \dfrac{(p-1)^2}{p^2-p+1}\left(\dfrac{1}{2p^s}+\dfrac{1}{3p^{2s}}+\cdots \right) \right)\!,
\end{equation}
\begin{equation}\label{P3}
\mathcal{P}_3(s) = \prod_{p\equiv 3\!\!\!\!\pmod{4}} \left(1-p^{-2s}\right)^{\frac{1}{4}}\left(1 + \dfrac{p^2-1}{p^2-p-1}\cdot\dfrac{1}{p^{2s}-1} \right)\!.
\end{equation}
Set
$$\tilde{\mathcal{H}}(s) = \dfrac{\sqrt[4]{(s-1)\zeta(s)}\mathcal{G}(s)}{s}.$$
Note that for $n \leqslant x$ we have
$$E(n)\leqslant \prod_{p|n}\left(1+ \dfrac{p}{p^2-p-1}\right)\ll  \dfrac{n}{\varphi(n)}\ll \ln \mathcal{L}.$$
Let us fix some $T_0 \leqslant T \leqslant \sqrt{x}$ and take $b = 1 + \frac{1}{\mathcal{L}}$. Then, according to Perron's formula,
$$H(x) = j + O\left(\dfrac{x\mathcal{L}\ln\mathcal{L}}{T}\right)\!,$$
where
$$j = \dfrac{1}{2\pi i}\int_{b-iT}^{b+iT}\mathcal{H}(s)\dfrac{x^s}{s}ds =\dfrac{1}{2\pi i}\int_{b-iT}^{b+iT}\dfrac{\tilde{\mathcal{H}}(s)x^s }{\sqrt[4]{s-1}}ds.$$
For some $c_0 > 0$ (see \cite[Ch. I; \S 6, Theorem 4 and \S 7, Theorem 2]{VoKar}), the function $\zeta(s) L(s, \chi_4)$ has no zeros in the region
\begin{equation}\label{Domain}
\sigma \geqslant 1 - \dfrac{c_0}{\ln T},\ \ |t|\leqslant T.
\end{equation}
In the same region, we have (see \cite[Ch. IV, \S 2, Theorem 2]{VoKar})
\begin{equation}\label{Zetabound}
|\zeta(s)|\ll (\ln T)^{\frac{2}{3}}
\end{equation}
and 
\begin{equation}\label{Lbound}
L(s, \chi_4) \ll \ln T.
\end{equation} 
Indeed, since for $n \leqslant T$ we have $|n^{-s}| \ll n^{-1}$, it follows that
\begin{equation*}
L(s, \chi_4) = \sum_{n\leqslant T}\dfrac{\chi_4(n)}{n^s} + s\int_{T}^{+\infty}\sum_{T<n\leqslant u}\chi_4(n)\dfrac{du}{u^{s+1}}
\ll \ln T + {T^{1 - \sigma}}\ll \ln T.
\end{equation*}
Note also that in the specified region the products $\mathcal{P}_1(s)$ and $\mathcal{P}_3(s)$ converge uniformly, and therefore define analytic functions, moreover
\begin{equation}\label{P1P3}
\mathcal{P}_1(s)\mathcal{P}_3(s)\ll 1.
\end{equation} 
Indeed, since $\sigma \geqslant 2/3$, for $\mathcal{P}_3(s)$ we have
$$
\mathcal{P}_3(s) \ll \prod_p \left(1 + O\left(\dfrac{1}{p^{2\sigma}}\right)\right) \ll 1.
$$
For $\mathcal{P}_1(s)$, by virtue of the equality $E(p) = 1 + O(1/p)$, we have
\begin{multline*}\mathcal{P}_1(s) = \prod_{p \equiv 1\!\!\!\!\!\pmod{4}}\left(1 - \dfrac{1}{2p^s} + O\left(\dfrac{1}{p^{2\sigma}}\right)  \right)\\
\times \left(1 + \left(1 + O\left( \dfrac{1}{p}\right) \right)\left(\dfrac{1}{2p^s} + O\left(\dfrac{1}{p^{2\sigma}}\right)  \right)   \right)\\
= \prod_{p \equiv 1\!\!\!\!\!\pmod{4}}\left(1 + O\left( \dfrac{1}{p^{\sigma + 1}} + \dfrac{1}{p^{2\sigma}}\right) \right)\ll 1.
\end{multline*}
Now set $a = 1 - \frac{c_0}{\ln T}$. Consider the rectangular contour $\Gamma$ with vertices at the points $a \pm i T$, $b \pm i T$, inside which a horizontal cut is made from $a$ to $1$. Since the functions $\zeta(s)$ and $L(s, \chi_4)$ have no zeros inside the contour $\Gamma$, the function $\mathcal{H}(s)$ is analytic inside $\Gamma$. Hence, by Cauchy's residue theorem, we have
\begin{multline}\dfrac{1}{2\pi i}\int_{\Gamma}\dfrac{\tilde{\mathcal{H}}(s)x^s }{\sqrt[4]{s-1}}ds = \dfrac{1}{2\pi i}\left( \int_{b-iT}^{b + iT} + \int_{b+iT}^{a+iT} + \int_{a+iT}^{a+i0}\right.\\
\left.+ \int_{a + i0}^{1+i0}+\int_{1-i0}^{a-i0} + \int_{a-i0}^{a-iT} + \int_{a-iT}^{b-iT}\right) \dfrac{\tilde{\mathcal{H}}(s)x^s }{\sqrt[4]{s-1}}ds \\
= j + j_1 + j_2 + j_3 + j_4 + j_5 + j_6 = 0,
\end{multline}
where the meanings of the notations $j_1, j_2, \ldots, j_6$ are clear. From this we obtain
$$
j = -j_1 - j_2 - j_3 - j_4 - j_5 - j_6.
$$
Let us find the asymptotic formula for $J = -j_3 - j_4$. We have
\begin{multline*}j_3 = \dfrac{1}{2\pi i}\int_{a}^1\dfrac{\tilde{\mathcal{H}}(\sigma) x^\sigma d\sigma}{\sqrt[4]{\sigma - 1 +i0}}
=\dfrac{1}{2\pi i}\int_0^{1-a}\dfrac{\tilde{\mathcal{H}}(1-u)x^{1-u}}{\sqrt[4]{-1+i0}}\cdot\dfrac{du}{\sqrt[4]{u}}\\
=\dfrac{x e^{-\frac{\pi i}{4}}}{2\pi i}\int_0^{1-a}\dfrac{\tilde{\mathcal{H}}(1-u)x^{-u}}{\sqrt[4]{u}}du.
\end{multline*}
Similarly, we find
$$j_4 = -\dfrac{x e^{\frac{\pi i}{4}}}{2\pi i}\int_0^{1-a}\dfrac{\tilde{\mathcal{H}}(1-u)x^{-u}}{\sqrt[4]{u}}du,$$
whence
$$J = \dfrac{x \sin{\frac{\pi}{4}}}{\pi}\int_0^{1-a}\dfrac{\tilde{\mathcal{H}}(1-u)x^{-u}}{\sqrt[4]{u}}du.$$
Using Taylor's formula, we obtain
$$
\tilde{\mathcal{H}}(1 - u) = \tilde{\mathcal{H}}(1) + O(u),
$$
whence
\begin{multline*}J = \dfrac{x \sin{\frac{\pi}{4}}}{\pi}\cdot \tilde{\mathcal{H}}(1)\int_0^{+\infty}u^{-\frac{1}{4}}{x^{-u}}du\\
+ O\left(x\int_{1-a}^{+\infty}u^{-\frac{1}{4}}x^{-u} du + x\int_0^{+\infty}u^{\frac{3}{4}}x^{-u}du\right)\\
= \dfrac{x \sin{\frac{\pi}{4}}}{\pi(\ln x)^{\frac{3}{4}}}\Gamma\left(\dfrac{3}{4}\right) \tilde{\mathcal{H}}(1) + O\left(x^a + \dfrac{x}{(\ln x)^{\frac{7}{4}}}\right)\!.
\end{multline*}
Using the definition of $\tilde{\mathcal{H}}(s)$ and the reflection formula for the gamma function, we obtain
$$J = \dfrac{\mathcal{G}(1)}{\Gamma(\frac{1}{4})}\cdot\dfrac{x}{(\ln x)^{\frac{3}{4}}} + O\left(x^a + \dfrac{x}{(\ln x)^{\frac{7}{4}}}\right)\!.$$
Let us estimate the remaining integrals $j_\nu$ for $\nu \neq 3, 4$. Using the estimates \eqref{Zetabound}, \eqref{Lbound}, and \eqref{P1P3}, we obtain
$$j_1 + j_6 \ll \dfrac{x \mathcal{L}^{\frac{5}{12}}}{T},$$
$$j_2 + j_5 \ll (\ln T)^{\frac{5}{12}}x^a\int_{-T}^T\dfrac{dt}{\sqrt{t^2 + a^2}}\ll x^a (\ln T)^\frac{17}{12}.$$
Let us take $T = \exp(\sqrt{\mathcal{L}})$, then
$$H(x) = \dfrac{\mathcal{G}(1)}{\Gamma(\frac{1}{4})}\cdot\dfrac{x}{(\ln x)^{\frac{3}{4}}} + O\left(\dfrac{x}{(\ln x)^{\frac{7}{4}}}\right)\!.$$
From this, using \eqref{QMT}, we find
$$Q^{MT} = \dfrac{c_1 x}{(\ln x)^{\frac{3}{4}}} + O\left(\dfrac{ x}{(\ln x)^{\frac{7}{4}}}\right)\!,$$
where $$c_1 = \dfrac{c\pi \mathcal{G}(1)}{4\Gamma(\frac{1}{4})}.$$
From this, using \eqref{const_c}, \eqref{P1}, and \eqref{P3}, as well as the equality
$$
\prod_p \left(1 - \dfrac{1}{p^2}\right) = \dfrac{6}{\pi^2},
$$
we obtain
\begin{multline*}c_1 = \dfrac{1}{\Gamma(\frac{1}{4})} \left(\dfrac{\pi^5}{128}\right)^{\frac{1}{4}}\\
\times\prod_{p\equiv 1\!\!\!\!\!\pmod{4}}\sqrt{1-\dfrac{1}{p}}\left(\dfrac{1}{p-1} + (p-1)\ln\dfrac{p}{p-1}\right)\prod_{p\equiv 3\!\!\!\!\!\pmod{4}}\left(1-\dfrac{1}{p^2}\right)^{\frac{1}{4}}\\ = \dfrac{\pi^{\frac{3}{4}}}{2\Gamma(\frac{1}{4})}\prod_{p\equiv 1\!\!\!\!\!\pmod{4}}\sqrt[4]{\dfrac{p-1}{p+1}}\left(\dfrac{1}{p-1} + (p-1)\ln\dfrac{p}{p-1}\right)\!.
\end{multline*}
Finally, from \eqref{QFinal} and \eqref{Finalerr}, we find
$$Q(x) = \dfrac{c_1 x}{(\ln x)^{\frac{3}{4}}} + O_{\varepsilon}\left(\dfrac{x}{(\ln x)^{1-\varepsilon}}\right)\!.$$
The theorem is proved.

\section*{Acknowledgments}

This work was supported by the Russian Science Foundation under grant no 24-71-10005, https://rscf.ru/project/24-71-10005/

The author thanks the anonymous reviewer for valuable corrections and constructive feedback that improved the manuscript.

\bigskip 

\textbf{V. V. Iudelevich}\\
Steklov Mathematical Institute,
Gubkina str., 8, Moscow, Russia, 119991\\
\textit{E-mail:} \texttt{vitaliiyudelevich@mail.ru}

\end{document}